\documentclass[final,leqno,onefignum,onetabnum]{article}

\usepackage[square,numbers]{natbib}

\usepackage{transparent,url,examplep}
\usepackage[utf8]{inputenc} 
\usepackage[T1]{fontenc}    
\usepackage{lmodern}
\usepackage{mathrsfs}
\usepackage[sans]{dsfont}   
\usepackage{libertine}
\usepackage[hidelinks,colorlinks,linkcolor=blue,citecolor=blue]{hyperref}       
\usepackage{url}            
\usepackage{booktabs}       
\usepackage{nicefrac}       
\usepackage{microtype}      

\usepackage{amsmath,amsthm,amssymb,amsfonts}   
\usepackage{latexsym}
\usepackage{graphicx}
\usepackage{caption}
\usepackage{subcaption}
\usepackage{epstopdf}
\usepackage{cleveref}
\usepackage{autonum}
\usepackage{color}
\usepackage{algorithm,algorithmic}
\usepackage{tcolorbox}
\usepackage{geometry}

\newtheorem{theorem}{Theorem}[section]
\newtheorem{lemma}[theorem]{Lemma}

\newtheorem{proposition}{Proposition}
\theoremstyle{remark}
\newtheorem{assumption}{Assumption}
\newtheorem{remark}{Remark}
\newtheorem{definition}[theorem]{Definition}
\newtheorem*{example}{Example}

\makeatletter
\DeclareRobustCommand{\varamalg}{%
  \mathbin{\mathpalette\var@malg\perp}%
}

\newcommand\var@malg[2]{%
  \rlap{$\m@th#1#2$}\mkern6mu{#1#2}%
}
\makeatother

\makeatletter
\def\@tvsp{\mathchoice{{}\mkern-4.5mu}{{}\mkern-4.5mu}{{}\mkern-2.5mu}{}}
\def\ltrivert{\left|\@tvsp\left|\@tvsp\left|}
\def\rtrivert{\right|\@tvsp\right|\@tvsp\right|}
\makeatother

\makeatletter
\def\@tvsp{\mathchoice{{}\mkern-4.5mu}{{}\mkern-4.5mu}{{}\mkern-2.5mu}{}}
\def\llangle{\langle\@tvsp\langle}
\def\rrangle{\rangle\@tvsp\rangle}
\makeatother

\newcommand{\esp}{\mathbb{E}}
\newcommand{\prob}{\mathbb{P}}
\newcommand{\probn}{\mathbf{P}}

\newcommand{\re}{\mathbb{R}}

\newcommand{\tr}{\mathsf{tr}}

\newcommand{\KL}{\mathsf{KL}}
\DeclareMathOperator*{\op}{op}

\DeclareMathOperator*{\rank}{rank}

\DeclareMathOperator*{\TP}{TP}

\DeclareMathOperator*{\IP}{IP}
\DeclareMathOperator*{\PP}{PP}
\DeclareMathOperator*{\ATP}{ATP}
\DeclareMathOperator*{\MP}{MP}

\DeclareMathOperator*{\WRE}{WRE}

\def\bfDelta{\boldsymbol{\Delta}}

\def\bfGamma{\boldsymbol{\Gamma}}
\def\bfSigma{\boldsymbol{\Sigma}}
\def\bfTheta{\boldsymbol{\Theta}}
\def\bfXi{\boldsymbol{\Xi}}
\def\bfA{\mathbf{A}}
\def\bfB{\mathbf{B}}

\def\bfI{\mathbf{I}}

\def\bfP{\mathbf{P}}
\def\bfQ{\mathbf{Q}}
\def\bfU{\mathbf{U}}
\def\bfV{\mathbf{V}}
\def\bfW{\mathbf{W}}
\def\bfX{\mathbf{X}}
\def\bfY{\mathbf{Y}}

\def\bb{\boldsymbol b}

\def\be{\boldsymbol e}
\def\boldf{\boldsymbol f}
\def\bg{\boldsymbol g}

\def\bx{\boldsymbol x}
\def\by{\boldsymbol y}
\def\bu{\boldsymbol u}
\def\bv{\boldsymbol v}
\def\bw{\boldsymbol w}
\def\bz{\boldsymbol z}
\def\bf0{\mathbf{0}}

\def\bepsilon{\boldsymbol\epsilon}

\def\bomega{\boldsymbol\omega}

\def\btheta{\boldsymbol\theta}
\def\bxi{\boldsymbol\xi}

\def\calA{\mathcal A}
\def\calB{\mathcal B}
\def\calC{\mathcal C}

\def\calE{\mathcal E}
\def\calF{\mathcal F}
\def\calG{\mathcal G}
\def\calI{\mathcal I}

\def\calJ{\mathcal J}
\def\calL{\mathcal L}

\def\calN{\mathcal N}
\def\calO{\mathcal O}
\def\calP{\mathcal P}
\def\calR{\mathcal R}
\def\calS{\mathcal S}
\def\calQ{\mathcal Q}
\def\calU{\mathcal U}
\def\calV{\mathcal V}
\def\calX{\mathcal X}

\def\sfC{\mathsf{C}}

\def\sa{\mathsf{a}}
\def\sb{\mathsf{b}}
\def\sc{\mathsf{c}}

\def\sf{\mathsf{f}}
\def\sh{\mathsf{h}}

\def\frC{\mathfrak{C}}

\def\frM{\mathfrak{M}}
\def\frS{\mathfrak{S}}

\def\frX{\mathfrak{X}}

\def\mbB{\mathbb{B}}
\def\mbC{\mathbb{C}}
\def\mbS{\mathbb{S}}
\def\mbN{\mathbb{N}}
\def\mbX{\mathbb{X}}

\def\mcC{\mathscr{C}}
\def\mcF{\mathscr{F}}

\def\mcP{\mathscr{P}}

\def\mdC{\mathds{C}}

\def\mdR{\mathds{R}}

\newcommand{\vertiii}[1]{{\left\vert\kern-0.25ex\left\vert\kern-0.25ex\left\vert #1 
    \right\vert\kern-0.25ex\right\vert\kern-0.25ex\right\vert}}
\newcommand{\verti}[1]{{\left\vert\kern-0.4ex #1 
\kern-0.4ex\right\vert}}
\newcommand{\Lpnorm}[1]{\verti{\,#1\,}_{p}}

\newcommand{\dist}{\mathsf{d}}

\DeclareMathOperator*{\argmin}{argmin}

\title{Outlier-robust sparse/low-rank least-squares regression and robust matrix completion}
\date{}

\author{Philip Thompson 
}

\newcommand{\Addresses}{{
  \bigskip
  \footnotesize

  P.~Thompson, \textsc{Krannert School of Management, Purdue University,
    West Lafayette, Indiana}\par\nopagebreak
  \texttt{thompsp@purdue.edu}
}}

\begin{document}

\maketitle
\Addresses

\begin{abstract}
We study high-dimensional least-squares regression within a subgaussian statistical learning framework with heterogeneous noise. It includes $s$-sparse and $r$-low-rank least-squares regression when a fraction $\epsilon$ of the labels are adversarially contaminated. We also present a novel theory of \emph{trace-regression with matrix decomposition} based on a new application of the product process. For these problems, we show novel near-optimal ``subgaussian'' estimation rates of the form $r(n,d_{\mbox{\tiny{eff}}})+\sqrt{\log(1/\delta)/n}+\epsilon\log(1/\epsilon)$, valid with probability at least $1-\delta$. Here, $r(n,d_{\mbox{\tiny{eff}}})$ is the optimal uncontaminated rate as a function of the effective dimension $d_{\mbox{\tiny{eff}}}$ but independent of the failure probability $\delta$. These rates  are valid \emph{uniformly} on $\delta$, i.e., the estimators' tuning do not depend on $\delta$. Lastly, we consider noisy robust matrix completion with non-uniform sampling. If only the low-rank matrix is of interest, we present a novel near-optimal rate that is independent of the \emph{corruption level} $a$. Our estimators are tractable and based on a new ``sorted'' Huber-type loss. No information on  $(s,r,\epsilon,a)$  are needed to tune these estimators. Our analysis makes use of novel $\delta$-optimal concentration inequalities for the multiplier and product processes which could be useful elsewhere. For instance, they imply novel sharp oracle inequalities for Lasso and Slope with optimal dependence on $\delta$. Numerical simulations confirm our theoretical predictions. In particular, ``sorted'' Huber regression can outperform classical Huber regression.
\end{abstract}

\section{Introduction}\label{s:intro}

Outlier-robust estimation has been a topic studied for many decades since the seminal work by \cite{1964huber}. 
One of the objectives of the field is to device estimators which are less sensitive to outlier sample contamination. 
The formalization of outlyingness and the construction of robust estimators matured in several directions. One common assumption is that the adversary can only change a fraction $\epsilon$ of the original sample.
For an extensive overview we refer, e.g., to 
\cite{2005hampel:ronchetti:rousseeuw:stahel},  \cite{2006maronna:martin:yohai}, \cite{2011huber:ronchetti} and references therein. 

Within a very general framework, the minimax optimality of several robust estimation problems have been recently obtained in a series of elegant works by \cite{2016chen:gao:ren,2018chen:gao:ren} and \cite{2020gao}. The construction, however, is based on Tukey's depth, a hard computational problem in higher dimensions. A recent trend of research, initiated by   \cite{2016diakonikolas:kane:karmalkar:price:stewart} and   \cite{2016lai:rao:vempala}, has focused in obtaining optimality of robust estimators within the class of computationally tractable algorithms. The \emph{oblivious} model assumes the contamination is independent of the original  sample. The above mentioned works also establish optimality 
for the \emph{adversary} model where the outliers may depend arbitrarily on the sample. As an example, near optimal mean estimators for the adversarial model can now be computed in nearly-linear time \cite{2019cheng:diakonikolas:ge, 2019dong:hopkins:li,  2019depersin:lecue, 2020dalalyan:minasyan}.  We refer to the recent survey \cite{2019diakonikolas:kane} for an extensive survey.  

In the realm of robust linear regression, two broad lines of investigations exist: (1) one in which only the response (label) is contaminated and (2) the more general setting in which the covariate (feature) is also corrupted  \cite{2019diakonikolas:kamath:kane:li:steinhardt:stewart};   \cite{2019diakonikolas:kong:stewart}. Model (1), albeit less general, has been considered in many applications and studied in numerous past and recent works \cite{2008candes:randall, 2013li, 2019suggala:bhatia:ravikumar:jain, 2020gao:lafferty, 2020pesmse:flammarion}. It has also some connection with the problems of robust matrix completion \cite{2011candes:li:ma:wright, 2011chen:xu:caramanis:sanghavi, 2013chen:jalali:sanghavi:caramanis, 2013li,  2017klopp:lounici:tsybakov} and matrix decomposition \cite{2011chandrasekaran:sanghavi:parrilo:willsky, 2011candes:li:ma:wright, 2011xu:caramanis:sanghavi, 2011hsu:kakade:zhang,  2012agarwal:negahban:wainwright}.          
Both models (1) and (2) have been considered assuming adversarial or oblivious contamination. For instance, an interesting property of model (1) with oblivious contamination is the existence of consistent estimators, a property not shared by the adversary model. See for instance the recent papers \cite{2014tsakonas:jalden:sidiropoulos:ottersten, 2017consistent:robust:regression, 2019suggala:bhatia:ravikumar:jain, 2020gao:lafferty, 2020pesmse:flammarion}.  We refer to Section \ref{ss:contributions:related:work} for further references and discussion. 

In this work, we revisit the problem of outlier-robust least-squares regression with adversarial label contamination. We are particularly interested on the high-dimensional scaling. In this setting, not only the sample is corrupted by $o$ outliers but its sample size $n$ is prohibitively smaller than the extrinsic dimension $p$. We will focus in the setting where the label-feature distribution is subgaussian. We pay particularly attention to the following points: 
\begin{itemize}
\item[\rm{(a)}] \emph{High dimensions.} We consider the general framework of \emph{trace-regression} with parameters in $\mdR^p:=\re^{d_1\times d_2}$  assuming $n\ll p:=d_1d_2$. It includes  in particular $s$-sparse linear regression \cite{1996tibshirani}, noisy compressed sensing with a low-rank parameter \cite{2011negahban:wainwright} and noisy robust matrix completion \cite{2017klopp:lounici:tsybakov}. One practical appeal of the established theory of high-dimensional estimation is the existence of efficient estimators  \emph{adaptive} to $(s,r)$. We wish to avoid knowledge of 
$(s,r,o)$, at least in the label contamination model.
\item[\rm{(b)}] \emph{Noise heterogeneity.} A large portion of the literature in outlier-robust sparse linear regression assumes the noise is independent of the feature. This model is relevant in its own right. However, we wish to assume no particular assumptions between noise and features and consider the statistical learning framework with the linear hypothesis class. 
\item[\rm{(c)}] \emph{Subgaussian rates and uniform confidence level.} Significant effort has been paid recently in obtaining minimax rates also with respect to the failure probability $\delta$. Just to illustrate what this means, consider the challenging problem of estimating the mean of a heavy-tailed $p$-dimensional random vector with identity covariance. Minsker's original bound \cite{2015minsker} in this case is $\sqrt{p\log(1/\delta)/n}$. The ``subgaussian'' rate $\sqrt{p/n}+\sqrt{\log(1/\delta)/n}$ was obtained recently after a series of works and include efficient estimators. For lack of space, we refer to the recent survey \cite{2019lugosi:mendelson-survey}.
The relevant point of the subgaussian rate is that 
$\log(1/\delta)$ does not multiply the ``effective dimension'' $p$. A second point is to what extent the estimator \emph{tuning} depends on $\delta$. Is it possible to attain the optimal \emph{subgaussian} rate \emph{uniformly} on $\delta$, that is, with a 
$\delta$-independent tuned estimator? In this work, we wish to obtain $\delta$-uniform optimal subgaussian rates for the particular model of least-squares regression with adversarial label contamination, subgaussian data and validity of (a)-(b). 
\item[\rm{(d)}] \emph{Matrix decomposition.} With identity designs, this challenging problem was first considered by \cite{2009wright:ganesh:rao:peng:ma},  \cite{2011chandrasekaran:sanghavi:parrilo:willsky}, \cite{2011candes:li:ma:wright}, \cite{2011xu:caramanis:sanghavi},  \cite{2011hsu:kakade:zhang} and \cite{2011mccoy:tropp}. A common assumption in most of these works is the ``incoherence'' condition. Alternatively, \cite{2012agarwal:negahban:wainwright}  studied this problem within a general framework for identity or random designs assuming the ``low-spikeness'' condition. For instance, their framework includes multivariate regression with random designs and low-rank plus sparse components. For this problem, one fact used is that the design operator is positive definite (one has $n\ge d_1$ albeit $n\ll d_1d_2$). Unfortunately, the same property does not hold for the trace-regression problem. Motivated by the problem of noisy compressed sensing with a matrix parameter \cite{2011negahban:wainwright}, we wonder if a correspondent theory exists for the problem of trace-regression with low-rank plus sparse components. 
\end{itemize}

The rest of the paper is organized as follows. We state our framework in Section \ref{s:framework}. Contributions and related work are discussed in Section \ref{ss:contributions:related:work}. The main results are stated formally in Section \ref{s:main:results}. Numerical simulations confirming the theoretical predictions are presented in Section \ref{s:simulation}. We finalize with a discussion in Section \ref{s:discussion}. The proofs are presented in the Supplemental Material.

\section{Framework}\label{s:framework}

Throughout the paper, we use standard notations for norms in $\mdR^p=\re^{d_1\times d_2}$. The $\ell_k$-norm ($1\le k\le\infty$) is denoted by $\Vert\cdot\Vert_k$, the Frobenius norm by $\Vert\cdot\Vert_F$, the nuclear norm by $\Vert\cdot\Vert_N$ and the operator norm by $\Vert\cdot\Vert_{\op}$. The inner product in $\re^k$ will be denoted by $\langle\bv,\bw\rangle=\bv^\top\bw$ while the inner product in $\mdR^p$ will be denoted by 
$\llangle\bfV,\bfW\rrangle=\tr(\bfV^\top\bfW)$.

\subsection{Sparse and trace regression}\label{ss:response:trace:regression:intro}

Let $(\bfX,y)\in\mdR^p\times\re$ be a zero mean feature-label pair. Within a statistical learning framework, we wish to explain $y$ trough $\bfX$ via the linear class 
$
\mcF=\{\llangle\cdot,\bfB\rrangle:\bfB\in\mdR^p\}.
$ 
Precisely, giving a sample of $(\bfX,y)$, we wish to estimate 
\begin{align}
\bfB^*\in\argmin_{\bfB\in\mdR^p}\esp\left[y-\llangle\bfX,\bfB\rrangle\right]^2.\label{equation:least-squares:regression}
\end{align}
In particular, one has
$
y=\llangle\bfX,\bfB^*\rrangle+\xi
$
with $\xi\in\re$ satisfying 
$\esp[\xi\bfX]=0$. We consider the following assumption on the available sample. 
\begin{tcolorbox}
\begin{assumption}[Adversarial label contamination]\label{assump:label:contamination}
Let $\{(y_i^\circ,\bfX_i^\circ)\}_{i\in[n]}$ be an iid sample from the distribution of $(\bfX,y)$. We assume available a sample
$\{(y_i,\bfX_i)\}_{i\in[n]}$ such that $\bfX_i=\bfX_i^\circ$ for all $i\in[n]$ and $o$ arbitrary outliers replace the label sample $\{y_i^\circ\}_{i\in[n]}$. We denote the fraction of contamination by 
$\epsilon:=\frac{o}{n}$.
\end{assumption}
\end{tcolorbox}

It will be useful to define the \emph{design operator} with components 
$\frX_i(\bfB):=\llangle\bfX_i,\bfB\rrangle$. Under Assumption \ref{assump:label:contamination}, one may write 
\begin{align}
\by=\frX(\bfB^*)+\sqrt{n}\btheta^*+\bxi,\label{equation:structural:equation}
\end{align}
where $\by=(y_i)_{i\in[n]}$, $\bxi=(\xi_i)_{i\in[n]}$ is an iid copy of $\xi$ and $\btheta^*\in\re^{n}$ is an arbitrary and unknown corruption vector  having at most $o$ nonzero components.

\begin{example}[Sparse linear regression]\label{example:sparse:linear:regression}
In sparse \emph{linear regression} \cite{2009geer:buhlmann, BRT}, we have $\bx:=\bfX\in\re^p$ and $\bb^*:=\bfB^*\in\re^p$  with at most $s\ll p$ nonzero coordinates ($d_1=p$, $d_2=1$). One may write the design as $\frX(\bb):=\mbX\bb$ where $\mbX$ is the \emph{design matrix} whose $i$th row 
is $\bx_i^\top$. 
\end{example}

\begin{example}[Low-rank trace-regression]\label{example:trace:regression}
In low-rank \emph{trace-regression} \cite{2011rohde:tsybakov, 2011negahban:wainwright}, the parameter 
$\bfB^*\in\mdR^{p}$ is assumed to have rank $r\ll d_1\wedge d_2$. 
\end{example}

Given nonincreasing positive sequence $\{\omega_i\}_{i\in[n]}$, the Slope norm \cite{2015bogdan:berg:sabatti:su:candes} at a point $\bu\in\re^n$ is defined by
$$
\Vert\bu\Vert_\sharp:=\sum_{i\in[n]}\omega_i\bu_i^\sharp,
$$
where $\bu_1^\sharp\ge\ldots\ge\bu_n^\sharp$ denotes the nonincreasing rearrangement of the absolute coordinates of $\bu$. Throughout this paper, we fix the sequence 
$\bomega:=\sqrt{(\omega_i)_{i\in[n]}}$ to be 
$
\omega_i:=\sqrt{\log(A n/i)}
$
for some $A\ge2$.

\begin{tcolorbox}
\begin{definition}[Sorted Huber's loss]\label{def:sorted:Huber:loss}
Given $\tau>0$, let $\rho_{\tau\bomega}:\re^n\rightarrow\re_+$ be the optimal value of the proximal map associated to the Slope norm $\tau\Vert\cdot\Vert_\sharp$. Precisely,
$$
\rho_{\tau\bomega}(\bu):=\min_{\bz\in\re^n}\frac{1}{2}\Vert\bz-\bu\Vert_2^2+\tau\Vert\bz\Vert_{\sharp}.
$$
Given $\tau>0$, we define the loss 
\begin{align}
\calL_{\tau\bomega}(\bfB):=\rho_{\tau\bomega}\left(\frac{\by-\frX(\bfB)}{\sqrt{n}}\right).
\end{align}
\end{definition}
\end{tcolorbox}
Albeit being convex, the loss in Definition \ref{def:sorted:Huber:loss} does not possess an explicit expression in general. One exception is when $\omega_1=\ldots=\omega_n$. In this case, if 
$\omega_1=1$,
\begin{align}
\calL_{\tau\bomega}(\bfB)=\tau^2\sum_{i=1}^n\Phi\left(\frac{y_i-\frX_i(\bfB)}{\tau\sqrt{n}}\right),
\end{align}
where $\Phi:\re\rightarrow\re$ is the Huber's function defined by $\Phi(t)=\min\{(\nicefrac{1}{2})t^2,|t|-\nicefrac{1}{2}\}$. Thus, Huber regression corresponds to $M$-estimation with the loss $\calL_{\tau\bomega}$ with \emph{constant} weighting sequence $\bomega$. 

In this work, we instead advocate the use of the loss
$\calL_{\tau\bomega}$ with weighting sequence
$\bomega=(\omega_i)_{i\in[n]}$ given by 
$\omega_i=\sqrt{\log(An/i)}$. This corresponds to $M$-estimation with a ``sorted'' modification of Huber regression. In high-dimensions, we additionally use a regularization norm. Given a convex norm 
$\calR$ over $\mdR^p$, we consider the estimator
\begin{align}\label{equation:sorted:Huber}
\hat\bfB\in 
&\argmin_{\bfB\in\mdR^p}\rho_{\tau\bomega}\left(\frac{\by-\frX(\bfB)}{\sqrt{n}}\right)+\lambda\calR(\bfB).
\end{align}

A well known fact in the literature is that linear regression with Huber's loss corresponds to a Lasso-type estimator in the augmented variable $[\bfB;\btheta]\in\mdR^p\times\re^n$ \cite{2011she:owen,2016donoho:montanari}. Similarly, estimator \eqref{equation:sorted:Huber} is equivalent to the following augmented  least-squares estimator:
\begin{align}
[\hat\bfB,\hat\btheta]\in 
&\argmin_{[\bfB;\btheta]\in\mdR^p\times\re^n}\frac{1}{2n}\sum_{i=1}^n\left(y_i-\llangle\bfX_i,\bfB\rrangle-\sqrt{n}\btheta_i\right)^2+\lambda\calR(\bfB)+\tau\Vert\btheta\Vert_\sharp.
\label{equation:aug:slope:rob:estimator}
\end{align}
If $\calR$ is either the $\ell_1$-norm,  the Slope norm in $\re^p$ or the nuclear norm in $\re^{d_1\times d_2}$, one practical appeal of problem  \eqref{equation:aug:slope:rob:estimator} is that it may be computed by alternated convex optimization using Lasso and Slope solvers  \cite{2015bogdan:berg:sabatti:su:candes}.

\subsection{Trace-regression with matrix decomposition}
\label{ss:matrix:decomposition:intro}

\cite{2012agarwal:negahban:wainwright} considered a  general framework for the \emph{matrix decomposition problem}:   to estimate a pair $[\bfB^*;\bfGamma^*]\in(\mdR^p)^2$ given a noisy linear observation of its sum. In the case of random design and noise, their model is
\begin{align}
\bfY=\frX(\bfB^*+\bfGamma^*)+\bfXi,  \label{equation:matrix:decomp}
\end{align}
where the design $\frX:\mdR^p\rightarrow\re^{n\times m}$ take values on matrices with $n$ iid rows and $\bfXi\in\re^{n\times m}$ is a noise matrix  independent of $\frX$ with $n$ centered iid rows. Among several results in \cite{2012agarwal:negahban:wainwright}, one application with random design is \emph{multi-task learning} with $d_1$ ``features'' and $d_2$ ``tasks'', assuming  normal covariates $\bx\in\re^{d_1}$. In this setting, model \eqref{equation:matrix:decomp} corresponds to having an iid sample $\{(\bx_i,\by_i)\}_{i=1}^{n}$ 
satisfying the model $\by_i=(\bfB^*+\bfGamma^*)^\top\bx_i+\bxi_i$ with $\bxi_i$ independent of $\bx_i$ and design operator 
 $\frX_i(\bfB):=\bx_i^\top\bfB$. 

Motivated by the correspondent problem in matrix ``compressed sensing'' \cite{2011negahban:wainwright}, we investigate matrix decomposition in the trace regression problem within a statistical learning framework: to estimate the pair
\begin{align}
[\bfB^*;\bfGamma^*]\in\argmin_{[\bfB;\bfGamma]\in(\mdR^p)^2}\esp\left[y-\llangle\bfX,\bfB+\bfGamma\rrangle\right]^2,\label{equation:least-sq:reg-matrix:decomp}
\end{align}
given sample of $(\bfX,y)\in\mdR^p\times\re$. In particular, one has
$
y=\llangle\bfX,\bfB^*+\bfGamma^*\rrangle+\xi
$
with $\xi\in\re$ satisfying $\esp[\xi\bfX]=0$. In high-dimensions, one assumes that $\bfB^*$ has low-rank and $\bfGamma^*$ is sparse. If $\{(y_i,\bfX_i)\}_{i\in[n]}$ is an iid sample of $(y,\bfX)$ and $\{\xi_i\}_{i\in[n]}$ an iid copy of $\xi$, one may write 
\begin{align}
\by=\frX(\bfB^*+\bfGamma^*)+\bxi,\label{equation:str:eq:trace:reg:matrix:decomp}
\end{align}
where $\frX:\mdR^p\rightarrow\re^n$ is as in Section \ref{ss:response:trace:regression:intro}, 
$\by=(y_i)_{i\in[n]}$ and $\bxi=(\xi_i)_{i\in[n]}$.

Following \cite{2012agarwal:negahban:wainwright}, we consider the assumption:
\begin{assumption}[Low-spikeness]\label{assump:low:spikeness}
Assume $\bfX$ is isotropic, there is, 
$\esp[\llangle\bfX,\bfV\rrangle^2]=\Vert\bfV\Vert_F^2$ for all $\bfV\in\mdR^p$. Moreover, assume there exists 
$\sa^*>0$ such that 
\begin{align}
\Vert\bfB^*\Vert_\infty\le\frac{\sa^*}{\sqrt{n}}. 
\end{align}
\end{assumption}
\begin{remark}
The low-spikeness condition in \cite{2012agarwal:negahban:wainwright} is $\Vert\bfB^*\Vert_\infty\le\sa^*/\sqrt{d_1d_2}$ for some 
$\sa^*>0$. In high-dimensions  ($n\le d_1d_2$) and assuming isotropy, it implies Assumption \ref{assump:low:spikeness}.
\end{remark}

We consider the constrained estimator
\begin{align}
\begin{array}{ccl}
[\hat\bfB,\hat\bfGamma]&\in
&\argmin_{[\bfB;\bfGamma]\in(\mdR^p)^2}\frac{1}{2n}\sum_{i=1}^n\left(y_i-\llangle\bfX_i,\bfB+\bfGamma\rrangle\right)^2+\lambda\Vert\bfB\Vert_N+\tau\Vert\bfGamma\Vert_{1}\\
&&\mbox{s.t.}\quad	 \Vert\bfB\Vert_\infty\le\frac{\sa^*}{\sqrt{n}}.
\label{equation:aug:estimator:trace:reg:matrix:decomp}
\end{array}
\end{align}

\subsection{Robust matrix completion}\label{ss:robust:matrix:completion:intro}

The matrix completion problem consists in estimating a low-rank matrix $\bfB^*\in\mdR^{p}$ having sampled an incomplete subset of its entries. Several works have obtained statistical bounds for this problem using nuclear norm relaxation \cite{2002fazel, 2004srebro}, either with the ``incoherence'' condition \cite{2009candes:recht}, the ``low-spikeness'' condition \cite{2012negahban:wainwright} or assuming an upper bound on the sup-norm of $\bfB^*\in\mdR^{p}$ \cite{2011koltchinskii:lounici:tsybakov, 2014klopp}. Note that matrix completion may be equivalently seen as the trace-regression problem with $\bfX$ having a discrete distribution $\Pi$ supported on the canonical basis 
\begin{equation}
\calX:=\{\be_j\bar\be_k^\top:j\in[d_1],k\in[d_2]\},\label{equation:canonical:basis}
\end{equation}
where $\be_j$ is the $j$th canonical vector in 
$\re^{d_1}$ and $\bar\be_k$ is the $k$th canonical vector in $\re^{d_2}$. The $i$th observed entry $\bfB_{j(i),k(i)}$ of $\bfB$ corresponds to $\bfX_i=\be_{j(i)}\bar\be_{k(i)}^\top$. 

In robust matrix completion, a fraction of the sampled entries are corrupted by outliers.  \cite{2017klopp:lounici:tsybakov} consider this problem in the noisy setting through the lens of the matrix decomposition model \eqref{equation:str:eq:trace:reg:matrix:decomp} where the corruption matrix $\bfGamma^*\in\mdR^p$ have at most 
$o$ nonzero entries. Optimal minimax bounds are derived on the class of parameters $[\bfB^*,\bfGamma^*]$ with bounded sup-norm, where $\bfB^*$ is low-rank and $\bfGamma^*$ is a sparse corruption matrix. 

It some applications, only the sampled low-rank matrix $\bfB^*$ is of interest. In that case, one may equivalently see the same problem as a response adversarial trace regression problem \eqref{equation:structural:equation} under Assumption \ref{assump:label:contamination}. In this work we use this point of view on the class of parameters $\bfB^*$ with low-rank and bounded sup norm. Recall the loss $\rho_{\tau\bomega}$ in Definition \ref{def:sorted:Huber:loss}. For tuning parameters $\lambda,\tau>0$ and $\sa>0$, we consider the constrained estimator
\begin{align}
\begin{array}{ccl}
\hat\bfB&\in
&\argmin_{\bfB\in\mdR^p}\rho_{\tau\bomega}\left(\frac{\by-\frX(\bfB)}{\sqrt{n}}\right)+\lambda\frac{\Vert\bfB\Vert_N}{\sqrt{p}}\\
&&\mbox{s.t.}\quad	 \Vert\bfB\Vert_\infty\le\sa.
\label{equation:slope:rob:estimator:MC}
\end{array}
\end{align}
Equivalently, $\hat\bfB$ can be computed via the augmented estimator
\begin{align}
\begin{array}{ccl}
[\hat\bfB,\hat\btheta]&\in
&\argmin_{[\bfB;\btheta]\in\mdR^p\times\re^n}\frac{1}{2n}\sum_{i=1}^n\left(y_i-\llangle\bfX_i,\bfB\rrangle-\sqrt{n}\btheta_i\right)^2+\lambda\frac{\Vert\bfB\Vert_N}{\sqrt{p}}+\tau\Vert\btheta\Vert_\sharp\\
&&\mbox{s.t.}\quad	 \Vert\bfB\Vert_\infty\le\sa.
\label{equation:aug:slope:rob:estimator:MC}
\end{array}
\end{align}

\section{Contributions and related work}\label{ss:contributions:related:work}

\subsection{Robust sparse least-squares regression}
For sparse least-squares regression with adversarial response contamination and subgaussian $(\bx,\xi)$, we show that estimator \eqref{equation:sorted:Huber} achieves the subgaussian rate
\begin{align}\label{eq:rate:linear:reg}
r_n+\sqrt{\log(1/\delta)/n}+\epsilon\log(1/\epsilon),
\end{align}
with probability at least $1-\delta$, for any given $\delta\in(0,1)$. Here, $r_n=\sqrt{s\log p/n}$ taking
$\calR$ to be the $\ell_1$-norm and $r_n=\sqrt{s\log (ep/s)/n}$
taking $\calR$ to be the Slope norm on 
$\re^p$ (see Theorem \ref{thm:response:sparse:regression}). Note that we show such property with tuning parameters $(\tau,\lambda)$ \emph{independent of the failure probability $\delta$}. The above bound is valid for a breakdown point 
$\epsilon\le c$, where $c$ is a constant.  The above rate is optimal up to a log factor in $1/\epsilon$ \cite{2016chen:gao:ren,2018chen:gao:ren, 2020gao, 2020cherapanamjeri:aras:tripuraneni:jordan:flammarion:bartlett}. These bounds are attained with no information on $(s,o)$ and within the statistical learning framework with the linear class. To the best of our knowledge, previous works on corrupted sparse linear regression have  assumed the noise independent of the features.  

Sparse linear regression with response contamination has been the subject of numerous works. From a methodological point of view, the $\ell_1$-penalized Huber's estimator has been considered in \cite{2001sardy:tseng:bruce, 2011she:owen, 2012lee:maceachern:jung}. Empirical evaluation for the choice of tuning parameters is comprehensively studied in these papers. As already observed in past work \cite{2011she:owen, 2016donoho:montanari}, Huber's estimator with $\ell_1$-penalization is equivalent to the augmented estimator 
\begin{align}
[\hat\bb,\hat\btheta]\in 
&\argmin_{[\bb;\btheta]\in\re^p\times\re^n}\frac{1}{2n}\sum_{i=1}^n\left(y_i-\langle\bx_i,\bb\rangle-\sqrt{n}\btheta_i\right)^2+\lambda\Vert\bb\Vert_1+\tau\Vert\btheta\Vert_1.\label{equation:aug:lasso}
\end{align}
In the response adversarial model with Gaussian data, fast rates for such estimator have been obtained in \cite{2008candes:randall, 2009laska:davenport:baraniuk, 2009dalalyan:keriven, 2012dalalyan:chen, 2013nguyen:tran}. The minimax optimality of estimator \eqref{equation:aug:lasso},  up to log factors, was achieved only recently in \cite{2019dalalyan:thompson}, showing it satisfies the rate 
$$
\sqrt{s\log(p/\delta)/n}+\epsilon\log(n/\delta), 
$$
with a breakdown point $\epsilon\le c/\log n$ for a constant $c>0$. In addition to $\log n$ factors, the above rate is not subgaussian in $\delta$. The $\ell_1$-penalized Huber's estimator was later shown to 
attain the subgaussian rate by \cite{2019chinot}  assuming the features are Gaussian, knowledge of 
$(s,o)$ and independence between features and noise. We highlight that the approximately linear growth in $\epsilon$ in the error estimate \eqref{eq:rate:linear:reg} using the Sorted Huber loss is confirmed in our numerical experiments. See Section \ref{s:simulation}, Figure \ref{fig.robust.linear.reg}(a). Besides the mentioned theoretical improvements in the rate, we observed in simulations that robust least-squares estimation with the Sorted Huber loss can outperform classical Huber regression (see Figure \ref{fig.robust.linear.reg}(b)).

We remark that label contamination in regression has been considered in different contamination models with dense bounded noise in \cite{2010wright:ma, 2013li, 2013nguyen:tran-dense, 2014foygel:mackey, 2018adcock:bao:jakeman:narayan}. This setting is also studied in \cite{2018karmalkar:price} with the LAD-estimator \cite{2007wang:li:jiang}. Alternatively, a refined analysis of iterative thresholding methods were considered in \cite{2015bhatia:jain:kar},  \cite{2017consistent:robust:regression}, \cite{2019suggala:bhatia:ravikumar:jain}, \cite{2019mukhoty:gopakumar:jain:kar}. They obtain sharp breakdown points and consistency bounds for the oblivious model. Works on sparse linear regression with covariate contamination were considered early on by \cite{2013chen:caramanis:mannor} and, more recently, in \cite{2017balakrishnan:du:li:singh}, albeit with worst rates and breakdown points compared to the response contamination model. Works by \cite{2012loh:wainwright, 2012loh:wainwright-AoS} have also studied the optimality of sparse linear regression in models with error-in-variables and missing-data covariates. 

Although out of scope, we mention for completeness that tractable algorithms for linear regression with covariate contamination have been  intensively investigated in the low-dimensional scaling ($n\ge p$), with initial works by \cite{2019diakonikolas:kamath:kane:li:steinhardt:stewart};  \cite{2019diakonikolas:kong:stewart,  2018prasad:suggala:balakrishnan:ravikumar} and more recent ones in \cite{2020depersin, 2020cherapanamjeri:aras:tripuraneni:jordan:flammarion:bartlett,   2020pensia:jog:loh}. 

\subsection{Robust trace-regression and matrix completion}

\emph{Robust trace-regression.} In \cite{2011negahban:wainwright, 2012negahban:ravikumar:wainwright:yu} it is proposed a general framework of $M$-estimators with decomposable regularizers. Among several different results, they obtain optimal rates for trace-regression with Gaussian designs for the first time. Precisely, they attain the minimax rate $\sqrt{rd_1/n}+\sqrt{rd_2/n}$ with failure probability $e^{-c(d_1+d_2)}$, for some constant $c>0$, and noise independent of $\bfX$. Of course, their bounds can be translated to a optimal bound in average. With respect to the failure probability, however, their bounds are not subgaussian optimal. Inspired by their results and with the objectives (a)-(c) of Section \ref{s:intro} in mind, we consider trace-regression with response adversarial contamination and subgaussian data. Within the statistical learning framework, we show that estimator \eqref{equation:sorted:Huber} with nuclear norm regularization achieves the subgaussian rate
\begin{align}\label{eq:rate:trace:reg}
\sqrt{rd_1/n}+\sqrt{rd_2/n}+\sqrt{\log(1/\delta)/n}+\epsilon\log(1/\epsilon),
\end{align}
with failure probability $\delta$ and breakdown point 
$\epsilon\le c$ for a positive constant $c$ (see Theorem \ref{thm:response:trace:regression}). The tuning of estimator \eqref{equation:sorted:Huber} is independent of $\delta$ and the above rate is optimal up to a log factor in $1/\epsilon$ \cite{2016chen:gao:ren,2018chen:gao:ren, 2020gao}. We are not aware of previous work showing that efficient estimators attain the above rate under our set of assumptions. We confirm in our simulations the approximately linear growth in $\epsilon$ in the rate \eqref{eq:rate:trace:reg}. See Figure \ref{fig.robust.trace-reg}(a) in Section \ref{s:simulation}.

\emph{Trace-regression with matrix decomposition.}
Several works studied the problem of matrix decomposition with identity designs and the ``incoherence'' condition \cite{2011candes:li:ma:wright, 2011chandrasekaran:sanghavi:parrilo:willsky, 2011hsu:kakade:zhang, 2013chen:jalali:sanghavi:caramanis}.  Alternatively, \cite{2012agarwal:negahban:wainwright} viewed such problem in a general framework assuming ``low-spikeness''. Among several applications, multi-task learning is one of them. We consider the different problem of matrix compressed sensing \cite{2011negahban:wainwright} with $r$-low-rank plus $s$-sparse components satisfying the ``low-spikeness'' condition and an isotropic design (see Section \ref{ss:matrix:decomposition:intro}).  Unlike multi-task learning, the design is not positive definite in high-dimensions, so an alternative argument is required. Under Assumption \ref{assump:low:spikeness} and subgaussian design, we show that estimator \eqref{equation:aug:estimator:trace:reg:matrix:decomp}, with tuning independent of $\delta$, attains the near-optimal subgaussian rate
\begin{align}\label{eq:rate:trace:reg:MD}
\sqrt{rd_1/n}+\sqrt{rd_2/n}+\sqrt{s\log p/n}
+\sa^*\sqrt{s/n}+\sqrt{\log(1/\delta)/n},
\end{align}
with failure probability $\delta$ (see Theorem \ref{thm:tr:reg:matrix:decomp}). This is valid within the statistical learning framework and with no information on 
$(r,s)$. We are not aware of previous work establishing any estimation theory for this problem. In our simulations, we were able to confirm the linear growth in $r$ (for fixed $s$) and $s$ (for fixed $r$) predicted by the rate \eqref{eq:rate:trace:reg:MD}. See Section \ref{s:simulation}, Figure \ref{fig.trace-reg-MD}.  

\emph{Robust matrix completion.} The literature on matrix completion using nuclear norm relaxation \cite{2002fazel, 2004srebro} is extensive. For instance, bounds for exact recovery were first obtained by  \cite{2009candes:recht} where the notion of ``incoherence'' was introduced. \cite{2012negahban:wainwright} considers noisy matrix completion with the different notion of ``low-spikeness''. Several works on exact and noisy set-up exist. As we are mainly interested in the corrupted model, a complete overview is out of scope. We refer to further references in \cite{2014klopp, 2017klopp:lounici:tsybakov} and the recent work \cite{2020chen:fan:ma:yan-nonconvex} for a comprehensive review on matrix completion under the ``incoherence'' assumption. 

Related to our work are  \cite{2011gaiffas:lecue, 2011klopp-rank, 2017klopp:gaiffas, 2011koltchinskii:lounici:tsybakov, 2014klopp, 2017klopp:lounici:tsybakov, 2016cai:zhou}. In these papers the main assumption is that an upper bound on the parameter sup-norm $\Vert\cdot\Vert_\infty$ is known. Robust matrix completion is considered in \cite{2016fan:wang:zhu, 2018elsener:geer} and 
\cite{2017klopp:lounici:tsybakov}.  \cite{2016fan:wang:zhu} and  \cite{2018elsener:geer} are mainly concerned with the heavy-tailed model.  Like this work, \cite{2017klopp:lounici:tsybakov} considers the outlier contamination model with noise. 
Assuming known $\sa>0$ such that $\sa\ge\Vert\bfB^*\Vert_\infty\vee \Vert\bfGamma^*\Vert_\infty$,  \cite{2017klopp:lounici:tsybakov} establish optimal minimax bounds for  model \eqref{equation:str:eq:trace:reg:matrix:decomp} over the class 
\begin{align}
\calA(r,o,\sa)=\{\bfB+\bfGamma\in\mdR^p:\rank(\bfB)\le r, \Vert\bfGamma\Vert_0\le o, \Vert\bfB\Vert_\infty\vee \Vert\bfGamma\Vert_\infty\le\sa\}. 
\end{align}
As usual, $\Vert\bfA\Vert_0$ denotes the number of nonzero entries of matrix $\bfA$. Assuming known $\sa>0$ such that $\sa\ge\Vert\bfB\Vert_\infty$, we instead consider noisy robust matrix completion for model \eqref{equation:structural:equation} on the class
\begin{align}
\bar\calA(r,o,\sa)=\{[\bfB;\btheta]\in\mdR^p\times\re^n:\rank(\bfB)\le r, \Vert\btheta\Vert_0\le o, \Vert\bfB\Vert_\infty\le\sa\}.\label{equation:class:RMC}
\end{align}
For simplicity, let us assume $d_1=d_2=d$. Under similar distributional assumptions of \cite{2017klopp:lounici:tsybakov}, we show that estimator \eqref{equation:aug:slope:rob:estimator:MC} 
attains, up to logs, the rate 
\begin{align}\label{eq:rate:MC}
\sa\sqrt{(dr/n)\log(1/\delta)}+\sa\sqrt{\log(1/\delta)/n}
+\sa\sqrt{\epsilon\log(1/\epsilon)}.
\end{align}
We refer to Theorem \ref{thm:robust:MC} for details. This rate is optimal on the class $\bar\calA(r,o,\sa)$ up to log factors. In Figure \ref{fig.robust.trace-reg}(b) in Section \ref{s:simulation}, we confirm the dependence on 
$\epsilon$ predicted in rate \eqref{eq:rate:MC}. Finally, we note that our bounds are not fully comparable to \cite{2017klopp:lounici:tsybakov} whose aim is to estimate the pair $[\bfB^*; \bfGamma^*]$. On the other hand, there are many applications for which only the low-rank parameter
$\bfB^*$ is of interest. In this setting, our bound \eqref{eq:rate:MC} show a significant improvement as it depends only on an upper bound on the sup norm of the \emph{low-rank parameter} $\bfB^*$. The \emph{corruption level} $\Vert\btheta^*\Vert_\infty$ is irrelevant, either in the rate or in tuning the estimator. We have verified this phenomenon in our simulations. See Section \ref{s:simulation}.  

\subsection{Comments on the proofs}

We finish with some technical remarks. Related to our work is \cite{2019dalalyan:thompson}. They obtain rates for sparse linear regression with adversarial label contamination for the Gaussian model and noise independent of features. The rates in \cite{2019dalalyan:thompson} are near-optimal and valid for $\epsilon\lesssim 1/\log n$. These rates, however, are not optimal with respect to the failure probability $\delta$. In order to establish subgaussian optimal rates (in the sense of item (c) in Section \ref{s:intro}) with $(p,n,s,\delta)$-independent breakdown point $\epsilon\lesssim 1/2$ and within the statistical learning framework, our theory relies on new concentration inequalities for the \emph{multiplier process} (MP) and \emph{product process} (PP). Impressive  bounds for these processes were obtained by  \cite{2016mendelson} which hold for general classes having only a few bounded moments. In establishing subgaussian rates, our arguments need concentration inequalities for the MP and PP with improvements concerning the confidence level. Tailored specifically to subgaussian classes, these are proven in Theorems \ref{thm:mult:process} and \ref{thm:product:process} in the supplemental material.  We believe these bounds may be useful elsewhere, in particular, Compressive Sensing theory. Their proof makes use of the ``generic chaining'' method, originally proposed by \cite{2014talagrand}. More precisely, our proof is inspired by improved generic chaining methods by  \cite{2015dirksen} and  \cite{2014bednorz} for the \emph{quadratic process}. We also highlight that our theory of robust regression makes crucial use of a high-probability  version of \emph{Chevet's inequality} for subgaussian processes. 

We are also inspired by the findings of  \cite{2018bellec:lecue:tsybakov} regarding Slope regularization in sparse linear regression. This paper goes beyond the \emph{exact} sparse setting presenting sharp \emph{oracle inequalities} for the Lasso and Slope estimators. It is shown for the first time in \cite{2018bellec:lecue:tsybakov} that Lasso indeed satisfies an oracle inequality with the \emph{subgaussian} rate (see item (c) in Section \ref{s:intro}). It is also shown that Lasso and Slope can be tuned \emph{independently} of $\delta$ assuming subgaussian data with noise independent of features. Although not the main objective of this paper, our new concentration inequality for the Multiplier Process (see Theorem \ref{thm:mult:process} in the Supplemental Material) imply oracle inequalities for the Lasso and Slope with $\delta$-independent tuning and with the optimal subgaussian rate when the noise may depend on the features. These oracle inequalities are presented in detail in Section \ref{ss:oracle:inequalities} of the Supplemental Material.  

The proof of Theorem \ref{thm:robust:MC} for noisy robust matrix completion has connections with \cite{2012negahban:wainwright,  2017klopp:lounici:tsybakov}. These papers, however, invoke a \emph{two-sided} concentration inequality for bounded processes. One noticeable difference in our proof is that we use a  \emph{one-sided} tail inequality by Bousquet in order to derive a sufficient lower bound. This lower bound readily suggests a cone in the augmented space $\mdR^p\times\re^n$ at which a restricted eigenvalue condition (RE) is satisfied for the augmented design 
$
\frM(\bfB;\btheta):=\frX(\bfB)+\sqrt{n}\btheta.
$
In the setting with no corruption ($\btheta\equiv0$), our argument requires a RE condition in a strictly \emph{smaller} cone than in \cite{2012negahban:wainwright}.

\section{Formal results}\label{s:main:results}
We first present some notation. We say that $a\lesssim b$   if $a\le Cb$ for some absolute constant $C>0$ and $a\asymp b$ if $a\lesssim b$ and $b\lesssim a$. Given $\ell\in\mathbb{N}$, $[\ell]:=\{1,\ldots, \ell\}$. The
$\psi_2$-Orlicz norm will be denoted by $|\cdot|_{\psi_2}$. Finally, if $\Pi$ is the distribution of $\bfX$, we define, for any $\bfV\in\mdR^p$, the pseudo-norm
$$
\Vert\bfV\Vert_{\Pi}:=\sqrt{\esp[\llangle\bfX,\bfV\rrangle^2]}.
$$ 

\subsection{Robust sparse/low-rank least-squares regression}

For the problems of robust sparse linear regression and trace-regression, the following standard subgaussian condition will be assumed (see Theorems \ref{thm:response:sparse:regression}, \ref{thm:response:trace:regression} and \ref{thm:tr:reg:matrix:decomp}). 
\begin{assumption}\label{assump:distribution:subgaussian}
Assume the random pair $(\bfX,\xi)$, possibly non-independent, is such that
\begin{itemize}
\item[\rm{(i)}] There exists $L\ge1$ such that
$
|\llangle\bfX,\bfV\rrangle|_{\psi_2}\le L\Vert\bfV\Vert_\Pi
$
 for all $\bfV\in\mdR^{p}$.
\item[\rm{(ii)}] $1\le|\xi|_{\psi_2}<\infty$. 
\end{itemize}
We define $\sigma_\xi^2:=\esp[\xi^2]$ and $\sigma^2:=|\xi|^2_{\psi_2}$.
\end{assumption}

We start stating our results for robust sparse linear regression.
\begin{theorem}[Robust sparse linear regression]\label{thm:response:sparse:regression}
In the framework of Section \ref{ss:response:trace:regression:intro}, suppose   
$\bb^*$ in Example \ref{example:sparse:linear:regression} is a $s$-sparse vector. Denote by $\bfSigma$ the covariance matrix of $\bx$ and 
let 
$$
\rho_1(\bfSigma):=\max_{j\in[p]}\sqrt{\bfSigma_{jj}}.
$$ 
Grant Assumptions \ref{assump:label:contamination}  and \ref{assump:distribution:subgaussian} such that 
$\sigma^2 L^2\epsilon\log(1/\epsilon)\le\sc$ for some universal constant $\sc\in(0,1)$. Define $r_n:=\sqrt{s\log p/n}$ and $S=s\log p$.

In estimator \eqref{equation:sorted:Huber} with sequence $\omega_i=\sqrt{\log(An/i)}$ for some $A\ge2$, take $\calR$ to be the $\ell_1$-norm and tuning parameters $\tau\asymp\sigma/\sqrt{n}$ and
\begin{align}
\lambda\asymp L\sigma\rho_1(\bfSigma)\sqrt{\frac{\log p}{n}}.
\end{align}

Then there are absolute positive constants $C$, $c_1$ and $c_2$ and constant $\mu(\bb^*)\ge0$ for which the following property holds. For any $\delta\in(0,1)$, if one has
\begin{align}
n&\ge C\sigma^2 L^2\rho^2_1(\bfSigma)\mu^2(\bb^*) S,\\
\delta&\ge \exp\left(-c_1\frac{n}{L^4}\right)\bigvee \exp\left(-c_2\frac{n}{\sigma^2 L^2}\right),
\end{align}
then, with probability at least $1-\delta$, 
\begin{align}
		\big\|\hat\bb-\bb^*\big\|_\Pi &\lesssim
L^3\sigma^2\rho^2_1(\bfSigma)\mu^2(\bb^*)\cdot r_n^2+L\sigma\rho_1(\bfSigma)\mu(\bb^*)\cdot r_n\\
&\quad+(L^3\sigma^2+L^2\sigma)\frac{1+\log(1/\delta)}{n}
+L\sigma\frac{1+\sqrt{\log(1/\delta)}}{\sqrt{n}}\\
&\quad+L\sigma^2\cdot\epsilon\log(1/\epsilon).
\end{align}

If, instead of Lasso penalization, one takes $\calR$ to be the Slope norm in $\re^p$ with the sequence $\bar w_j=\log(\bar A p/j)$ for some $\bar A\ge2$ and tuning 
\begin{align}
\lambda\asymp \frac{L\sigma\rho_1(\bfSigma)}{\sqrt{n}},
\end{align}
a similar bound holds but with $S=s\log(ep/s)$ and $r_n=\sqrt{s\log(ep/s)/n}$.
\end{theorem}

Our bound for the robust low-rank trace-regression is stated next.
\begin{theorem}[Robust trace-regression]\label{thm:response:trace:regression}
In Example \ref{example:trace:regression} in Section \ref{ss:response:trace:regression:intro}, suppose   
$\bfB^*$ has rank $r$. Denote by $\bfSigma$ the covariance matrix of $\bfX$ seen as a vector in $\re^p$ and let 
$$
\rho_N(\bfSigma):=\sup_{\Vert[\bz;\bv]\Vert_2=1}\sqrt{\esp[(\bz^\top\bfX\bv)^2]}.
$$
Grant Assumptions \ref{assump:label:contamination}  and \ref{assump:distribution:subgaussian} such that $\sigma^2 L^2\epsilon\log(1/\epsilon)\le\sc$ for some universal constant $\sc\in(0,1)$. In estimator \eqref{equation:sorted:Huber} with sequence $\omega_i=\sqrt{\log(An/i)}$ for some $A\ge2$, take $\calR$ to be the 
nuclear norm and tuning parameters $\tau\asymp\sigma/\sqrt{n}$ and 
\begin{align}
\lambda\asymp L\sigma\rho_N(\bfSigma)\left(\sqrt{\frac{d_1}{n}}+\sqrt{\frac{d_2}{n}}\right).
\end{align}

Then there are absolute positive constants $C$, $c_1$ and $c_2$ and constant $\mu(\bfB^*)\ge0$ for which the following property holds. For any $\delta\in(0,1)$, if one has
\begin{align}
n&\ge C\sigma^2 L^2\rho^2_N(\bfSigma)\mu^2(\bfB^*)\cdot r(d_1+d_2),\\
\delta&\ge \exp\left(-c_1\frac{n}{L^4}\right)\bigvee \exp\left(-c_2\frac{n}{\sigma^2 L^2}\right),
\end{align}
then, with probability at least $1-\delta$, 
\begin{align}
		\big\|\hat\bfB-\bfB^*\big\|_\Pi &\lesssim
L^3\sigma^2\rho^2_N(\bfSigma)\mu^2(\bfB^*)\frac{r(d_1+d_2)}{n}+L\sigma\rho_N(\bfSigma)\mu(\bfB^*)\left(\sqrt{\frac{r d_1}{n}}+\sqrt{\frac{r d_2}{n}}\right)\\
&\quad+(L^3\sigma^2+L^2\sigma)\frac{1+\log(1/\delta)}{n}
+L\sigma\frac{1+\sqrt{\log(1/\delta)}}{\sqrt{n}}\\
&\quad+L\sigma^2\cdot\epsilon\log(1/\epsilon).
\end{align}
\end{theorem}

The rates in Theorems \ref{thm:response:sparse:regression} and \ref{thm:response:trace:regression} are optimal up to a log factor in $1/\epsilon$ \cite{2016chen:gao:ren,2018chen:gao:ren, 2020gao}. We make a remark concerning the constant $\rho_1(\bfSigma)$ used in the tuning parameters 
$(\lambda,\tau)$. In Theorem \ref{thm:response:sparse:regression}, we may assume without loss on generality that $\rho_1(\bfSigma)$ is \emph{unknown} replacing them by 
the estimates $\hat\rho_1:=\max_{j\in[p]}\Vert\mbX_{\bullet,j}\Vert_2$. Indeed, concentration upper bounds\footnote{See Proposition \ref{prop:gen:TP} in the supplemental material.} for subgaussian $\bx$ implies that 
$|\hat\rho_1-\rho_1(\bfSigma)|\le c_n\rho_1(\bfSigma)$, where  $c_n$ converges to zero with the sample size scaling of Theorem \ref{thm:response:sparse:regression}. The same observation holds for Theorem \ref{thm:response:trace:regression}, replacing
$\rho_N(\bfSigma)$ by the estimate 
$$
\hat\rho_N:=\sup_{\Vert[\bz;\bv]\Vert_2=1}\sqrt{\frac{1}{n}\sum_{i\in[n]}(\bz^\top\bfX_i\bv)^2]}. 
$$

\begin{remark}[Restricted eigenvalue constants]
In Theorem \ref{thm:response:sparse:regression}, 
$
\mu(\bb^*)=\sup_{\bv\in\mcC}\nicefrac{\Vert\bv\Vert_2}{\Vert\bv\Vert_\Pi}
$
is the usual restricted eigenvalue constant, where $\mcC$ is a dimension reduction cone associated to the $\ell_1$-norm and ``sparsity support'' of $\bb^*$ (see Section \ref{ss:cones:RE} in the Supplemental Material).  Analogous observations hold for  Theorem \ref{thm:response:trace:regression}. In that case, 
$
\mu(\bfB^*)=\sup_{\bfV\in\mcC'}\nicefrac{\Vert\bfV\Vert_F}{\Vert\bfV\Vert_\Pi}
$ 
where $\mcC'$ is a dimension reduction cone associated
to the nuclear norm and the ``low-rank support'' of $\bfB^*$. The constant $\mu(\bfB^*)$ is the condition number measuring how far 
$\Pi$ is from the isotropic distribution. 
\end{remark}

Finally, we present the following upper bound on the estimation of the trace-regression problem with matrix decomposition.
\begin{theorem}[Trace-regression with matrix decomposition]
\label{thm:tr:reg:matrix:decomp}
In the problem of trace-regression with matrix decomposition of Section \ref{ss:matrix:decomposition:intro}, suppose   
$\bfB^*$ has rank $r$ and $\bfGamma^*$ has at most $s$ nonzero entries. Grant Assumptions \ref{assump:low:spikeness} and \ref{assump:distribution:subgaussian}. In estimator \eqref{equation:aug:estimator:trace:reg:matrix:decomp}, take tuning parameters
\begin{align}
\lambda\asymp\sigma L\left(\sqrt{\frac{d_1}{n}}+\sqrt{\frac{d_2}{n}}\right),\quad
\tau\asymp\sigma L\sqrt{\frac{\log p}{n}}+\frac{\sa^*}{\sqrt{n}}.
\end{align}

Then there are absolute positive constants $C$, $c_1$ and $c_2$ for which the following property holds. For any $\delta\in(0,1)$, if one has
\begin{align}
n&\ge C\left[\sigma^2 L^2\cdot r(d_1+d_2)\right]\bigvee
\left[\left(\sigma^2L^2\log p
+(\sa^*)^2\right)\cdot s\right],\\
\delta&\ge \exp\left(-c_1\frac{n}{L^4}\right)\bigvee \exp\left(-c_2\frac{n}{\sigma^2L^2}\right).
\end{align}
then, with probability at least $1-\delta$, 
\begin{align}
\sqrt{\big\|\hat\bfB-\bfB^*\big\|_F^2
+\big\|\hat\bfGamma-\bfGamma^*\big\|_F^2} &\lesssim
\sigma L\left(\sqrt{\frac{r d_1}{n}}+\sqrt{\frac{r d_2}{n}}\right)
+\sigma L\sqrt{\frac{s\log p}{n}}
+\sa^*\sqrt{\frac{s}{n}}\\
&\quad+\sigma L\frac{1+\sqrt{\log(1/\delta)}}{\sqrt{n}}
+\sigma^2 L^2\frac{1+\log(1/\delta)}{n}.
\end{align}
\end{theorem}

The next proposition assures that the previous rate is optimal up to a log factor and constants. Its proof follows from similar arguments in \cite{2012agarwal:negahban:wainwright} for the noisy matrix decomposition problem with identity design. Define the class
\begin{align}
\calA(r,s,\sa^*)=\left\{\bfTheta^*:=[\bfB^*;\bfGamma^*]\in(\mdR^p)^2:\rank(\bfB^*)\le r, \Vert\bfGamma^*\Vert_0\le s, \Vert\bfB^*\Vert_\infty\le\frac{\sa^*}{\sqrt{n}}\right\}. 
\end{align}
For any $\bfTheta^*:=[\bfB^*,\bfGamma^*]\in(\mdR^p)^2$, let $\prob_{\bfTheta^*}$ denote the distribution of the data $\{y_i,\bfX_i\}_{i\in[n]}$ satisfying \eqref{equation:str:eq:trace:reg:matrix:decomp} with parameters $[\bfB^*,\bfGamma^*]$. Finally, for some 
$\sigma>0$, let
\begin{align}
\Psi_n(r,s,\sa^*):=\sigma^2\left\{\frac{r(d_1+d_2)}{n}
+\frac{s}{n}\log\left(\frac{p-s}{s/2}\right)\right\}
+(\sa^*)^2\frac{s}{n}. 
\end{align}

\begin{proposition}\label{prop:lower:bound:tr:matrix:decomp}
In the model \eqref{equation:str:eq:trace:reg:matrix:decomp} assume that $\{\xi_i\}_{i\in[n]}$ are iid $\calN(0,\sigma^2)$ independent of $\{\bfX_i\}_{i\in[n]}$ and Assumption \ref{assump:low:spikeness} holds. Assume $d_1,d_2\ge10$, $\sa^*\ge 32\sqrt{\log p}$ and $s<p$. 

Then there exists universal constants $c>0$ and $\beta\in(0,1)$ such that
\begin{align}
\inf_{\hat\bfTheta}\sup_{\bfTheta^*\in\calA(r,s,\sa^*)}\prob_{\bfTheta^*}
\left\{\Vert\hat\bfB-\bfB^*\Vert_F^2+\Vert\hat\bfGamma-\bfGamma^*\Vert_F^2
\ge c\Psi_n(r,s,\sa^*)\right\}\ge\beta,
\end{align}
where the infimum is taken over all estimators 
$\hat\bfTheta=[\hat\bfB,\hat\bfGamma]$ constructed from the data $\{y_i,\bfX_i\}_{i\in[n]}$. 
\end{proposition}

\subsection{Robust matrix completion}

We follow the same distribution assumption of \cite{2014klopp, 2017klopp:lounici:tsybakov} for the robust matrix completion problem. For simplicity we only consider subgaussian noise. It is stated as follows.
\begin{assumption}\label{assump:distribution:MC}
Assume the random pair $(\bfX,\xi)\in\mdR^p\times\re$ is such that
\begin{itemize}
\item[\rm{(i)}] $\bfX$ has a discrete distribution $\Pi=\{\pi_{k,\ell}\}_{(k,\ell)\in[d_1]\times[d_2]}$ with support on $\calX$ defined in \eqref{equation:canonical:basis}. Let $d:=d_1+d_2$ and $m=d_1\wedge d_2$.
We also define
\begin{align}
R_k:=\sum_{\ell=1}^{d_2}\pi_{k,\ell},\quad\quad
C_\ell:=\sum_{k=1}^{d_1}\pi_{k,\ell},\quad\quad
L:=m\cdot\max_{k\in[d_1],\ell\in[d_2]}\{R_k,C_\ell\}.
\end{align}
\item[\rm{(ii)}] $\sigma:=|\xi|_{\psi_2}\in(1,\infty)$ and 
$\sigma_\xi^2:=\esp[\xi^2]\ge1$.
\end{itemize}
\end{assumption}

Given $\delta\in(0,1)$ and $\mu>0$, define
$$
r_{n,\delta}(\mu):=\mu\sqrt{\frac{Lpr}{mn}}\sqrt{\log(d/\delta)}
		+ \mu\log^{1/2}\left(\frac{\sigma m}{\sigma_\xi}\right)\frac{\sqrt{pr}}{n}\log(d/\delta) 
		+ \frac{\sqrt{1+\log(1/\delta)}}{\sqrt{n}}.
$$

\begin{theorem}\label{thm:robust:MC}
In the robust matrix completion problem of Section \ref{ss:robust:matrix:completion:intro}, suppose   
$\bfB^*$ has rank $r$. Grant Assumptions \ref{assump:label:contamination}  and \ref{assump:distribution:MC} such that $\epsilon<0.5$. Let $\delta\in(0,1)$. In estimator \eqref{equation:aug:slope:rob:estimator:MC} with sequence $\omega_i=\sqrt{\log(An/i)}$ for some $A\ge2$, take $\calR$ to be the 
nuclear norm, $\sa>0$ such that $\Vert\bfB^*\Vert_\infty\le\sa$  and tuning parameters $\tau\asymp\frac{\sigma\vee\sa}{\sqrt{n}}$
and
\begin{align}
\lambda\asymp\left[(\sigma\vee\sa)\sqrt{\frac{Lp}{mn}}\sqrt{(\log(d/\delta)}\right]\bigvee
\left[(\sigma\vee\sa)\log^{1/2}\left(\frac{\sigma m}{\sigma_\xi}\right)\frac{\sqrt{p}}{n}\log(d/\delta)\right].
\end{align}

Let $\mbC_{\bfB^*}:=\calC_{\bfB^*,\Vert\cdot\Vert_N}(4)$ be the cone defined in Definition \ref{def:dim:red:cones} of Section \ref{ss:cones:RE} in the supplement and $\mu(\bfB^*):=\mu(\mbC_{\bfB^*})/\sqrt{p}$. Then, with probability at least $1-\delta$, 
\begin{align}
\Vert[\bfDelta_{\bfB_*},\bfDelta^{\hat\btheta}]\Vert_\Pi
	&\lesssim (\sa\vee\sigma)r_{n,\delta}(\mu(\bfB^*)) + (\sa\vee\sigma)\sqrt{\epsilon\log(1/\epsilon)},\\
	\lambda\frac{\Vert\bfDelta_{\bfB_*}\Vert_N}{\sqrt{p}} + \tau \|\bfDelta^{\hat\btheta}\|_\sharp
	&\lesssim (\sa\vee\sigma)^2r_{n,\delta}^2(\mu(\bfB^*)) + (\sa\vee\sigma)^2\epsilon\log(1/\epsilon). 
\end{align}
\end{theorem}
Let us make a few comments about Theorem \ref{thm:robust:MC}. A similar argument by  \cite{2017klopp:lounici:tsybakov} shows that the rate of Theorem \ref{thm:robust:MC} is optimal up to log factors over the class $\bar\calA(r,o,\sa)$ defined in \eqref{equation:class:RMC}. If $\Pi$ is the uniform distribution  then $L=1$ and $\mu(\bfB^*)=1$. The above rate is also meaningful over a large class of non-uniform sampling distributions $\Pi$ having $L$ and $\mu(\bfB^*)$ of reasonable magnitudes (see Remark \ref{rem:RE:constants} in the following). In practice, one may just know an upper bound on $L$ rather than its exact value.\footnote{As usual, the corresponding rate may scale with a larger constant.} Finally, the estimator and correspondent rate in \cite{2017klopp:lounici:tsybakov} depend on an upper bound on the corruption sup-norm $\Vert\bfGamma^*\Vert_\infty$. In many applications, it is reasonable to expect that only the low-rank matrix 
$\bfB^*$ is of interest. In that specific setting, Theorem \ref{thm:robust:MC} reveals that only an upper bound on $\Vert\bfB^*\Vert_\infty$ is of relevance. In particular, the corruption level is irrelevant: it does not affect the rate nor it is necessary for tuning the estimator. We confirm this theoretical finding in our numerical experiments.   

\begin{remark}[Restricted eigenvalue constants]\label{rem:RE:constants}
In Theorem \ref{thm:robust:MC}, 
$
\mu(\bfB^*)=\sup_{\bfV\in\mcC'}\nicefrac{\Vert\bfV\Vert_F}{\sqrt{p}\Vert\bfV\Vert_\Pi}
$
for a specific cone $\mcC'$. If $\Pi$ is the uniform distribution then $\mu(\bfB^*)=1$. In general, 
$\mu(\bfB^*)$ is the condition number measuring how far 
$\Pi$ is from the uniform distribution. 
\end{remark}

\section{Simulation results}\label{s:simulation}
We report simulation
results in \texttt{R} with synthetic data demonstrating excellent agreement between our theoretical 
predictions (Theorems \ref{thm:response:sparse:regression}-\ref{thm:robust:MC}) and the behavior in practice. The code can be accessed in \texttt{https://github.com/philipthomp/Outlier-robust-regression}. For robust sparse linear regression and trace-regression problems, we simulate a design with i.i.d. 
$\calN(0,1)$ entries and $\calN(0,1)$ noise. For noisy robust matrix completion, the sampling design is uniformly distributed over $[d_1]\times[d_2]$ and the noise is $\calN(0,1)$. For numerical reasons, we  solve \eqref{equation:aug:slope:rob:estimator:MC} with the scaled design $\sqrt{p}\bfX_i$. We solve \eqref{equation:aug:slope:rob:estimator}, \eqref{equation:aug:slope:rob:estimator:MC} and \eqref{equation:aug:estimator:trace:reg:matrix:decomp} implementing a batch version of a proximal gradient method on the separable variables 
$[\bfB;\btheta]$ using a stepsize equal to $0.25$. The proximal map of the Slope or $\ell_1$ norms are computed with the function \texttt{prox\_sorted\_L1()} of the \texttt{glmSLOPE} package \cite{2015bogdan:berg:sabatti:su:candes}. The ($\ell_\infty$-constrained) proximal map of the nuclear norm is computed via ($\ell_\infty$-constrained) soft-thresholding of the singular value decomposition.

\emph{Robust sparse linear regression}. We simulate model  \eqref{equation:structural:equation}  with   $p=100$ and $n=1000$. This model is simulated for 3 different sparsity indexes $s = 15, 25, 35$ over a grid of corruption fraction $\epsilon=\frac{o}{n}$. Respectively, $\bb^*$ and $\btheta^*$ are set with the first $s$ and $o$ entries equal to $10$ and zero otherwise. For each $s$ and 
$\epsilon$, we solve \eqref{equation:aug:slope:rob:estimator} taking $A=10$ and 
$\calR$ to be the Slope norm with $\bar A=10$. For each $s$, the plot of the square root $\sqrt{\texttt{MSE}}$ of the mean squared error ($\texttt{MSE}$) as a function of $\epsilon$ (averaged over 100 repetitions) is shown in Figure \ref{fig.robust.linear.reg}(a). As predicted by Theorem \ref{thm:response:sparse:regression}, the plot fits fairly well a linear growth with $\epsilon$.  Taking $\calR$ to be the Slope norm, we also compare Huber's regression (H) with the ``Sorted Huber's loss'' (S) in Definition \ref{def:sorted:Huber:loss} for $s=25$. Fixing all model parameters, we increase the nonzero entries of $\bb^*$ and $\btheta^*$ to $50$. In Figure \ref{fig.robust.linear.reg}(b), we show the correspondent plots of the square root of the MSE as a function of $\epsilon$ (averaged over 100 repetitions). We see that ``S'' clearly outperforms ``H''. The first 25 entries estimated by ``H'' and ``S'' fluctuated around $40$ and $48$ respectively.

\begin{figure}
\hspace*{\fill}%
\subcaptionbox{Different sparsity.}{\includegraphics[scale=0.27]{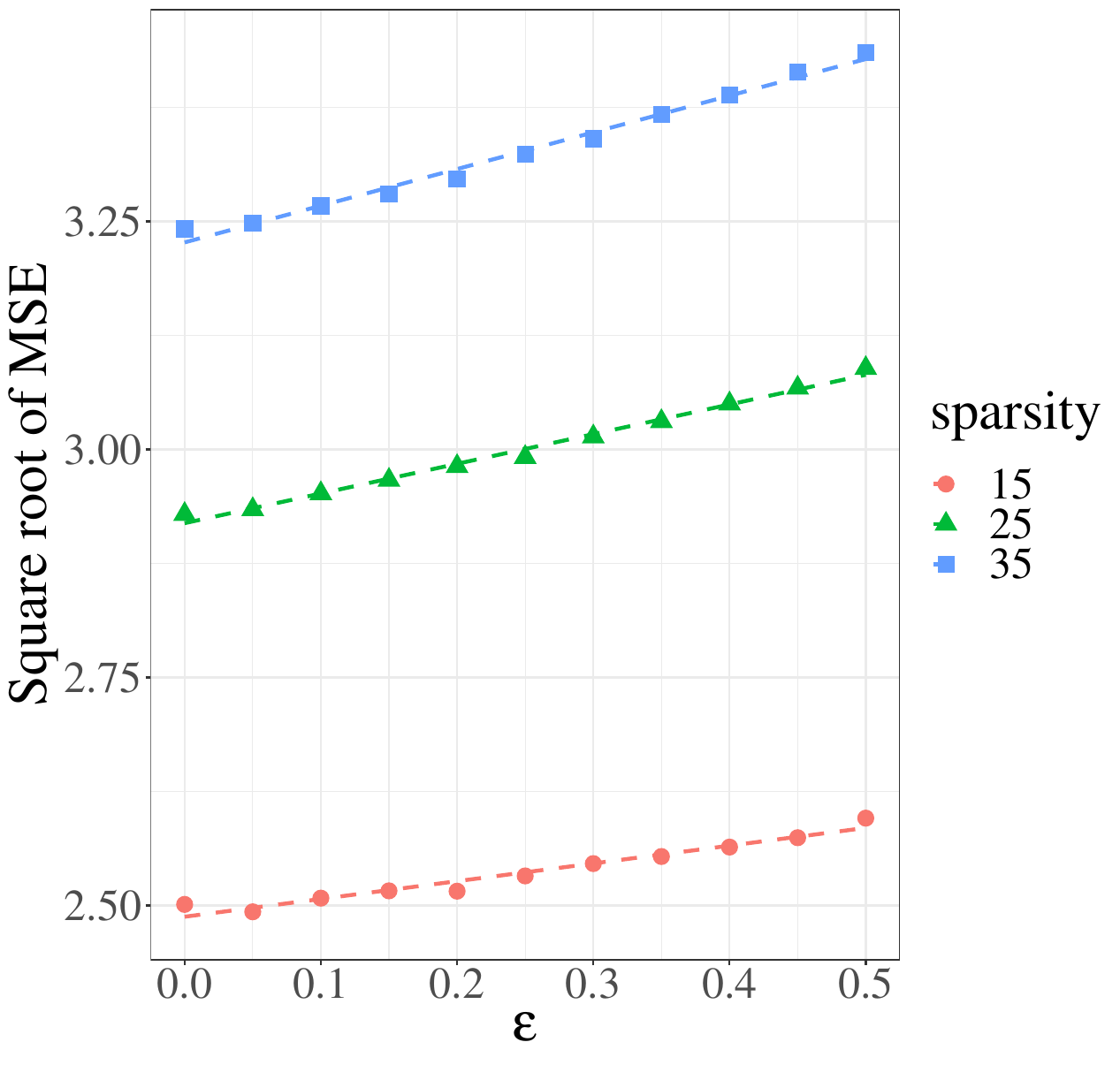}}\hfill%
\subcaptionbox{Different losses.}{\includegraphics[scale=0.27]{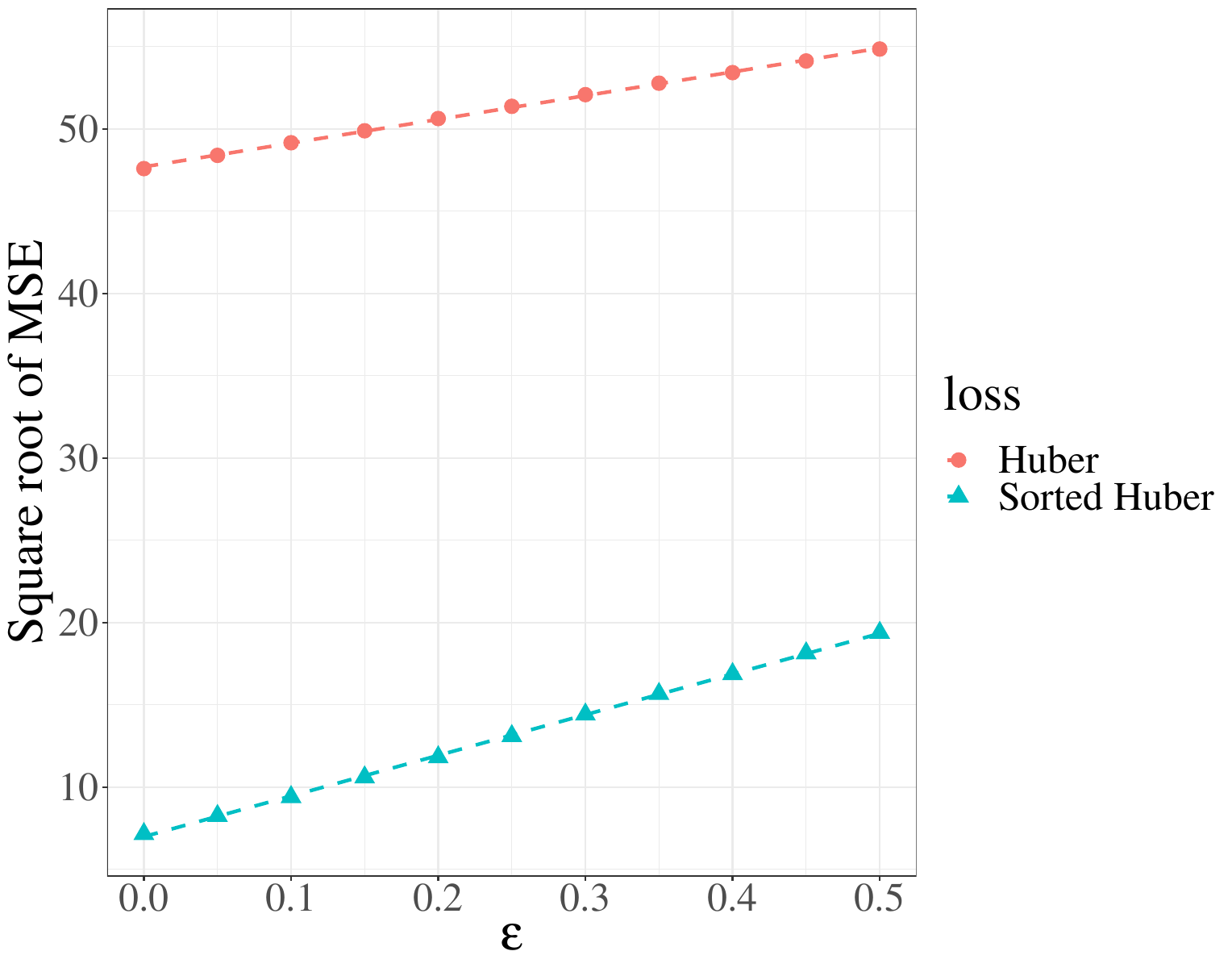}}
\hspace*{\fill}%
\caption{Robust sparse linear regression: $\sqrt{\texttt{MSE}}$ versus $\epsilon$.}\label{fig.robust.linear.reg}
\end{figure}

\emph{Robust low-rank trace-regression}. We simulate model  \eqref{equation:structural:equation} with $d_1=d_2=10$ and $n=1000$. This model is simulated for 3 different rank values $r = 1, 2, 3$ over a grid of corruption fraction $\epsilon=\frac{o}{n}$. The low-rank parameter is generated by randomly
choosing the spaces of left and right singular vectors with all nonzero singular values equal to $10$. The corruption vector $\btheta^*$ is set to have the first $o$ entries equal to $10$ and zero otherwise. We solve \eqref{equation:aug:slope:rob:estimator} taking $\calR$ to be the nuclear norm and $A=10$. For each $r$, the plot of $\sqrt{\texttt{MSE}}$ as a function of $\epsilon$ (averaged over 50 repetitions) is given in Figure \ref{fig.robust.trace-reg}(a). The plot fits fairly well a linear growth with $\epsilon$ as predicted by Theorem \ref{thm:response:trace:regression}. 

\begin{figure}
\hspace*{\fill}%
\subcaptionbox{Robust trace regression.}{\includegraphics[scale=0.27]{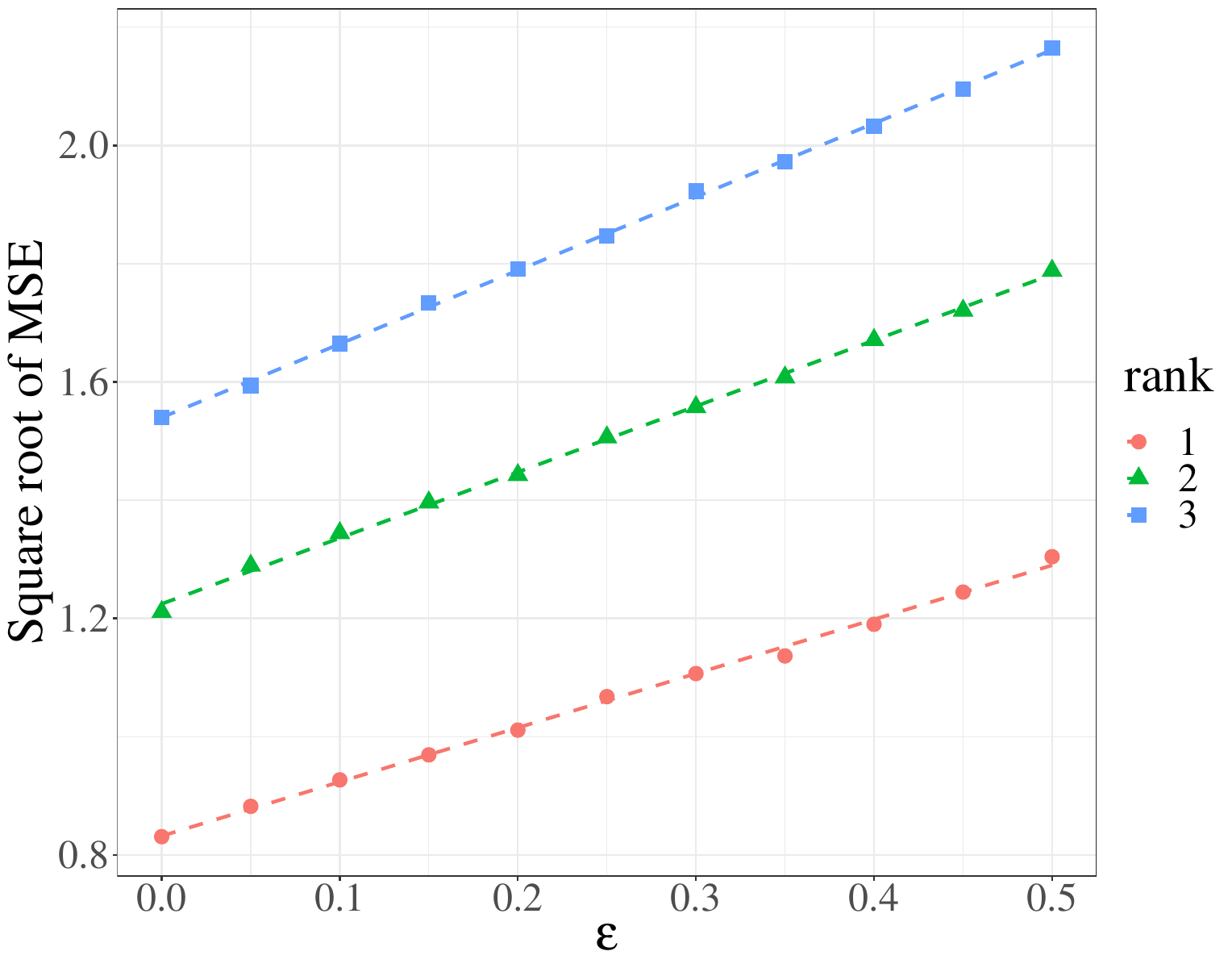}}\hfill%
\subcaptionbox{Robust matrix completion.}{\includegraphics[scale=0.27]{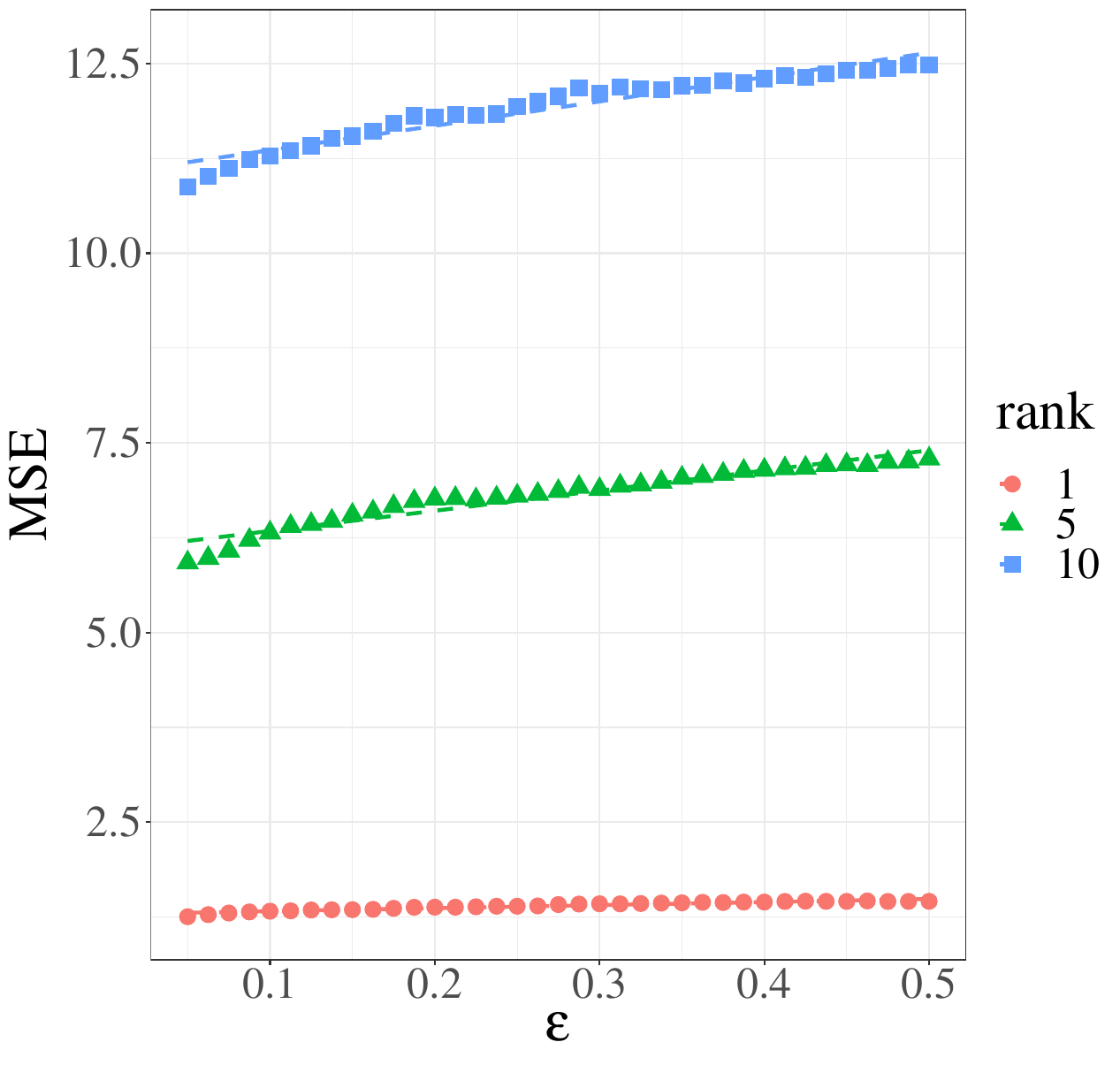}}
\hspace*{\fill}%
\caption{Robust trace regression \& matrix completion: error versus $\epsilon$.}\label{fig.robust.trace-reg}
\end{figure}

\emph{Robust matrix completion}. We simulate model  \eqref{equation:structural:equation} with $d_1=d_2=10$ and $n=80$. This model is simulated for 3 different rank values $r = 1, 5, 10$ over a grid of corruption fraction $\epsilon=\frac{o}{n}$. The low-rank parameter is generated by randomly
choosing the spaces of left and right singular vectors such that $\Vert\bfB^*\Vert_\infty\le1$. The corruption vector $\btheta^*$ is set to have the first $o$ entries equal to $10$ and zero otherwise. We solve \eqref{equation:aug:slope:rob:estimator:MC} taking 
$\calR$ to be the nuclear norm and $A=10$. For each $r$, the plot of $\texttt{MSE}$ as a function of $\epsilon$ (averaged over 100 repetitions) is given in Figure \ref{fig.robust.trace-reg}(b). The plot fits fairly well a linear growth with $\epsilon$ as predicted by Theorem \ref{thm:robust:MC}.\footnote{We have noted, however, that the log factor $\log(C/\epsilon)$ in the MSE rate of the form $\epsilon\mapsto\epsilon\log(C/\epsilon)$, as predicted in Theorem \ref{thm:robust:MC}, is more present around
$\epsilon=0$.} We also study the variability of the corruption level $\Vert\btheta^*\Vert_\infty$. Fixing the previous model parameters and taking $r=3$, we show in the highlighted table 
$\sqrt{\texttt{MSE}}$ versus $\epsilon\in[0.05,0.35]$ (averaged over 100 repetitions) for $\Vert\btheta^*\Vert_\infty = 10, 100, 1000$. For a precise comparison, the same data set is used for  the three values of $\Vert\btheta^*\Vert_\infty$ in each repetition. As predicted by Theorem \ref{thm:robust:MC}, the resulted error is fairly adaptive and robust  with respect to the corruption level $\Vert\btheta^*\Vert_\infty$. 

\begin{table}
\centering
 \begin{tabular}{||c || c c c c c c c ||} 
  \hline
  &0.05&0.1&0.15&0.2&0.25&0.3&0.35 \\ [0.5ex] 
 \hline\hline
 10  &1.895272&1.967005&2.000942&2.032264&2.056413&2.074631&2.075563 \\ 
 \hline
 100 &1.895273&1.966962&2.000966&2.032263&2.056416&2.074635&2.075567 \\
 \hline 
 1000 &1.895276&1.967034&2.000997&2.032263&2.056426&2.074644&2.075579 \\ [1ex] 
 \hline
 \end{tabular}
\end{table}

\emph{Trace-regression with matrix decomposition}. We simulate model \eqref{equation:matrix:decomp} with $d_1=d_2=10$ and $n=1000$. The low-rank parameter is generated by randomly
choosing the spaces of left and right singular vectors such that $\Vert\bfB^*\Vert_\infty\le\frac{\sa^*}{\sqrt{n}}$ with $\sa^*=1$. The sparse parameter is simulated with the non-zero entries of value 10 chosen uniformly at random. We solve \eqref{equation:aug:slope:rob:estimator:MC} in two settings. First, this model is simulated for 2 different sparsity values $s = 5, 80$ over a grid of ranks $r$. For each $s$, the plot of $\texttt{MSE}$ versus $r$ (averaged over 20 repetitions) is given in Figure \ref{fig.trace-reg-MD}(a). Secondly, we simulate for rank $r = 5$ over a grid of sparsity levels $s$. The plot of the $\texttt{MSE}$ versus $s$ (averaged over 20 repetitions) is given in Figure \ref{fig.trace-reg-MD}(b). Both plots fit fairly well a linear growth as predicted by Theorem \ref{thm:tr:reg:matrix:decomp}. 

\begin{figure}
\hspace*{\fill}%
\subcaptionbox{Fixed sparsity.}{\includegraphics[scale=0.27]{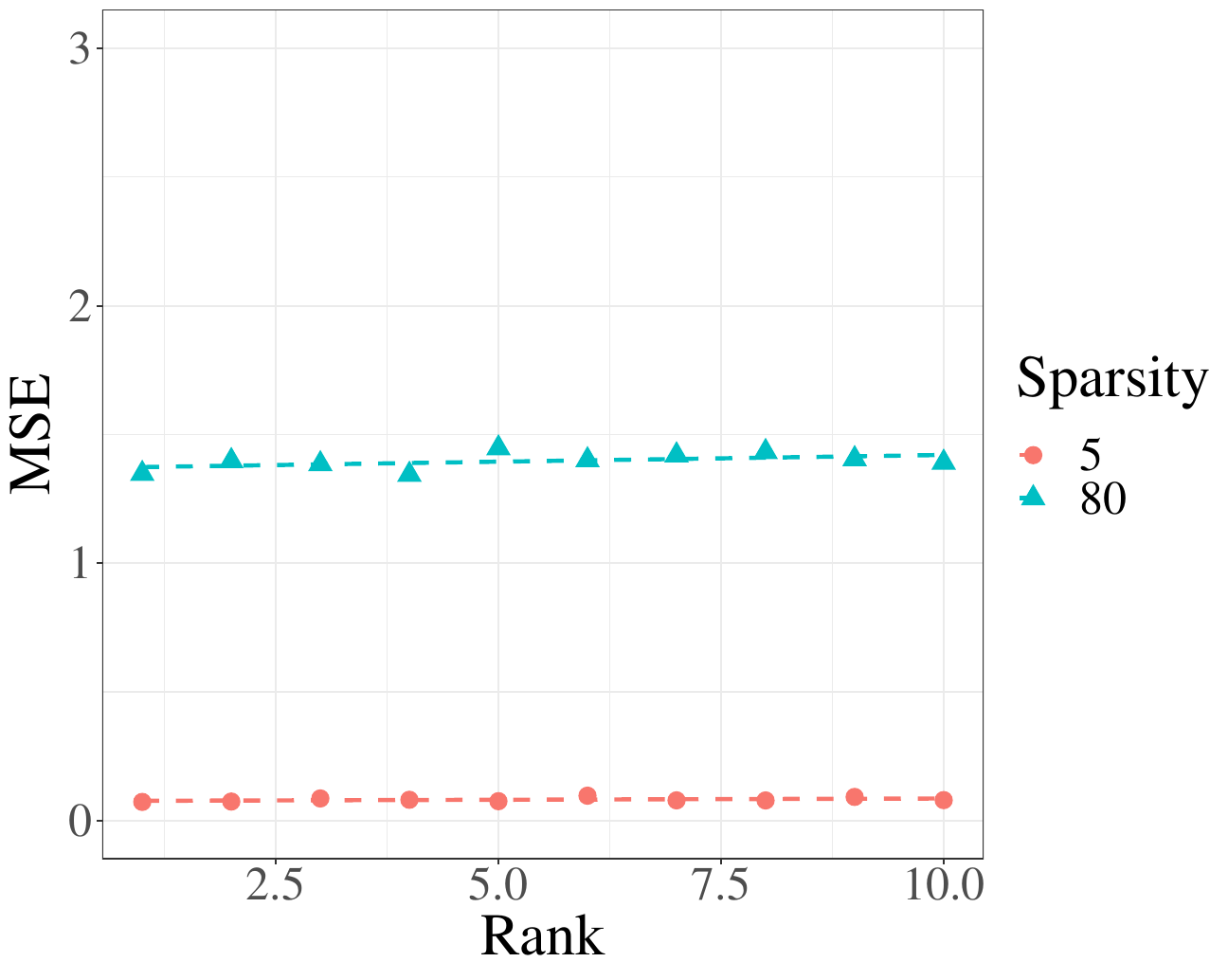}}\hfill%
\subcaptionbox{Fixed rank.}{\includegraphics[scale=0.3]{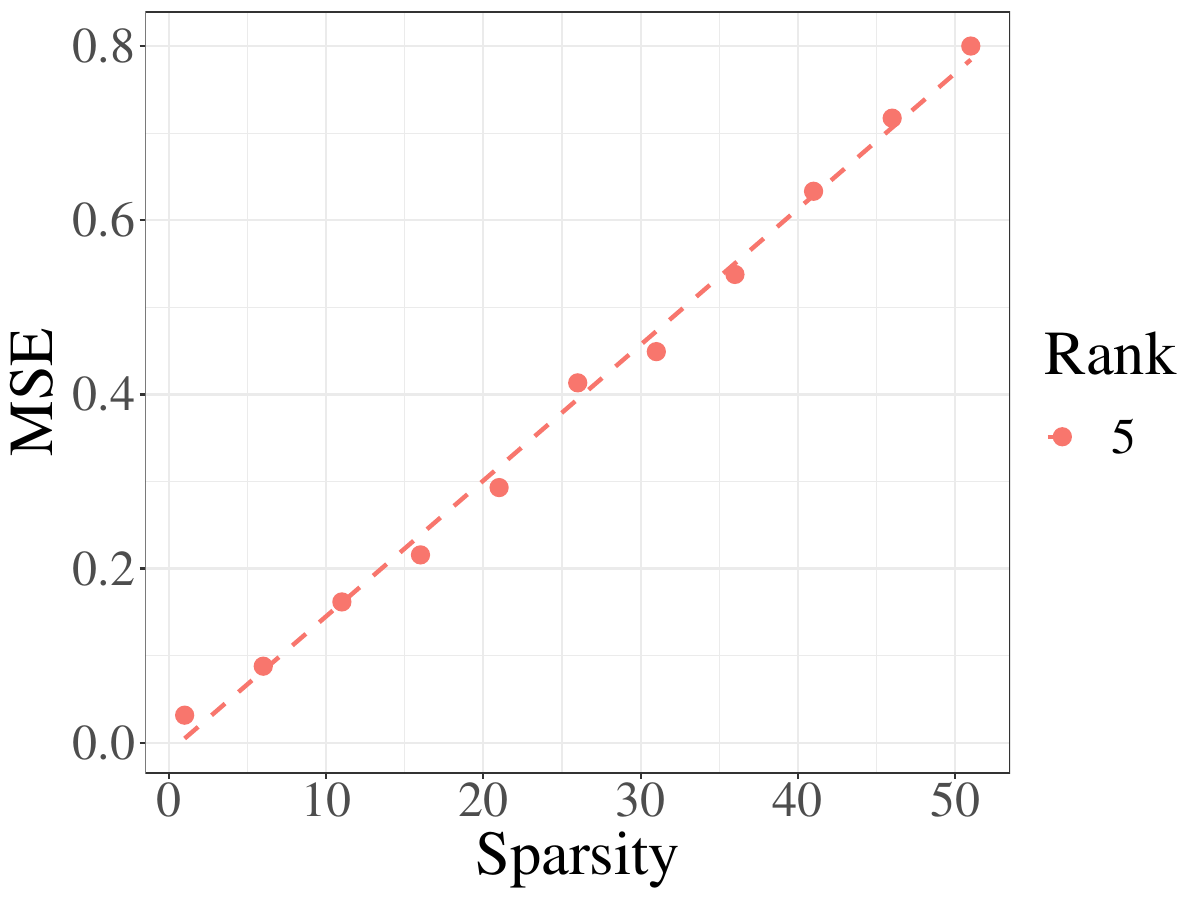}}
\hspace*{\fill}%
\caption{Trace regression with matrix decomposition: plot of $\texttt{MSE}$.}\label{fig.trace-reg-MD}
\end{figure}

\section{Discussion}\label{s:discussion}

We finalize with some additional discussions complementing our contributions already highlighted in Section \ref{ss:contributions:related:work}. The $\ell_1$ norm treats each variable in the same manner. Alternatively, the Slope norm adjusts the weight according to the variable's magnitude. The papers \cite{2015bogdan:berg:sabatti:su:candes, 2018bellec:lecue:tsybakov} propose and study the Slope norm as a finer \emph{regularization} norm for sparse linear regression. Indeed, \cite{2018bellec:lecue:tsybakov} shows that sparse linear regression with Slope regularization achieves the \emph{exact} optimal rate 
$\sqrt{s\log(ep/s)/n}$ adaptively to $s$, a property not satisfied by the $\ell_1$-norm. In this work, we exploit the Slope norm for a related but considerably different purpose: to construct a finer \emph{robust loss} against label contamination (see Definition \ref{def:sorted:Huber:loss}). Differently than classical Huber regression, the ``Sorted Huber's loss'' in Definition \ref{def:sorted:Huber:loss} assign, for each $i$th data point, higher penalization to higher outlier levels. One message of the present work is that $M$-estimation with the Sorted Huber's loss seem to have theoretical improvements compared to Huber regression: it achieves the \emph{subgaussian} optimal rate in response adversarial robust regression (up to a log factor in $1/\epsilon$) with a breakdown point $\epsilon\le c$ for a constant $c\in(0,1)$, adaptively to $(\epsilon,d_{\mbox{\tiny eff}},\delta)$, where $d_{\mbox{\tiny eff}}$ denotes the ``effective'' dimension and $\delta$ is the failure probability. This property holds despite the parsimony structure of the parameter (sparsity or low-rankness). We also show that these rates are valid \emph{uniformly} on the failure probability $\delta$, i.e., the tuning of our estimators do not depend on $\delta$. Although the details are not presented explicitly, it is clear that, in \emph{low-dimensions} ($n\ge d_1d_2$), subgaussian optimality is also achieved in least-squares regression with the Sorted Huber's loss. Our simulations also shown in practice that Sorted Huber regression can outperform classical Huber regression (Section \ref{s:simulation}, Figure \ref{fig.robust.linear.reg}(b)).

The Huber loss is a seminal and fundamental robust loss which has been considered in numerous works under various different perspectives \cite{2013nguyen:tran, 2011she:owen,2013bean:bickel:karoui:yu,2016donoho:montanari,2020bellec}. One fundamental open question would be to further compare Huber regression with Sorted Huber regression proposed in Definition \ref{def:sorted:Huber:loss} with respect to other statistical metrics considered in the literature. This would be interesting on the theoretical level and with further numerical study. One example would be to give a tight comparison between the breakdown points of both losses, at least in the Gaussian set-up. 

To the best of our knowledge, there seems to be no prior estimation theory for the trace-regression problem with matrix decomposition (Theorem \ref{thm:tr:reg:matrix:decomp}). It is interesting to remark that our analysis of this problem crucially makes use of a concentration inequality for the Product Process. This seems to be a new application in the context of matrix decomposition theory \cite{2011candes:li:ma:wright, 2012agarwal:negahban:wainwright}. In this paper we are mainly concerned with the noisy setting. In parallel with the established theory of signal recovery in Compressed Sensing, it would be interesting to investigate further when \emph{exact} recovery is possible in this context. This probably would require different set of assumptions such as the \emph{incoherence} condition \cite{2009candes:recht}.  Other analytic tools from nonconvex optimization \cite{2020chen:fan:ma:yan-nonconvex} and the corresponding corrupted model are also worthy of investigation. 


\bibliographystyle{plain}
\bibliography{references_JRSSB}

\begin{tcolorbox}
\textbf{Web-based supporting materials for ``Outlier-robust sparse/low-rank least-squares regression and robust matrix completion''}
\end{tcolorbox}

\section{Basic notation}
Throughout the paper, given $\ell\in\mathbb{N}$, $A^{(\ell)}:=A/\sqrt{\ell}$ whenever $A$ is a number, vector or function. Regarding the tuning parameters $(\lambda,\tau)$, we will sometimes use the definition 
$\gamma:=\lambda/\tau$. With respect to the Slope norm in 
$\re^n$ with sequence 
$\bomega=(\omega_i)_{i\in[n]}$ given by $\omega_i=\sqrt{\log(An/i)}$ for some $A\ge2$, we set 
$\Omega:=\{\sum_{i=1}^o\omega_i^2\}^{1/2}$. Note that from the Stirling formula, $\Omega\asymp o\log(n/o)$ \cite{2018bellec:lecue:tsybakov}. Throughout the paper, $\bfDelta^{\hat\bfB}:=\hat\bfB-\bfB^*$, 
$\bfDelta^{\hat\bfGamma}:=\hat\bfGamma-\bfGamma^*$ and 
$\bfDelta^{\hat\btheta}:=\hat\btheta-\btheta^*$. It will be useful to define the map
$
\frM(\bfB;\btheta):=\frX(\bfB)+\sqrt{n}\btheta.
$

Letting $\Pi$ be the distribution of $\bfX$, we define the bilinear form 
$$
\llangle\bfV,\bfW\rrangle_\Pi:=\esp[\llangle\bfX,\bfV\rrangle\llangle\bfX,\bfW\rrangle],
$$
and, as before,  the
$L^2(\Pi)$-pseudo-distance 
$$
\Vert\bfV\Vert_\Pi:=\{\esp[\llangle\bfV,\bfX\rrangle_\Pi^2]\}^{1/2}.
$$ We denote by $\frS$ the covariance operator of $\bfX$, that is, the self-adjoint linear operator on
$\mdR^p$ satisfying 
$\llangle\frS(\bfV),\bfW\rrangle=\llangle\bfV,\bfW\rrangle_\Pi$ for all $\bfV,\bfW$. It will be useful to define the pseudo-norms 
$\Vert[\bfV;\bu]\Vert_\Pi:=\{\Vert\bfV\Vert_\Pi^2+\Vert\bu\Vert_2^2\}^{1/2}$ on $\mdR^p\times\re^n$ and 
$\Vert[\bfV;\bfW]\Vert_\Pi:=\{\Vert\bfV\Vert_\Pi^2+\Vert\bfW\Vert_\Pi^2\}^{1/2}$ on $(\mdR^p)^2$. 

We define the unit balls
$
\mbB_\Pi:=\{\bfV\in\mdR^p:\Vert\bfV\Vert_\Pi\le1\},
$
$
\mbB_F:=\{\bfV\in\mdR^p:\Vert\bfV\Vert_F\le1\},
$
and
$
\mbB_\ell^k:=\{\bv\in\re^k:\Vert\bv\Vert_\ell\le1\}.
$
With a slight abuse of notation, we denote by 
$
\mbB_\calR:=\{\bfV\in\mdR^p:\calR(\bfV)\le1\}
$
and
$
\mbB_\calQ:=\{\bu\in\re^n:\calQ(\bu)\le1\}
$
the correspondent unit balls for norms $\calR$ on 
$\mdR^p$ and $\calQ$ on $\re^n$. All the corresponding unit spheres will take the symbol $\mbS$. 

Finally, the \emph{Gaussian width} of a compact set 
$\calB\subset\re^{k\times \ell}$ is the quantity
$
\mathscr G(\calB):=\esp[\sup_{\bfV\in\calB}\llangle\bfV,\bfXi\rrangle],
$
where $\bfXi\in\re^{k\times \ell}$ is random matrix with iid $\calN(0,1)$ entries. 

\section{Proofs of Theorems \ref{thm:response:sparse:regression} and \ref{thm:response:trace:regression}}
\label{s:proof:thm:sparse:trace:reg} 

Recall estimator \eqref{equation:aug:slope:rob:estimator}. 
\begin{tcolorbox}
Throughout Sections \ref{ss:cones:RE} and \ref{ss:deterministic:bounds},  
 $\Vert\cdot\Vert_\Pi$ can be regarded as a generic pseudo-norm.
Moreover, $\{(\bfX_i,\xi_i)\}_{i\in[n]}$ satisfying \eqref{equation:structural:equation} can be regarded as deterministic with 
$\bxi=(\xi_i)_{i\in[n]}$. As before, 
$\frX$ and $\frM$ are the design and augmented design operators associated to  the sequence $\{\bfX_i\}_{i\in[n]}$. Probabilistic assumptions are used only in Sections \ref{ss:subgaussian:designs}, \ref{ss:proof:thm:response:sparse:regression} and \ref{ss:proof:thm:response:trace:regression}.
\end{tcolorbox}

\subsection{Design properties, cones and restricted eigenvalues}\label{ss:cones:RE}

Next we present some structural properties for the design operator 
$\frX:\mdR^p\rightarrow\re^n$. 

\begin{tcolorbox}
\begin{definition}[Transfer principles]\label{def:design:property}
Let $\calR$ be a norm over $\mdR^p$, $\calQ$ be a norm over $\re^n$ and subsets $\mbC\subset\mdR^p$ and 
$\mbC'\subset\mdR^p\times\re^n$. 
\begin{itemize}\itemsep=0pt
\item[\rm (i)] Given positive numbers $\sa_1$ and 
$\sa_2$, we say that $\frX$ satisfies 
$\TP_{\calR}(\sa_1;\sa_2)$ on $\mbC$ if 
\vspace{-3pt}
\begin{align}
\forall\bfV\in\mbC,\quad\big\|\frX^{(n)}(\bfV)\big\|_2\ge \sa_1\Vert\bfV\Vert_\Pi-\sa_2\calR(\bfV).
\end{align}
\item[\rm (ii)] Given positive numbers $\sb_1$, $\sb_2$ and $\sb_3$, we say that $\frX$ satisfies 
$\IP_{\calR,\calQ}(\sb_1;\sb_2;\sb_3)$ if
\vspace{-3pt}
\begin{align}
\forall[\bfV;\bu]\in\mdR^p,\quad|\langle\bu,\frX^{(n)}(\bfV)\rangle|\le \sb_1\left\Vert\bfV\right\Vert_{\Pi}
\Vert\bu\Vert_2+\sb_2\calR(\bfV)\Vert\bu\Vert_2+\sb_3\left\Vert\bfV\right\Vert_\Pi\calQ(\bu).
\end{align}
\item[\rm (iii)] Given positive numbers
$\sc_1$, $\sc_2$ and $\sc_3$, we say that $\frX$ satisfies $\ATP_{\calR,\calQ}(\sc_1;\sc_2;\sc_3)$ on $\mbC'$ if
\vspace{-3pt}
\begin{align}
\forall[\bfV;\bu]\in\mbC',\quad\|\frX^{(n)}(\bfV)+\bu\|_2\ge 
\sc_1\Vert[\bfV;\bu]\Vert_\Pi-\sc_2\calR(\bfV)-\sc_3\calQ(\bu).
\end{align}

\item[\rm (iv)] Given positive numbers
$\sf_1$, $\sf_2$ and $\sf_3$, we will say that $(\frX,\bxi)$ satisfies $\MP_{\calR,\calQ}(\sf_1;\sf_2;\sf_3)$ if 
\vspace{-3pt}
\begin{align}
\forall[\bfV;\bu]\in\mdR^p,\quad|\langle\bxi^{(n)},\frX^{(n)}(\bfV)+\bu\rangle|\le 
\sf_1\Vert[\bfV;\bu]\Vert_\Pi
+\sf_2\calR(\bfV)
+\sf_3\calQ(\bu).
\end{align}
\end{itemize}

If either $\mbC=\mdR^p$ or $\mbC'=\mdR^p\times\re^n$, we omit the reference to the set where the above properties hold. 
\end{definition}
\end{tcolorbox}

$\TP$ is essentially ``restricted strong convexity'' \cite{2012loh:wainwright-AoS}, a well known fundamental property in high-dimensional estimation. Indeed, $\TP_\calR(\sa_1;0)$ on 
$\mbC$ is strong convexity of $\frX^{(n)}$ on
$\mbC$ with respect to the pseudo-norm $\Vert\cdot\Vert_\Pi$.\footnote{To be precise, $\sa_1$ is an absolute constant for general classes of designs. The usual notion of restricted eigenvalue, as e.g. in \cite{BRT,2018bellec:lecue:tsybakov}, is with respect to the Frobenius norm. In that case, the eigenvalue constant represents a ``condition number''. In this sense, it is more precise to say that $\TP$ is a relaxation of Bernstein's condition.} 
$\MP(\sf_1;\sf_2;0)$ implies a bound on the ``multiplier process'' $\bfV\mapsto\frac{1}{n}\sum_{i\in[n]}\bxi_i\llangle\bfX_i,\bfV\rrangle$, also an essential property used in high-dimensional estimation. In this work, we give improved bounds on the MP property regarding the confidence level.

\begin{remark}
We make use of the term ``transfer principle'' or ``reduction principle'' making homage to the terms coined in \cite{2013rudelson:zhou, 2013oliveira, 2016oliveira}. 
\end{remark}

The ``augmented'' notions of $\IP$, $\ATP$ and $\MP$ 
will be useful for robust linear regression. From the next deterministic lemma, 
$\ATP$ is a consequence of $\TP$ and $\IP$. We will show in Section \ref{ss:subgaussian:designs} that $\TP$, $\IP$ and $\MP$ are satisfied by subgaussian designs with high probability.

\begin{lemma}[$\TP+\IP\Rightarrow\ATP$, Lemma 7 in \cite{2019dalalyan:thompson}]\label{lemma:ATP}
Let $\calR$ be a norm over $\mdR^p$ and $\calQ$ a norm over $\re^n$ with unit ball $\mathbb B_\calQ^n$. Suppose $\frX:\mdR^p\rightarrow\re^{n}$ satisfies  
$\TP_{\calR}(\sa_1;\sa_2)$
and $\IP_{\calR,\calQ}(\sb_1;\sb_2;\sb_3)$ for some positive numbers $\sa_1$, $\sa_2$,
$\sb_1$, $\sb_2$ and $\sb_3$.  Then, for any 
$\alpha>0$,
$\frX$ satisfies the $\ATP_{\calR,\calQ}(\sc_1;\sc_2;\sc_3)$ with constants
$\sc_1=\{\sa_1^2- \sb_1 -\alpha^2\}^{1/2}$, $\sc_2=\sa_2+\sb_2/\alpha$
and $\sc_3=\sb_3/\alpha$. Taking $\alpha = \sa_1/2$, we obtain that
$\ATP_{\calR,\calQ}(\sc_1;\sc_2;\sc_3)$ holds with constants
$\sc_1=\{(3/4)\sa_1^2- \sb_1\}^{1/2}$, $\sc_2=\sa_2+2\sb_2/\sa_1$ and $\sc_3=2\sb_3/\sa_1$.
\end{lemma}

We now recall the definition of decomposable norms \cite{2012negahban:ravikumar:wainwright:yu, 2011koltchinskii:lounici:tsybakov}.

\begin{definition}[Decomposable norm]\label{def:decomposable:norm}
A norm $\calR$ over $\mdR^{p}$ is said to be decomposable if for all $\bfB\in\mdR^{p}$, there exist linear map $\bfV\mapsto\calP_{\bfB}^\perp(\bfV)$ such that, for all $\bfV\in\mdR^{p}$, defining 
$\calP_{\bfB}(\bfV):=\bfV-\calP^\perp_{\bfB}(\bfV)$,
\begin{itemize}
\item $\calP_{\bfB}^\perp(\bfB)=0$,
\item $\llangle\calP_{\bfB}(\bfV),\calP^\perp_{\bfB}(\bfV)\rrangle=0$,
\item $\calR(\bfV)=\calR(\calP_{\bfB}(\bfV))+\calR(\calP^\perp_{\bfB}(\bfV))$.
\end{itemize}
In particular, $\Vert\bfV\Vert_F^2=\Vert\calP_{\bfB}(\bfV)\Vert_F^2+\Vert\calP_{\bfB}^\perp(\bfV)\Vert_F^2$.
For $\bfV\in\mdR^{p}\setminus\{0\}$, we define
$$
\Psi_\calR(\bfV):=\frac{\calR(\bfV)}{\Vert\bfV\Vert_F}.
$$ 
We omit the subscript $\calR$ when it is clear in the context. 
\end{definition}

Two well known examples of decomposable norms are the $\ell_1$ and nuclear norms.
\begin{example}[$\ell_1$-norm]
Given $\bfB\in\mdR^{p}$ with \emph{sparsity support} 
$\mathscr S(\bfB):=\{[j,k]:\bfB_{j,k}\neq0\}$, the $\ell_1$-norm in $\mdR^{p}$ satisfies the above decomposability condition with the map
$
\bfV\mapsto\calP^\perp_{\bfB}(\bfV):=\bfV_{\calS(\bfB)^c}
$
where $\bfV_{\calS(\bfB)^c}$ denotes the $d_1\times d_2$ matrix whose entries are zero at indexes in $\mathscr S(\bfB)$. 
\end{example}
\begin{example}[Nuclear norm]
Let $\bfB\in\mdR^{p}$ with rank $r:=\rank(\bfB)$, singular values $\{\sigma_j\}_{j\in[r]}$ and singular vector decomposition $\bfB=\sum_{j\in[r]}\sigma_j\bu_j\bv_j^\top$. Here  $\{\bu_j\}_{j\in[r]}$ are the left singular vectors spanning the subspace $\calU$ and $\{\bv_j\}_{j\in[r]}$ are the right singular vectors spanning the subspace $\calV$. The pair $(\calU,\calV)$ is sometimes referred as the \emph{low-rank support} of $\bfB$. Given subspace $S\subset\re^\ell$ let  
$\bfP_{S^\perp}$ denote the matrix defining the orthogonal projection onto $S^\perp$. Then, the map
$
\bfV\mapsto\calP^\perp_{\bfB}(\bfV):=\bfP_{\calU^\perp}\bfV\bfP_{\calV^\perp}^\top
$
satisfy the decomposability condition for the nuclear norm $\Vert\cdot\Vert_N$. 
\end{example}

Decomposability is particularly useful because of the following well known lemmas \cite{2012negahban:ravikumar:wainwright:yu,   2018bellec:lecue:tsybakov}.
\begin{lemma}\label{lemma:A1:B} 
Let $\calR$ be a decomposable norm over $\mdR^p$. Let $\bfB,\hat\bfB\in\mdR^{p}$ and 
$\bfV:=\hat\bfB-\bfB$. Then, for any $\nu\in[0,1]$, 
\begin{align}
\nu\calR(\bfV)+\calR(\bfB) - \calR(\hat\bfB)\le
(1+\nu)\calR(\calP_\bfB(\bfV)) -(1-\nu)\calR(\calP_\bfB^\perp(\bfV)).
\end{align}
\end{lemma}
\begin{lemma}\label{lemma:A1}
Let $o\in[n]$, $\btheta,\hat\btheta\in\re^n$ such that $\Vert\btheta\Vert_0\le o$. Set 
$\bu:=\hat\btheta-\btheta$. Then
$
\|\btheta\|_\sharp-\|\hat\btheta\|_\sharp
\le\sum_{i=1}^o\omega_i\bu_i^\sharp-\sum_{i=o+1}^n\omega_i\bu_i^\sharp.
$
In particular, for any $\nu\in[0,1]$, 
\begin{align}
\nu\|\bu\|_\sharp+\|\btheta\|_\sharp - \|\hat\btheta\|_\sharp \le
(1+\nu)\sum_{i=1}^o\omega_i\bu_i^\sharp -(1-\nu)\sum_{i=o+1}^n\omega_i\bu_i^\sharp.
\end{align}
\end{lemma}

Finally, we state some cone definitions and recall the definition of restricted eigenvalue constant. 

\begin{tcolorbox}
\begin{definition}\label{def:dim:red:cones}
Let $\calR$ be a decomposable norm over $\mdR^p$. Given $\bfB\in\mdR^p$ and $c_0,\gamma,\eta>0$, we define the following cones
\begin{align}
\calC_{\bfB,\calR}(c_0)&:=\left\{\bfV:
\calR(\calP_{\bfB}^\perp(\bfV))\le c_0\calR(\calP_{\bfB}(\bfV))\right\},\\
\calC_{\bfB,\calR}(c_0,\gamma,\eta)&:=\left\{[\bfV;\bu]: \gamma\calR(\calP_{\bfB}^\perp(\bfV))
+\sum_{i=o+1}^n\omega_i\bu_i^\sharp\le c_0\left[\gamma\calR(\calP_{\bfB}(\bfV))+\eta\Vert\bu\Vert_2\right]\right\}.
\end{align}
We will omit the subscript $\calR$ when the norm is clear in the context. 
\end{definition}
\end{tcolorbox}
 
\begin{definition}[Restricted eigenvalue]\label{def:RE}
Given a convex cone $\mbC$ on $\mdR^p$, we define
$$
\mu(\mbC):=\sup_{\bfV\in\mbC}\frac{\Vert\bfV\Vert_F}{\Vert\bfV\Vert_\Pi}.
$$
\end{definition}

\subsection{Deterministic bounds}\label{ss:deterministic:bounds}

Throughout this section $\calR$ is a general decomposable norm on $\mdR^p$ (see Definition \ref{def:decomposable:norm}).

\begin{lemma}[Dimension reduction]\label{lemma:dim:reduction}
Suppose 
\begin{itemize}
\item[\rm (i)] $(\frX,\bxi)$ satisfies the $\MP_{\calR,\Vert\cdot\Vert_\sharp}(\sf_1;\sf_2;\sf_3)$ 
for some positive numbers $\sf_1$, $\sf_2$ and $\sf_3$. 
\item[\rm{(ii)}] $\frX$ satisfies the $\ATP_{\calR,\Vert\cdot\Vert_\sharp}(\sc_1;\sc_2;\sc_3)$ for some positive numbers $\sc_1,\sc_2,\sc_3$. 
\item[\rm (iii)] 
$
\lambda = \gamma\tau
\ge2[\sf_2+(\nicefrac{\sf_1\sc_2}{\sc_1})],\quad\text{and}\quad
\tau \ge 2[\sf_3+(\nicefrac{\sf_1\sc_3}{\sc_1})].
$
\end{itemize}
Then either $[\bfDelta^{\hat\bfB};\bfDelta^{\hat\btheta}]\in\calC_{\bfB^*}(6,\gamma,\Omega)$ or
\begin{align}
\Vert[\bfDelta^{\hat\bfB};\bfDelta^{\hat\btheta}]\Vert_\Pi&\le 
2\frac{\sf_1}{\sc_1^2}+\bigg(\frac{\sc_2}{\lambda}\bigvee \frac{\sc_3}{\tau}\bigg)\frac{28\sf_1^2}{3\sc_1^3},\label{lemma:dim:reduction:rate:l2}\\
\lambda\calR(\bfDelta^{\hat\bfB})+\tau\|\bfDelta^{\hat\btheta}\|_\sharp&\le \frac{28\sf_1^2}{3\sc_1^2}.\label{lemma:dim:reduction:rate:l1}
\end{align}
\end{lemma}
\begin{proof}
The first order condition of \eqref{equation:aug:slope:rob:estimator} at 
$[\hat\bfB,\hat\btheta]$ is equivalent to the statement: there exist 
$\bfV\in\partial\calR(\hat\bfB)$ and $\bu\in\partial\Vert\hat\btheta\Vert_\sharp$ such that for all $[\bfB;\btheta]$,
\begin{align}
\sum_{i\in[n]}\left[y_i^{(n)}-\frX^{(n)}_i(\widehat\bfB)-\hat\btheta_i\right]\llangle\bfX^{(n)}_i,\hat\bfB-\bfB\rrangle&\ge\lambda\llangle\bfV,\hat\bfB-\bfB\rrangle,\\
\langle\by^{(n)}-\frX^{(n)}(\hat\bfB)-\hat\btheta,\hat\btheta-\btheta\rangle&\ge\tau\langle\bu,\hat\btheta-\btheta\rangle.\label{equation:first:order:condition}
\end{align}
Evaluating at $[\bfB^*;\btheta^*]$ and using that 
$
\by^{(n)}=\frX^{(n)}(\bfB^*)+\btheta^*+\bxi^{(n)}
$
we obtain
\begin{align}
\sum_{i\in[n]}\left[\frX^{(n)}_i\left(\bfDelta^{\hat\bfB}\right)+\bfDelta_i^{\hat\btheta}\right]\llangle\bfX^{(n)}_i,\bfDelta^{\hat\bfB}\rrangle&\le\sum_{i\in[n]}\xi_i^{(n)}\llangle\bfX_i^{(n)},\bfDelta^{\hat\bfB}\rrangle-\lambda\llangle\bfV,\bfDelta^{\hat\bfB}\rrangle,\\
\left\langle\frX^{(n)}\left(\bfDelta^{\hat\bfB}\right)+\bfDelta^{\hat\btheta},\bfDelta^{\hat\btheta}\right\rangle
&\le\left\langle\bxi^{(n)}-\tau\bu,\bfDelta^{\hat\btheta}\right\rangle,
\end{align}
so summing both equations we get
\begin{align}
\Vert\frM^{(n)}(\bfDelta^{\hat\bfB},\bfDelta^{\hat\btheta})\Vert_2^2\le
\llangle\bxi^{(n)},\frM^{(n)}(\bfDelta^{\hat\bfB},\bfDelta^{\hat\btheta})\rrangle
-\lambda\llangle\bfV,\bfDelta^{\hat\bfB}\rrangle
-\tau\langle\bu,\bfDelta^{\hat\btheta}\rangle.
\end{align}

There is $\bfU$ such that 
$\calR^*(\bfU)\le 1$ and $\llangle\bfU,\hat\bfB\rrangle = \calR(\hat\bfB)$.\footnote{Recall the subdifferential of a norm $\calR$ at a point $\bfW$ is $\partial\calR(\bfW)=\{\bfU:\calR^*(\bfU)\le 1,\llangle\bfU,\bfW\rrangle = \calR(\bfW)\}$.} Hence, we get 
$$
-\llangle\bfDelta^{\hat\bfB},\bfV\rrangle = \llangle\bfB^*-\hat\bfB,\bfV\rrangle =
\llangle\bfB^*,\bfV\rrangle -\calR(\hat\bfB)\le 
\calR(\bfB^*)-\calR(\hat\bfB).
$$
Similarly, $-\langle\bfDelta^{\hat\btheta},\bu\rangle \le \|\btheta^*\|_\sharp-\|\hat\btheta\|_\sharp$.
From these bounds we obtain
\begin{align}
\Vert\frM^{(n)}(\bfDelta^{\hat\bfB},\bfDelta^{\hat\btheta})\Vert_2^2
 &\stackrel{\rm (i)}{\le} \sf_1\Vert[\bfDelta^{\hat\bfB};\bfDelta^{\hat\btheta}]\Vert_{\Pi}
+\sf_2\calR(\bfDelta^{\hat\bfB})+\sf_3\|\bfDelta^{\hat\btheta}\|_\sharp\\
&+ \lambda \big(\calR(\bfB^*) - \calR(\hat\bfB)\big) +  
\tau\big(\|\btheta^*\|_\sharp -\|\hat\btheta\|_\sharp\big)\\
&\stackrel{\rm (ii)}{\le}\frac{\sf_1}{\sc_1}\Vert\frM^{(n)}(\bfDelta^{\hat\bfB},\bfDelta^{\hat\btheta})\Vert_2
+\left(\sf_2+\frac{\sf_1\sc_2}{\sc_1}\right)\calR(\bfDelta^{\hat\bfB})+\left(\sf_3+\frac{\sf_1\sc_3}{\sc_1}\right)\|\bfDelta^{\hat\btheta}\|_\sharp\\
&+ \lambda \big(\calR(\bfB^*) - \calR(\hat\bfB)\big) +  
\tau\big(\|\btheta^*\|_\sharp -\|\hat\btheta\|_\sharp\big)\\
&\stackrel{\rm (iii)}{\le}\frac{\sf_1}{\sc_1}\Vert\frM^{(n)}(\bfDelta^{\hat\bfB},\bfDelta^{\hat\btheta})\Vert_2
+(\nicefrac{\lambda}{2})\calR(\bfDelta^{\hat\bfB})+(\nicefrac{\tau}{2})\|\bfDelta^{\hat\btheta}\|_\sharp\\
&+ \lambda \big(\calR(\bfB^*) - \calR(\hat\bfB)\big) +  \tau \big(\|\btheta^*\|_\sharp -\|\hat\btheta\|_\sharp\big)\\
&\le\frac{\sf_1}{\sc_1}\Vert\frM^{(n)}(\bfDelta^{\hat\bfB},\bfDelta^{\hat\btheta})\Vert_2+\triangle,
\label{lemma:dim:reduction:eq1}
\end{align}
where in last inequality we used  the decomposability of $\calR$ and Lemmas \ref{lemma:A1:B}-\ref{lemma:A1} with $\nu:=1/2$ and have defined 
\begin{align}
\triangle:=(\nicefrac{3\lambda}{2})(\calR\circ\calP_{\bfB^*})(\bfDelta^{\hat\bfB}) 
-(\nicefrac{\lambda}{2})(\calR\circ\calP_{\bfB^*}^\perp)(\bfDelta^{\hat\bfB})
+(\nicefrac{3\tau\Omega}{2})\Vert\bfDelta^{\hat\btheta}\Vert_2 -(\nicefrac{\tau}{2})\sum_{i=o+1}^n\omega_i(\bfDelta^{\hat\btheta})_i^\sharp.
\end{align}

Define also $G:=\Vert\frM^{(n)}(\bfDelta^{\hat\bfB},\bfDelta^{\hat\btheta})\Vert_2$ and $H:=(\nicefrac{3\lambda}{2})(\calR\circ\calP_{\bfB^*})(\bfDelta^{\hat\bfB})+(\nicefrac{3\tau\Omega}{2})\Vert\bfDelta^{\hat\btheta}\Vert_2$. We split the argument in two cases. 
\begin{description}
\item[Case 1:] $\frac{\sf_1}{\sc_1}G\le H$. In that case, from \eqref{lemma:dim:reduction:eq1} we obtain 
$[\bfDelta^{\hat\bfB};\bfDelta^{\hat\btheta}]\in\calC_{\bfB^*}(6,\gamma,\Omega)$.

\item[Case 2:] $\frac{\sf_1}{\sc_1}G\ge H$. In that case we obtain $G^2\le\frac{2\sf_1}{\sc_1}G\Rightarrow G\le\frac{2\sf_1}{\sc_1}$. Therefore $H\le2\frac{\sf_1^2}{\sc_1^2}$. We first establish a bound on $\gamma\calR(\bfDelta^{\hat\bfB})+\|\bfDelta^{\hat\btheta}\|_\sharp$. Again from \eqref{lemma:dim:reduction:eq1}, we obtain that 
\begin{align}
\lambda(\calR\circ\calP_{\bfB^*}^\perp)(\bfDelta^{\hat\bfB})
+\tau\sum_{i=o+1}^n\omega_i(\bfDelta^{\hat\btheta})_i^\sharp
\le4\frac{\sf_1}{\sc_1}G\le8\frac{\sf_1^2}{\sc_1^2}.
\end{align}
This fact and decomposability imply
\begin{align}
\lambda\calR(\bfDelta^{\hat\bfB})+\tau\|\bfDelta^{\hat\btheta}\|_\sharp
\le \frac{2}{3}H+8\frac{\sf_1^2}{\sc_1^2}\le\frac{28\sf_1^2}{3\sc_1^2}.
\end{align}
From the $\ATP$ in $\rm{(ii)}$ and the above bound,
\begin{align}
\sc_1\Vert[\bfDelta^{\hat\bfB};\bfDelta^{\hat\btheta}]\Vert_\Pi
\le\|\frM^{(n)}(\bfDelta^{\hat\bfB},\bfDelta^{\hat\btheta})\|_2+\sc_2\calR(\bfDelta^{\hat\bfB})
+\sc_3\|\bfDelta^{\hat\btheta}\|_\sharp
\le 2\frac{\sf_1}{\sc_1}+\bigg(\frac{\sc_2}{\lambda}\bigvee \frac{\sc_3}{\tau}\bigg)\frac{28\sf_1^2}{3\sc_1^2}.
\end{align}
\end{description}
This finishes the proof. 
\end{proof}

\begin{proposition}\label{prop:suboptimal:rate}
Suppose that, in addition to (i)-(iii) in Lemma \ref{lemma:dim:reduction}, the following condition holds:
\begin{itemize}
\item[\rm{(iv)}] For $R:=\Psi(\calP_{\bfB^*}(\bfDelta^{\hat\bfB}))\mu(\calC_{\bfB^*}(12))$, assume
$$
		14\big({\sc_2}\vee \gamma\sc_3\big)
		\left(R^2+ \frac{4\Omega^2}{\gamma^2}\right)^{1/2} \le \sc_1.
$$
\end{itemize}
Then either \eqref{lemma:dim:reduction:rate:l2}-\eqref{lemma:dim:reduction:rate:l1} hold or 
\begin{align}
	\big\|[\bfDelta^{\hat\bfB};\bfDelta^{\hat\btheta}]\big\|_\Pi
	&\le \frac{4}{\sc_1^2}\sf_1+\frac{6}{\sc_1^2}\sqrt{\lambda^2R^2+4\tau^2\Omega^2}, \label{prop:suboptimal:rate:l2}\\
\lambda\calR(\bfDelta^{\hat\bfB}) + \tau\|\bfDelta^{\hat\btheta}\big\|_\sharp	
	& \le \frac{56}{\sc_1^2}\sf_1^2
	+\frac{56}{\sc_1^2}\lambda^2R^2+ \frac{196}{\sc_1^2}\tau^2\Omega^2.\label{prop:suboptimal:rate:l1}
\end{align} 
\end{proposition}
\begin{proof}
From Lemma \ref{lemma:dim:reduction}, we only need to consider the case when $[\bfDelta^{\hat\bfB};\bfDelta^{\hat\btheta}]\in\calC_{\bfB^*}(6,\gamma,\Omega)$. A slight variation of the argument to establish \eqref{lemma:dim:reduction:eq1} leads to
\begin{align}
\Vert\frM^{(n)}(\bfDelta^{\hat\bfB},\bfDelta^{\hat\btheta})\Vert_2^2
\le\sf_1\big\|[\bfDelta^{\hat\bfB};\bfDelta^{\hat\btheta}]\big\|_\Pi+\triangle,
\end{align}
where $\triangle$ was already defined in the proof of Lemma \ref{lemma:dim:reduction}. $\ATP$ as stated in item $\rm{(ii)}$ of Lemma \ref{lemma:dim:reduction} further leads to
\begin{align}
\left\{\sc_1 \big\|[\bfDelta^{\hat\bfB};\bfDelta^{\hat\btheta}]\big\|_\Pi
	-\sc_2\calR(\bfDelta^{\hat\bfB}) - \sc_3\|\bfDelta^{\hat\btheta}\|_\sharp\right\}_+ \le \sqrt{\sf_1\big\|[\bfDelta^{\hat\bfB};\bfDelta^{\hat\btheta}]\big\|_\Pi+\triangle}.
	\label{prop:suboptimal:rate:eq1}
\end{align}
We now split our arguments in two cases.  
 
\begin{description}
\item[Case 1:] $12\calR(\calP_{\bfB^*}(\bfDelta^{\hat\bfB})) \ge \calR(\calP_{\bfB^*}^\perp(\bfDelta^{\hat\bfB}))$. Hence 
$\bfDelta^{\hat\bfB}\in\calC_{\bfB^*}(12)$. Decomposability of $\calR$ and
$[\bfDelta^{\hat\bfB};\bfDelta^{\hat\btheta}]\in\calC_{\bfB^*}(6,\gamma,\Omega)$ further imply
\begin{align}
\sc_2\calR(\bfDelta^{\hat\bfB}) + \sc_3\|\bfDelta^{\hat\btheta}\|_\sharp &\le
\bigg(\frac{\sc_2}{\lambda}\bigvee \frac{\sc_3}{\tau}\bigg)\big(
\lambda\calR(\bfDelta^{\hat\bfB}) + \tau\|\bfDelta^{\hat\btheta}\|_\sharp)\\
&\le 7\bigg(\frac{\sc_2}{\lambda}\bigvee \frac{\sc_3}{\tau}\bigg)\big(
\lambda\calR(\calP_{\bfB^*}(\bfDelta^{\hat\bfB})) + \tau\Omega
\|\bfDelta^{\hat\btheta}\|_2)\\
&\le 7\bigg(\frac{\sc_2}{\lambda}\bigvee \frac{\sc_3}{\tau}\bigg)
\bigg(\lambda^2R^2+ \tau^2\Omega^2\bigg)^{1/2}
\Vert[\bfDelta^{\hat\bfB};\bfDelta^{\hat\btheta}]\Vert_\Pi.
\label{prop:suboptimal:rate:eq2}
\end{align}
Similarly,  
\begin{align}
	\triangle&\le(\nicefrac{3\lambda}{2})\calR(\calP_{\bfB^*}(\bfDelta^{\hat\bfB})) + (\nicefrac{3\tau\Omega}{2})\|\bfDelta^{\hat\btheta}\|_2\\
	&\le (\nicefrac{3}{2})\bigg(\lambda^2\Psi^2(\calP_{\bfB^*}(\bfDelta^{\hat\bfB}))\mu^2(\calC_{\bfB^*}(12))+ \tau^2\Omega^2\bigg)^{1/2}\Vert[\bfDelta^{\hat\bfB};\bfDelta^{\hat\btheta}]\Vert_\Pi.
\label{prop:suboptimal:rate:eq3}
\end{align}
To ease notation, define $x=\big\|[\bfDelta^{\hat\bfB};\bfDelta^{\hat\btheta}]\big\|_\Pi$, 
\begin{align}
A & = 7\big(\frac{\sc_2}{\lambda}\bigvee \frac{\sc_3}{\tau}\big)
\big(\lambda^2R^2+ \tau^2\Omega^2\big)^{1/2},\\
B &= \sf_1+(\nicefrac{3}{2})\big(\lambda^2R^2+ \tau^2\Omega^2\big)^{1/2}.
\end{align}
From \eqref{prop:suboptimal:rate:eq1}, \eqref{prop:suboptimal:rate:eq2} and \eqref{prop:suboptimal:rate:eq3} we get 
\begin{align}
		\sc_1 x \le A x + \sqrt{Bx}\quad\Longrightarrow
		\quad x \le \frac{B}{(\sc_1 -A)^2}
\end{align}
provided that ${A\le \sc_1}$. Assuming ${2A\le \sc_1}$, we get
\begin{align}
	\big\|[\bfDelta^{\hat\bfB};\bfDelta^{\hat\btheta}]\big\|_\Pi	\le \frac{4B}{\sc_1^2}=\frac{4}{\sc_1^2}\sf_1+\frac{6}{\sc_1^2}(\lambda^2 R^2+\tau^2\Omega^2)^{1/2}.
\end{align}
For deriving the bound on $\gamma\calR(\bfDelta^{\hat\bfB})+\Vert\bfDelta^{\hat\btheta}\Vert_\sharp$, we again use the decomposability of $\calR$ and $[\bfDelta^{\hat\bfB};\bfDelta^{\hat\btheta}]\in\calC_{\bfB^*}(6,\gamma,\Omega)$, obtaining 
\begin{align}
	\lambda\calR(\bfDelta^{\hat\bfB}) + \tau \|\bfDelta^{\hat\btheta}\|_\sharp
	&\le (7\lambda\calR(\calP_{\bfB^*}(\bfDelta^{\hat\bfB})) + 6\tau\Omega\|\bfDelta^{\hat\btheta}\|_2)\\
	&\le 7\bigg(\lambda^2R^2+ \tau^2\Omega^2\bigg)^{1/2}
	\big\|[\bfDelta^{\hat\bfB};\bfDelta^{\hat\btheta}]\big\|_\Pi\\
	&\le \frac{28}{\sc_1^2}\sf_1\bigg(\lambda^2R^2+ \tau^2\Omega^2\bigg)^{1/2}
	+\frac{42}{\sc_1^2}\bigg(\lambda^2R^2+ \tau^2\Omega^2\bigg)
	\\
&\le \frac{56}{\sc_1^2}\sf^2_1+\frac{56}{\sc_1^2}\bigg(\lambda^2R^2+ \tau^2\Omega^2\bigg). 
\end{align}

\item[Case 2:] $12\calR(\calP_{\bfB^*}(\bfDelta^{\hat\bfB})) < \calR(\calP_{\bfB^*}^\perp(\bfDelta^{\hat\bfB}))$. As 
$[\bfDelta^{\hat\bfB};\bfDelta^{\hat\btheta}]\in\calC_{\bfB^*}(6,\gamma,\Omega)$, we get
\begin{align} 
6\gamma\calR(\calP_{\bfB^*}(\bfDelta^{\hat\bfB}))+
\sum_{i=o+1}^n\omega_i(\bfDelta^{\hat\btheta})^\sharp_i\le6\Omega\|\bfDelta^{\hat\btheta}\|_2.
\end{align}
This and decomposability of $\calR$ imply
\begin{align}
\sc_2\calR(\bfDelta^{\hat\bfB}) + \sc_3\|\bfDelta^{\hat\btheta}\|_\sharp &\le
\Big(\frac{\sc_2}{\lambda}\bigvee \frac{\sc_3}{\tau}\Big)\big(
\lambda\calR(\bfDelta^{\hat\bfB}) + \tau\|\bfDelta^{\hat\btheta}\|_\sharp)\\
&\le 7\bigg(\frac{\sc_2}{\lambda}\bigvee \frac{\sc_3}{\tau}\bigg)\big(
\lambda\calR(\calP_{\bfB^*}(\bfDelta^{\hat\bfB})) + \tau\Omega\|\bfDelta^{\hat\btheta}\|_2)\\
&\le 14\bigg(\frac{\sc_2}{\lambda}\bigvee \frac{\sc_3}{\tau}\bigg)\tau\Omega\|\bfDelta^{\hat\btheta}\|_2.\label{prop:suboptimal:rate:eq4}
\end{align}
Similarly,  
\begin{align}
	\triangle\le(\nicefrac{3}{2})\lambda\calR(\calP_{\bfB^*}(\bfDelta^{\hat\bfB}))  +
    (\nicefrac{3}{2})\tau\Omega\|\bfDelta^{\hat\btheta}\|_2
	\le 3\tau\Omega\|\bfDelta^{\hat\btheta}\|_2.
\label{prop:suboptimal:rate:eq5}
\end{align}

Again it is convenient to define $x=\big\|[\bfDelta^{\hat\bfB}\,;\,\bfDelta^{\hat\btheta}]\big\|_\Pi$, $A'=14\big(\frac{\sc_2}{\lambda}\bigvee \frac{\sc_3}{\tau}\big)
\tau\Omega$ and $B' = \sf_1+3\tau\Omega$. From 
\eqref{prop:suboptimal:rate:eq1}, \eqref{prop:suboptimal:rate:eq4} and \eqref{prop:suboptimal:rate:eq5},
\begin{align}
		\sc_1 x \le A' x + \sqrt{B'x}\quad\Longrightarrow
		\quad x\le \frac{B'}{(\sc_1 -A')^2} \le \frac{4B'}{\sc_1^2}
\end{align}
provided that $2A'\le \sc_1$. In conclusion,
\begin{align}
		\big\|[\bfDelta^{\hat\bfB}\,;\,\bfDelta^{\hat\btheta}]\big\|_\Pi \le\frac{4\sf_1}{\sc_1^2}+\frac{12\tau\Omega}{\sc_1^2},
\end{align}
implying
\begin{align}
		\lambda\calR(\bfDelta^{\hat\bfB}) 
		+\tau\big\|\bfDelta^{\hat\btheta}\big\|_\sharp
		&\le 7\lambda\calR(\calP_{\bfB^*}(\bfDelta^{\hat\bfB}))  +7\tau\Omega\big\|\bfDelta^{\hat\btheta}\big\|_2\\ 
		&\le 14\tau\Omega\,\big\|\bfDelta^{\hat\btheta}\big\|_2\\
		&\le \frac{56}{\sc_1^2}\sf_1\tau\Omega+\frac{168}{\sc_1^2}\tau^2\Omega^2\\
		&\le \frac{28}{\sc_1^2}\sf_1^2 + \frac{196}{\sc_1^2}\tau^2\Omega^2.
\end{align}
\end{description}
The proof is complete by noting that the bounds in the statement of the proposition are larger than the bounds we have just established in the above cases.
\end{proof}

\begin{lemma}\label{lemma:recursion:delta:bb:general:norm}
\begin{align}
\left\langle\frM^{(n)}(\bfDelta^{\hat\bfB} ,\bfDelta^{\hat\btheta}),\frX^{(n)}(\bfDelta^{\hat\bfB})\right\rangle\le 
\langle\bxi^{(n)},\frX^{(n)}(\bfDelta^{\hat\bfB})\rangle
+\lambda\left(2\calR(\calP_{\bfB^*}(\bfDelta^{\hat\bfB}))
-\calR(\bfDelta^{\hat\bfB})\right).
\end{align}
\end{lemma}
\begin{proof}
As $[\hat\bfB;\hat\btheta]$ is the minimizer of \eqref{equation:aug:slope:rob:estimator}, in particular
\begin{align}
\hat\bfB\in 
&\argmin_{\bfB}\left\{\frac{1}{2}\Vert\by^{(n)}-\frX^{(n)}(\bfB)-\hat\btheta\Vert_2^2+\lambda\calR(\bfB)\right\}.
\end{align}
The KKT conditions of the above minimization problem and the expression of the subdifferential of any norm imply that there exists $\bfV\in\re^{m_1\times m_2}$ with 
$\calR^*(\bfV)\le1$ and  $\llangle\bfV,\hat\bfB\rrangle=\calR(\hat\bfB)$ such that, for all $\bfB\in\re^{m_1\times m_2}$,
\begin{align}
0&\le\sum_{i\in[n]}\left[\frX^{(n)}_i(\widehat\bfB)+\widehat\btheta_i-y_i^{(n)}\right]\llangle\bfX^{(n)}_i,\bfB-\hat\bfB\rrangle+\lambda\llangle\bfV,\bfB-\hat\bfB\rrangle\\
&=\sum_{i\in[n]}\left[\frX^{(n)}_i(\bfDelta^{\widehat\bfB})+\bfDelta^{\widehat\btheta}_i-\xi_i^{(n)}\right]\llangle\bfX^{(n)}_i,\bfB-\hat\bfB\rrangle+\lambda\llangle\bfV,\bfB-\hat\bfB\rrangle.
\end{align}
We can take $\bfB:=\bfB^*$ above and obtain
\begin{align}
0&\le-\left\langle\frX^{(n)}(\bfDelta^{\hat\bfB})+\bfDelta^{\hat\btheta},\frX^{(n)}(\bfDelta^{\hat\bfB})\right\rangle 
+\langle\frX^{(n)}(\bfDelta^{\hat\bfB}),\bxi^{(n)}\rangle-\lambda\llangle\bfDelta^{\hat\bfB},\bfV\rrangle.
\end{align}
Using that $\llangle\bfV,\widehat\bfB\rrangle=\calR(\hat\bfB)$ and $\llangle\bfV,\bfB^*\rrangle\le\calR(\bfB^*)$ (since $\calR^*(\bfV)\le1$), we obtain that
$
-\llangle\bfDelta^{\hat\bfB},\bfV\rrangle\le\calR(\bfB^*)-\calR(\hat\bfB).
$
Moreover, from the triangle inequality and the decomposability property for the norm $\calR$, one checks that 
\begin{align}
\calR(\bfB^*)-\calR(\hat\bfB)\le\calR(\calP_{\bfB^*}(\bfDelta^{\hat\bfB}))-\calR(\calP_{\bfB^*}^\perp(\bfDelta^{\hat\bfB}))
=2\calR(\calP_{\bfB^*}(\bfDelta^{\hat\bfB}))-\calR(\bfDelta^{\hat\bfB}).
\end{align}
Combining the two previous displays finishes the proof. 
\end{proof}

\begin{theorem}[Trace regression]\label{thm:improved:rate}
Suppose the following condition holds: 
\begin{itemize}\itemsep=0pt
\item[\rm (i)] $(\frX,\bxi)$ satisfies the $\MP_{\calR,\Vert\cdot\Vert_\sharp}(\sf_1;\sf_2;\sf_3)$ 
for some positive numbers $\sf_1$, $\sf_2$ and $\sf_3$. 
\item[\rm{(ii)}] $\frX$ satisfies the $\ATP_{\calR,\Vert\cdot\Vert_\sharp}(\sc_1;\sc_2;\sc_3)$ for some positive numbers $\sc_1,\sc_2,\sc_3$. 
\item[\rm (iii)] 
$
\lambda = \gamma\tau
\ge2[\sf_2+(\nicefrac{\sf_1\sc_2}{\sc_1})],\quad\text{and}\quad
\tau \ge 2[\sf_3+(\nicefrac{\sf_1\sc_3}{\sc_1})].
$
\item[\rm(iv)] $\frX$ satisfies the $\IP_{\calR,\Vert\cdot\Vert_\sharp}\left(\sb_1;\sb_2;\sb_3\right)$.
\item[\rm(v)] $\frX$ satisfies the $\TP_{\calR}\left(\sa_1;\sa_2\right)$.
\end{itemize}
Suppose that $\mu(\calC_{\bfB^*}(12))<\infty$. Let $R:=\Psi(\calP_{\bfB^*}(\bfDelta^{\hat\bfB}))\mu(\calC_{\bfB^*}(12))$ and suppose further that
\begin{align}
		14\big({\sc_2}\vee \gamma\sc_3\big)
		\left[R^2+ \frac{4\Omega^2}{\gamma^2}\right]^{1/2} &\le \sc_1,\label{cond1:general:norm}\\		
		\frac{4}{\sc_1^2}\sf_1+\frac{6}{\sc_1^2}\sqrt{\lambda^2R^2+4\tau^2\Omega^2}\le \frac{\lambda}{6\sb_2}.  \label{cond2:general:norm}
\end{align}
Define the quantities
$
\square:=\frac{1.5\sa_2}{\lambda\sa_1}\bigvee \frac{1.5\sb_3}{\tau\sa_1^2},
$
$
\triangle:=\frac{\sc_2}{\lambda}\bigvee \frac{\sc_3}{\tau}
$
and
\begin{align}
\Phi_{\square,\lambda R}&:=\left(\frac{56\square}{\sc_1^2}+\frac{4.5}{\sc_1^2}\right)(\lambda R)^2+\frac{1.5}{\sa_1^2}(\lambda R),\\
\Phi_{\square,\sb_1,\sf_1}&:=\frac{56\square}{\sc_1^2}\sf_1^2+\frac{6}{\sc_1^2}\sb_1\sf_1+\frac{4.5\sb_1^2}{\sc_1^2}+\frac{1.5\sf_1}{\sa_1^2}\\
\omega_{\square,\tau\Omega}&:=\left(\frac{196\square}{\sc_1^2}+\frac{18}{\sc_1^2}\right)(\tau\Omega)^2.
\end{align}

Then 
\begin{align}
		\big\|\bfDelta^{\hat\bfB}\big\|_\Pi &\le\left\{\Phi_{\square,\lambda R}+\Phi_{\square,\sb_1,\sf_1}+\omega_{\square,\tau\Omega}\right\}
		\bigvee\left\{2\frac{\sf_1}{\sc_1^2}+\triangle\frac{28\sf_1^2}{3\sc_1^3}\right\}.
\label{prop:improved:rate:sample:size:general:norm}
\end{align}
\end{theorem}
\begin{proof}        
Condition \eqref{cond1:general:norm} and items $\rm{(i)}$-$\rm{(iii)}$ imply that the claims of Proposition \ref{prop:suboptimal:rate} hold. If \eqref{lemma:dim:reduction:rate:l2}-\eqref{lemma:dim:reduction:rate:l1} hold we have nothing to prove. Otherwise, \eqref{prop:suboptimal:rate:l2}-\eqref{prop:suboptimal:rate:l1} hold. In particular,
\begin{align}
\|\bfDelta^{\hat\btheta}\|_2
		\le \frac{4}{\sc_1^2}\sf_1+\frac{6}{\sc_1^2}\sqrt{\lambda^2R^2+4\tau^2\Omega^2}\stackrel{\eqref{cond2:general:norm}}{\le} \frac{\lambda}{6\sb_2}.\label{cond3:general:norm}
\end{align} 

By Lemma \ref{lemma:recursion:delta:bb:general:norm} and $\MP$ as stated in item $\rm{(i)}$,
\begin{align}
\Vert\frX^{(n)}(\bfDelta^{\hat\bfB})\Vert_2^2
&\le -\langle\frX^{(n)}(\bfDelta^{\hat\bfB}),\bfDelta^{\hat\btheta}\rangle
+\sf_1\big\|\bfDelta^{\hat\bfB}\big\|_\Pi+\sf_2\calR(\bfDelta^{\hat\bfB})
+\lambda\left(2\calR(\calP_{\bfB^*}(\bfDelta^{\hat\bfB}))
-\calR(\bfDelta^{\hat\bfB})\right)\\
&\stackrel{(\rm iii)}{\le} -\langle\frX^{(n)}(\bfDelta^{\hat\bfB}),\bfDelta^{\hat\btheta}\rangle+\sf_1\big\|\bfDelta^{\hat\bfB}\big\|_\Pi+\frac{\lambda}{2}\calR(\bfDelta^{\hat\bfB})
+\lambda\left(2\calR(\calP_{\bfB^*}(\bfDelta^{\hat\bfB}))
-\calR(\bfDelta^{\hat\bfB})\right)\\
&\stackrel{(\rm iv)}{\le}
\sb_1\big\|\bfDelta^{\hat\bfB}\big\|_\Pi
\Vert\bfDelta^{\hat\btheta}\Vert_2+\sb_3\big\Vert\bfDelta^{\hat\bfB}\big\Vert_\Pi\Vert\bfDelta^{\hat\btheta}\Vert_\sharp+\sf_1\big\|\bfDelta^{\hat\bfB}\big\|_\Pi
+2\lambda\calR(\calP_{\bfB^*}(\bfDelta^{\hat\bfB}))-\frac{\lambda}{3}\calR(\bfDelta^{\hat\bfB})\\
&\quad+\sb_2\calR(\bfDelta^{\hat\bfB})\Vert\bfDelta^{\hat\btheta}\Vert_2 - \frac{\lambda}{6}\calR(\bfDelta^{\hat\bfB})\\
&\stackrel{\eqref{cond3:general:norm}}{\le}
\sb_1\big\|\bfDelta^{\hat\bfB}\big\|_\Pi
\Vert\bfDelta^{\hat\btheta}\Vert_2+\sb_3\big\Vert\bfDelta^{\hat\bfB}\big\Vert_\Pi\Vert\bfDelta^{\hat\btheta}\Vert_\sharp+\sf_1\big\|\bfDelta^{\hat\bfB}\big\|_\Pi
+2\lambda\calR(\calP_{\bfB^*}(\bfDelta^{\hat\bfB}))-\frac{\lambda}{3}\calR(\bfDelta^{\hat\bfB}).
\label{prop:improved:rate:eq1:general:norm}
\end{align}

We now define the local variables $x: = \Vert\bfDelta^{\hat\bfB}\Vert_\Pi$ and
\begin{align}
A &:=\sb_1 \Vert\bfDelta^{\hat\btheta}\Vert_2+\sb_3
\Vert\bfDelta^{\hat\btheta}\Vert_\sharp+\sf_1,\\
B &:= \left\{2\lambda\calR(\calP_{\bfB^*}(\bfDelta^{\hat\bfB}))-\frac{\lambda}{3}\calR(\bfDelta^{\hat\bfB})\right\}_+.
\end{align}
On the one hand, combining the last inequality and the 
$\TP$, as stated in item (v), we
arrive at
\begin{align}
    (\sa_1 x -\sa_2\calR(\bfDelta^{\hat\bfB}))_+^2\le A x+ B.
\end{align}
This implies that either $x\le (\sa_2/\sa_1) \calR(\bfDelta^{\hat\bfB})$ or
\begin{align}
		\Big(\sa_1 x -\sa_2\calR(\bfDelta^{\hat\bfB})-\frac{A}{2\sa_1}\Big)^2\le
		B + \frac{A^2}{4\sa_1^2} + \frac{A\sa_2 }{\sa_1}\,\calR(\bfDelta^{\hat\bfB}).
\end{align}
In both cases,
\begin{align}
		x &\le  \frac{\sa_2}{\sa_1}\calR(\bfDelta^{\hat\bfB}) + \frac{A}{2\sa_1^2} +
		\frac1{\sa_1}\Big\{B + \frac{A^2}{4\sa_1^2} + \frac{A\sa_2 }{\sa_1}\,
		\calR(\bfDelta^{\hat\bfB})\Big\}^{1/2}\\
		&\le  1.5\frac{\sa_2}{\sa_1}\calR(\bfDelta^{\hat\bfB}) + 1.5\frac{A}{\sa_1^2} +
		\frac{B^{1/2}}{\sa_1}.\label{ineq:22:general:norm}
\end{align}
On the other hand,
\begin{align}
	B\le 2\lambda\calR(\calP_{\bfB^*}(\bfDelta^{\hat\bfB}))
	\le 2\lambda R x\le \bigg(\frac{\sa_1 x}{2} +
	\frac{2\lambda R}{\sa_1}\bigg)^2.
	\label{ineq:23:general:norm}
\end{align}
Combining \eqref{ineq:22:general:norm} and \eqref{ineq:23:general:norm}, we get
\begin{align}
		\frac{x}{2}\le  \frac{1.5\sa_2}{\sa_1}\calR(\bfDelta^{\hat\bfB}) + \frac{1.5A}{\sa_1^2} +
	 \frac{2\lambda R}{\sa_1^2}.\label{ineq:24:general:norm}
\end{align}
Replacing $A$ and $x$ by their expressions, we arrive at
\begin{align}
		\frac{1}{2}\big\|\bfDelta^{\hat\bfB}\big\|_\Pi & \le
		\frac{1.5\sa_2}{\sa_1}\calR(\bfDelta^{\hat\bfB}) +
		\frac{1.5\sb_1 \Vert\bfDelta^{\hat\btheta}\Vert_2+1.5\sb_3\Vert\bfDelta^{\hat\btheta}\Vert_\sharp+1.5\sf_1}{\sa_1^2}
		+ \frac{1.5\lambda R}{\sa_1^2}\\
		&\le
		\bigg(\frac{1.5\sa_2}{\lambda\sa_1}\bigvee \frac{1.5\sb_3}{\tau\sa_1^2}\bigg)
		\big(\lambda\calR(\bfDelta^{\hat\bfB}) + \tau\Vert\bfDelta^{\hat\btheta}\Vert_\sharp\big)
		+1.5\sb_1\Vert\bfDelta^{\hat\btheta}\Vert_2
		+ \frac{1.5\lambda R}{\sa_1^2}+\frac{1.5\sf_1}{\sa_1^2}.\label{ineq:25:general:norm}
\end{align}

By \eqref{cond3:general:norm},
\begin{align}
1.5\sb_1\Vert\bfDelta^{\hat\btheta}\Vert_2\le 6\frac{\sb_1\sf_1}{\sc_1^2}
+9\frac{\sb_1}{\sc_1^2}\sqrt{\lambda^2R^2+4\tau^2\Omega^2}
\le 6\frac{\sb_1\sf_1}{\sc_1^2}
+4.5\frac{\sb_1^2}{\sc_1^2}+4.5\frac{\lambda^2R^2+4\tau^2\Omega^2}{\sc_1^2}.
\end{align}
The two previous inequalities and \eqref{prop:suboptimal:rate:l1} lead to the claimed rate on $\Vert\bfDelta^{\hat\bfB}\Vert_\Pi$.
\end{proof}


\subsection{Properties for subgaussian $(\bfX,\bxi)$}\label{ss:subgaussian:designs}
In this section we prove that all properties of Definition \ref{def:design:property} are satisfied with high-probability.
\begin{tcolorbox}
Throughout this section, we additionally assume $(\bfX,\xi)\in\mdR^p\times\re$ is a centered (not necessarily independent) random pair satisfying Assumption \ref{assump:distribution:subgaussian} and $\{(\bfX_i,\xi_i)\}_{i\in[n]}$ is an iid copy of $(\bfX,\xi)$. Moreover, $\calR$ is any norm on $\mdR^p$ and $\calQ$ is any  norm on
$\re^n$.
\end{tcolorbox}

The proof that $L$-subgaussian designs satisfy $\TP$ in Definition \ref{def:design:property} will follow from a concentration result for the quadratic process due to Dirksen \cite{2015dirksen} and Bednorz \cite{2014bednorz} and, in addition, a peeling argument.

\begin{theorem}[Theorem 5.5 in Dirksen \cite{2015dirksen}, Theorem 1 in Bednorz \cite{2014bednorz} ]\label{thm:bednorz}
Let $V$ be a compact subset of $\mbB_\Pi$.
  
Then, for universal constant $C>0$, for any $n\ge1$ and $t>0$, with probability at least $1-2e^{-t}$, 
$$
\sup_{\bfV\in V}\left|\Vert\frX(\bfV)\Vert_2^2-n\Vert\bfV\Vert_\Pi^2\right| \le  C^2\left[\mathscr{G}^2\big(\frS^{1/2}(V))+ L\sqrt{n}(\mathscr{G}\big(\frS^{1/2}(V)\big) + L^2 \max(t,\sqrt{ nt})\right].
$$
\end{theorem}

\begin{proposition}[$\TP$]\label{prop:gen:TP}
Grant the assumptions of Theorem \ref{thm:bednorz}.

Then, for the universal constant $C>0$ stated in Theorem \ref{thm:bednorz}, for all $\epsilon\in(0,1)$, $\delta\in(0,1]$ and $n\in\mathbb{N}$, with probability at least $1-\delta$, the following property holds: for all $\bfV\in\mdR^p$,
\begin{align}
\Vert\frX^{(n)}(\bfV)\Vert_2&\ge
\left\{
\left(1+\frac{C^2L^2\epsilon}{4}\right)^{1/2}
-CL\sqrt{\epsilon}
-\frac{CL}{\sqrt{\epsilon n}}\left[3+\sqrt{\log(18/\delta)}\right]
\right\}\Vert\bfV\Vert_\Pi\\
&-1.2C\frac{\mathscr
G\left(\calR(\bfV)\frS^{1/2}(\mbB_\calR)\cap\Vert\bfV\Vert_\Pi\mbB_F\right)}{\sqrt{\epsilon n}}.
\end{align}
In addition, for all $\epsilon\in(0,1)$, 
$\delta\in(0,1]$ and $n\in\mathbb{N}$, with probability at least $1-\delta$, the following property holds: for all
$\bfV\in\re^{m_1\times m_2}$,
\begin{align}
\Vert\frX^{(n)}(\bfV)\Vert_2&\le
\left\{
\left(1-\frac{C^2L^2\epsilon}{4}\right)_+^{1/2}
+CL\sqrt{\epsilon}
+\frac{CL}{\sqrt{\epsilon n}}\left[3+\sqrt{\log(18/\delta)}\right]
\right\}\Vert\bfV\Vert_\Pi\\
&+1.2C\frac{\mathscr
G\left(\calR(\bfV)\frS^{1/2}(\mbB_\calR)\cap\Vert\bfV\Vert_\Pi\mbB_F\right)}{\sqrt{\epsilon n}}.
\end{align}
\end{proposition}
\begin{proof}
Let $R_1>0$ and define the set
\begin{align}
V_1&:=\{\bfV\in\mdR^p:\Vert\bfV\Vert_\Pi=	1,\calR(\bfV)\le R_1\}.
\end{align}
Note that, 
$
\mathscr{G}(\frS^{1/2}(V_1))\le R_1\mathscr G(\frS^{1/2}(\mbB_\calR)\cap R_1^{-1}\mbB_F).
$
Define for convenience the function
$
f(r):=r\mathscr G(\frS^{1/2}(\mbB_\calR)\cap r^{-1}\mbB_F).
$
By Theorem \ref{thm:bednorz}, there is universal $C>0$ such that, for any $R_1>0$ and 
$t\ge0$, with probability at least $1-2e^{-t}$, \begin{align}
1-\inf_{\bfV\in \calB}\Vert\frX^{(n)}(\bfV)\Vert_2^2&\le C^2\left[\frac{f^2(R_1)}{n}+ L\frac{f(R_1)}{\sqrt{n}} + L^2 \max\left(\frac{t}{n},\sqrt{\frac{t}{n}}\right)\right].
\end{align}
By dividing in the cases $t\ge n$ and $t\le n$ and completing the squares, the above relation implies in particular that, for any $\epsilon\in(0,1)$ and $\delta\in(0,1/2)$, with probability at least $1-2\delta$, 
\begin{align}
\inf_{\bfV\in \calB}\left[
\Vert\frX^{(n)}(\bfV)\Vert_2-\left\{\left(1+\frac{C^2L^2\epsilon}{4}\right)^{1/2}
-\sqrt{\epsilon}CL\right\}
\right]\ge-C\frac{f(R_1)}{\sqrt{\epsilon n}}
-CL\sqrt{\frac{\log(1/\delta)}{\epsilon n}}.
\end{align}
We use the above property and the single-parameter peeling Lemma  \ref{lemma:peeling:1dim} with constraint set $V:=\mbS_\Pi$, functions
$M(\bfV):=\Vert\frX^{(n)}(\bfV)\Vert_2-\{(1+\frac{C^2L^2\epsilon}{4})^{1/2}
-\sqrt{\epsilon}CL\}$, 
$h(\bfV):=\calR(\bfV)$, $g(r):=C\frac{f(r)}{\sqrt{\epsilon n}}$ and constants $c:=2$ and $b:=\frac{CL}{\sqrt{\epsilon n}}$. Note that the claimed inequality trivially holds if $\Vert\bfV\Vert_\Pi=0$ by $L$-sub-Gaussianity. The desired inequality follows from Lemma \ref{lemma:peeling:1dim} combined with the fact that
$
\nicefrac{\bfV}{\Vert\bfV\Vert_\Pi}\in
V,
$
for all $\bfV\in\re^{m_1\times m_2}$ such that 
$\Vert\bfV\Vert_\Pi\neq0$ and the homogeneity of norms. The proof for the upper bound is similar.
\end{proof}

We now show $\IP$ in Definition \ref{def:design:property} for $L$-subgaussian designs. The proof follows from Chevet's inequality and a peeling argument in two parameters. A high-probability version of Chevet's inequality is suggested as an exercise in Vershynin \cite{2018vershynin}. We next give a proof for completeness. 
\begin{lemma}\label{lemma:constrained:upper:bound}
Let $V$ be any bounded subset of 
$\mbB_\Pi\times\mbB_2^{n}$ Define $V_1 := \{\bfV:\exists\, \bu \text{ s.t. } (\bfV,\bu)\in V\}$
and $V_2 := \{\bu:\exists\, \bfV \text{ s.t. } (\bfV,\bu)\in V\}$.

Then, there exists universal numerical constant $C>0$, such that, for any $n\ge1$ and $t>0$, with probability at least $1-2\exp(-t^2)$,
$$
\sup_{[\bfV;\bu]\in V}\langle\bu,\frX(\bfV)\rangle \le  CL[\mathscr G\big(\frS^{1/2}(V_1)) + \mathscr G\big(V_2\big) + t].
$$
\end{lemma}
\begin{proof}
In the following, the numerical constant $C>0$ may change from line to line. For each $(\bfV,\bu)\in V$, we define
\begin{align}
Z_{\bfV,\bu}&:=
\langle\bu,\frX(\bfV)\rangle=\sum_{i\in[n]}\bu_i\llangle\bfX_i,\bfV\rrangle,\qquad
W_{\bfV,\bu}:= L(\llangle\bfV,\frS^{1/2}(\bfXi)\rrangle+\langle\bu,\bxi\rangle),	
\end{align}
where $\bfXi\in\mdR^{p}$ and $\bxi\in\re^n$ are independent each one having iid $\calN(0,1)$ entries. Therefore,
$(\bfV,\bu)\mapsto W_{\bfV,\bu}$ defines a centered Gaussian process indexed by $V$.

We may easily bound the $\psi_2$-norm of the increments using rotation invariance of sub-Gaussian random variables. Indeed,  using that $\{\bfX_i\}$ is an iid sequence and Proposition 2.6.1  in \citep{2018vershynin}, there is an universal numerical constant $C>0$ such that, given $[\bfV;\bu]$ and $[\bfV';\bu']$ in $V$, 
\begin{align}
|Z_{\bfV,\bu}-Z_{\bfV',\bu'}|_{\psi_2}^2&=\left|\sum_{i\in[n]}\llangle\bfX_i,\bu_i\bfV-\bu'_i\bfV'\rrangle\right|_{\psi_2}^2\\
&\le C\sum_{i\in[n]}\left|\llangle\bfX_i,\bu_i\bfV-\bu_i'\bfV'\rrangle\right|_{\psi_2}^2\\
&\le2C\sum_{i\in[n]}\left|\llangle\bfX_i,(\bu_i-\bu_i')\bfV\rrangle\right|^2_{\psi_2}+2C\sum_{i\in[n]}\left|\llangle\bfX_i,\bu_i'(\bfV-\bfV')\rrangle\right|_{\psi_2}^2\\
&\le2CL^2\Vert\bu-\bu'\Vert_2^2\Vert\bfV\Vert_\Pi^2
+2CL^2\Vert\bu'\Vert_2^2\Vert\bfV-\bfV'\Vert_\Pi^2\le2CL^2\dist([\bfV;\bu],[\bfV';\bu']),
\label{lemma:aux2:eq1}
\end{align}
with the pseudo-metric $\dist([\bfV;\bu],[\bfV';\bu']):=\sqrt{\Vert\bu-\bu'\Vert_2^2+\Vert\bfV-\bfV'\Vert_\Pi^2}$, using that 
$\Vert\bfV\Vert_\Pi\le1$ and $\Vert\bu'\Vert_2\le1$.  On the other hand, by definition of the process $W$ it is easy to check that 
\begin{align}
\esp[(W_{\bfV,\bu}-W_{\bfV',\bu'})^2]=L^2(\Vert \bfV-\bfV'\Vert_\Pi^2+\Vert \bu-\bu'\Vert_2^2).
\label{lemma:aux2:eq2}
\end{align}	

From \eqref{lemma:aux2:eq1},\eqref{lemma:aux2:eq2}, we conclude that the processes $W$ and $Z$ satisfy the conditions of Talagrand's majoration and minoration generic chaining bounds for sub-Gaussian processes (e.g. Theorems 8.5.5 and 8.6.1 in \citep{2018vershynin}). Hence, there is a universal numerical constant $C>0$ such that, for any $t\ge0$, with probability at least $1-2e^{-t^2}$,
\begin{align}
\sup_{[\bfV;\bu]\in V}|Z_{\bfV,\bu}|\le CL\left\{\esp\left[\sup_{[\bfV;\bu]\in V}W_{\bfV,\bu}\right]+t\right\}.\label{lemma:aux2:eq3}
\end{align}
In above we used that $Z_{\bfV_0,\bu_0}=0$ at $[\bfV_0,\bu_0]=0$ and that the diameter of $V\subset\mbB_\Pi^{m_1\times m_2}\times\mbB_2^{n}$ under the metric 
$\dist$ is less than $2\sqrt{2}$. We also have
\begin{align}
\esp\bigg[\sup_{[\bfV;\bu]\in V}
W_{\bfV,\bu}\bigg]\le \esp\bigg[\sup_{\bfV\in V_1}\llangle\bfXi,\frS^{1/2}(\bfV)\rrangle\bigg]+
\esp\bigg[\sup_{\bu\in V_2}\langle\bu,\bxi\rangle\bigg]=
\mathscr G\big(\frS^{1/2}(V_1))+\mathscr G(V_2).
\end{align}
Joining the two previous inequalities complete the proof of the claimed inequality. 
\end{proof}

\begin{proposition}[$\IP$]\label{prop:gen:IP}
There exists universal constant $C>0$, such that for all $\delta\in(0,1]$
and $n\in\mathbb{N}$, with probability at least
$1-\delta$, the following property holds: for all
$[\bfV;\bu]\in\mdR^p\times \re^n$,
\begin{align}
\left|\langle\bu,\frX^{(n)}(\bfV)\rangle\right|&\le CL
\frac{1+\sqrt{\log(1/\delta)}}{\sqrt{n}}\big\|\bfV\big\|_\Pi
\Vert\bu\Vert_2\\
&\quad+CL\frac{\mathscr
G\left(\calR(\bfV)\frS^{1/2}(\mbB_\calR)\cap\Vert\bfV\Vert_\Pi\mbB_F\right)}{\sqrt{n}}\Vert\bu\Vert_2+CL
\frac{\mathscr G\left(\calQ(\bu)\mathbb B_\calQ^n\cap \Vert\bu\Vert_2\mathbb B_2^n\right)}{\sqrt{n}}\big\|\bfV\big\|_\Pi.
\end{align}
\end{proposition}
\begin{proof}
Let $R_1,R_2>0$ and define the sets
\begin{align}
V_1&:=\{\bfV\in\mdR^{p}:\Vert\bfV\Vert_\Pi\le 1,\calR(\bfV)\le R_1\},\\
V_{2}&:=\{\bu\in\re^n:\Vert\bu\Vert_2=1,\calQ(\bu)\le R_2\}.
\end{align}
Note that, 
\begin{align}
\mathscr{G}\left(\frS^{1/2}(V_1)\right)&\le R_1\mathscr G\left(\frS^{1/2}(\mbB_\calR)\cap R_1^{-1}\mbB_F\right),\quad\quad
\mathscr{G}(V_{2})\le R_2\mathscr{G}\left(\mathbb{B}_\calQ^n\cap R_2^{-1}\mathbb{B}_2^n\right).
\end{align}
Define for convenience the functions
\begin{align}
g(r):=CL\frac{\mathscr G\left(\frS^{1/2}(\mbB_\calR)\cap r^{-1}\mbB_F\right)}{\sqrt{n}}r,
\quad\bar g(\bar r):=CL\frac{\mathscr{G}\left(\mathbb{B}_\calQ^n\cap \bar r^{-1}\mathbb{B}_2^n\right)}{\sqrt{n}} \bar r.
\end{align}
By Lemma \ref{lemma:constrained:upper:bound}, there is universal constant $C>0$ such that, for any $R_1,R_2>0$ and $\delta\in(0,1]$, with probability at least $1-\delta$, the following inequality holds:
\begin{align}
\sup_{[\bfV;\bu]\in V_1\times V_2}\langle\bu,\frX^{(n)}(\bfV)\rangle
\le  g(R_1) + \bar g(R_2)+CL\sqrt{\frac{\log(1/\delta)}{n}}.
\end{align}
We use the above property and the bi-parameter peeling Lemma  \ref{lemma:peeling:2dim} with constraint set $V:=\mbB_\Pi\times\mbB_2^n$, functions
$M(\bfV,\bu):=-\langle\bu,\frX^{(n)}(\bfV)\rangle$, 
$h(\bfV,\bu):=\calR(\bfV)$, $\bar h(\bfV,\bu):=\calQ(\bu)$, $g$ and $\bar g$ and constants $c:=1$ and $b:=CL\sqrt{1/n}$. Note that the claimed inequality trivially holds if $\Vert\bfV\Vert_\Pi=0$ or $\bu=0$. Indeed, since $\bfX_i$ is $L$-sub-Gaussian, $\Vert\bfV\Vert_\Pi=0$ implies that $\llangle\bfX_i,\bfV\rrangle=0$ with probability 1. The desired inequality follows from Lemma \ref{lemma:peeling:2dim} combined with the fact that
$
[\nicefrac{\bfV}{\Vert\bfV\Vert_\Pi};\nicefrac{\bu}{\Vert\bu\Vert_2}]\in
V,
$
for all $[\bfV;\bu]\in\re^{m_1\times m_2}\times\re^n$ such that $\Vert\bfV\Vert_\Pi\neq0$ and $\bu\neq0$ and the homogeneity of norms.
\end{proof}

We now turn our attention to property $\MP$ in Definition \ref{def:design:property}. Of course,  
$$
\langle\bxi^{(n)},\frM^{(n)}(\bfV,\bu)\rangle=\frac{1}{n}\sum_{i\in[n]}\xi_i\llangle\bfV,\bfX_i\rrangle+\frac{1}{\sqrt{n}}\sum_{i\in[n]}\xi_i\bu_i.
$$
The control of the first term will follow from a bound for the multiplier process (Theorem \ref{thm:mult:process} in the supplemental material). As for the second, one may avoid chaining by using a standard symmetrization-contraction argument which we present for completeness.

\begin{lemma}\label{lemma:noise:concentration}
Let $U$ be any bounded subset of $\mbB^n_2$. 

Then, for any $n\ge1$ and $t>0$, with probability at least $1-\exp(-t^2/2)$,
we have
$$
\sup_{\bu\in U}\langle\bxi,\bu\rangle \le  8\sigma\left[\mathscr G\big(V_2\big) + t\right].
$$
\end{lemma}
\begin{proof}
Let $\bepsilon\in\re^n$ be a vector whose components are iid Rademacher random variables. Let $t\ge0$. The symmetrization inequality (e.g. Exercise 11.5 in \citep{2013boucheron:lugosi:massart}) yields
\begin{align}
\esp\left[\exp\left(t\sup_{\bz\in U}\langle\bxi,\bz\rangle\right)\right]\le
\esp\left[\exp\left(t\sup_{\bz\in U}2\langle\bepsilon\odot\bxi,\bz\rangle\right)\right],
\end{align}
where $\bepsilon\odot\bxi$ is the vector with components 
$\{\bepsilon_i\bxi_i\}_{i\in[n]}$. For each $i\in[n]$, $\epsilon_i\bxi_i$ is a symmetric sub-Gaussian random variable with $\psi_2$-norm not greater than $\sigma$. Let $\bg\sim\calN(0,\bfI_n)$ a standard normal vector in $\re^n$. Hence, the following tail dominance holds: 
$\prob(|\epsilon_i\bxi_i|>\tau)\le 4\prob(\sigma|\bg_i|>\tau)$ for all $i\in[n]$ and for all $\tau>0$. From the contraction principle as stated in Lemma 4.6 in \cite{1991ledoux:talagrand}, 
\begin{align}
\esp\left[\exp\left(t\sup_{\bz\in U}2\langle\bepsilon\odot\bxi,\bz\rangle\right)\right]
\le\esp\left[\exp\left(8\sigma t\sup_{\bz\in U}\langle\bg,\bz\rangle\right)\right].
\end{align}
Since $U\subset\mbB_2^n$, the function $\bg\mapsto\sup_{\bz\in U}\langle\bg,\bz\rangle$ is $1$-Lipschitz under the $\ell_2$-norm. By Theorem 5.5 in \cite{2013boucheron:lugosi:massart}, the RHS of the previous inequality is upper bounded by 
\begin{align}
\exp\left(8\sigma t\mathscr{G}(U)+32t^2\sigma^2\right).
\end{align}
A standard Chernoff bound concludes the proof. 
\end{proof}

\begin{proposition}[$\MP$]\label{prop:gen:MP}
For $t,s>0$ define
$$
\triangle(t,s):=\frac{1}{\sqrt{n}}\sqrt{\log t}+\frac{1}{n}\log t+\frac{1}{n}\sqrt{(\log t)(\log s)}
+\frac{1}{n}\sqrt{\log s}+\frac{1}{\sqrt{n}}[1+\sqrt{\log s}].
$$

Then there exists universal constants $C>0$, $c_0,c\ge2$ such that for all $n\in\mathbb{N}$ and all $\delta\in(0,1/c]$ and $\rho\in(0,1/c_0]$, with probability at least
$1-\delta-\rho$, the following property holds: for all
$[\bfV;\bu]\in\mdR^p\times \re^n$,
\begin{align}
\langle\bxi^{(n)},\frM^{(n)}(\bfV,\bu)\rangle &\le C\sigma L\cdot\triangle(\nicefrac{1}{\delta},\nicefrac{1}{\rho})\cdot\Vert[\bfV;\bu]\Vert_\Pi\\
&+C\sigma L\left[1+\frac{\sqrt{\log(1/\rho)}}{\sqrt{n}}\right]\frac{\mathscr G\big(\frS^{1/2}(\mbB_{\calR}))}{\sqrt{n}}\calR(\bfV)
+C\sigma \frac{\mathscr G\big(\mbB_{\calQ}\big)}{\sqrt{n}}\calQ(\bu).
\end{align} 
\end{proposition}
\begin{proof}
Define the set
\begin{align}
V_0:=\{[\bfV;\bu]:\Vert[\bfV;\bu]\Vert_\Pi\le1,\calR(\bfV)\le R_1,\calQ(\bu)\le R_2\}.
\end{align}
Theorem \ref{thm:mult:process} in the Appendix (together with Talagrand's majorization theorem), Lemma \ref{lemma:noise:concentration} and an union bound imply that, for universal constants $C>0$ and $c,c_0\ge2$, for all $\delta\in(0,1/c)$ and $\rho\in(0,1/c_0)$, with probability at least $1-\delta-\rho$,
\begin{align}
\sup_{[\bfV;\bu]\in V_0}\langle\bxi^{(n)},\frM^{(n)}(\bfV,\bu)\rangle&\le C\sigma L\left(\sqrt{\frac{\log(1/\rho)}{n}}+1\right)
\frac{\mathscr G\big(\frS^{1/2}(\mbB_\calR)\big)}{\sqrt{n}}R_1\\
&+C\sigma L\left(\sqrt{\frac{\log(1/\delta)}{n}}+\frac{\log(1/\delta)}{n}
+\sqrt{\frac{\log(1/\delta)\log(1/\rho)}{n^2}}\right)\\
&+C\sigma \frac{\mathscr G\big(\mbB^n_2\big)}{\sqrt{n}}R_2+C\sigma\sqrt{\frac{\log(1/\delta)}{n}}.
\end{align}
We will now apply Lemma \ref{lemma:peeling:multiplier:process} with set $V
:=\{[\bfV;\bu]:\Vert[\bfV;\bu]\Vert_\Pi\le1\}$, functions $M(\bfV,\bu):=-\langle\bxi^{(n)},\frM^{(n)}(\bfV,\bu)\rangle$, $h(\bfV):=\calR(\bfV)$, $\bar h(\bu):=\calQ(\bu)$, 
$
g(R_1):=\mathscr G\big(\frS^{1/2}(\mbB_\calR)\big)R_1,
$
$
\bar g(R_2):=\mathscr G\big(\mbB^n_2\big)R_2/L,
$
and constant $b:=C\sigma L$, where we recall $L\ge1$. The result then follow from such lemma and homogeneity noting that 
$\frac{[\bfV;\bu]}{\Vert[\bfV;\bu]\Vert_\Pi}\in V_0$ for $[\bfV;\bu]$ such that $\Vert[\bfV;\bu]\Vert_\Pi\neq0$.
\end{proof}

\subsection{Proof of Theorem \ref{thm:response:sparse:regression}}\label{ss:proof:thm:response:sparse:regression}

We now set $\calR:=\Vert\cdot\Vert_1$, 
$\calQ:=\Vert\cdot\Vert_\sharp$. A standard Gaussian maximal inequality implies $\mathscr{G}(\bfSigma^{1/2}\mbB_1^p)\lesssim\rho_1(\bfSigma)\sqrt{\log p}$. Proposition E.2 in  \cite{2018bellec:lecue:tsybakov} implies 
$\mathscr{G}(\mbB_\sharp)\lesssim1$. 

We next use Proposition \ref{prop:gen:TP} with $\epsilon=\frac{c}{C^2L^2}$ for sufficiently small $c\in(0,1)$ and assuming
\begin{align}
\delta\ge \exp\left(-c_1\frac{n}{L^4}\right),\label{equation:delta:n}
\end{align}
for large enough universal constant $c_1>0$. It follows that for an universal constant $\sa_1\in(0,1)$ and
$$
\sa_2\asymp L\rho_1(\bfSigma)\sqrt{\frac{\log p}{n}},
$$
on an event $\Omega_1$ of probability at least $1-\delta/3$, property 
$\TP_{\Vert\cdot\Vert_1}(\sa_1;\sa_2)$ is satisfied. 

From Proposition \ref{prop:gen:IP}, for every $\delta\in(0,1)$ and
$$
\sb_1\asymp L\frac{1+\sqrt{\log(1/\delta)}}{\sqrt{n}},
\quad
\sb_2\asymp L\rho_1(\bfSigma)\sqrt{\frac{\log p}{n}},\quad
\sb_3\asymp \frac{L}{\sqrt{n}},
$$
on an event $\Omega_2$ of probability at least $1-\delta/3$, property $\IP_{\Vert\cdot\Vert_1,\Vert\cdot\Vert_\sharp}(\sb_1;\sb_2;\sb_3)$ is satisfied. 

From Lemma \ref{lemma:ATP}, by enlarging $c_1$ if necessary, for an universal constant $\sc_1\in(0,1)$ and
$$
\sc_2\asymp L\rho_1(\bfSigma)\sqrt{\frac{\log p}{n}},\quad \sc_3\asymp \frac{L}{\sqrt{n}},
$$
$\ATP_{\Vert\cdot\Vert_1,\Vert\cdot\Vert_\sharp}(\sc_1;\sc_2;\sc_3)$ is satisfied on $\Omega_1\cap\Omega_2$.  

We now use Proposition \ref{prop:gen:MP} (with $\delta=\rho$). By enlarging $c_1$ in \eqref{equation:delta:n} if necessary, if we take
$$
\sf_1\asymp\sigma L\frac{1+\sqrt{\log(1/\delta)}}{\sqrt{n}},\quad
\sf_2\asymp\sigma L\rho_1(\bfSigma)\sqrt{\frac{\log p}{n}}, \quad
\sf_3\asymp \frac{\sigma}{\sqrt{n}}, 
$$
we have by Proposition \ref{prop:gen:MP} that on an event 
$\Omega_3$ of probability at least $1-\delta/3$,  
$\MP_{\Vert\cdot\Vert_1,\Vert\cdot\Vert_\sharp}(\sf_1;\sf_2;\sf_3)$ is satisfied.

By an union bound and enlarging constants, for $\delta$ satisfying \eqref{equation:delta:n}, on the event 
$\Omega_1\cap\Omega_2\cap\Omega_3$ of probability at least $1-\delta$, all properties $\TP$, $\IP$, $\ATP$ and 
$\MP$ hold with constants as specified above. We next assume such event is realized and invoke Theorem \ref{thm:improved:rate}. It is straightforward to check 
$\rm{(iii)}$ by the definitions of $\tau$ and $\lambda=\gamma\tau$ in Theorem \ref{thm:response:sparse:regression}. Note that that $R\le\sqrt{s}\mu(\bb^*)$ with $\mu(\bb^*):=\mu(\calC_{\bb^*,\Vert\cdot\Vert_1}(12))$. In item $\rm{(iv)}$, conditions \eqref{cond1:general:norm} and \eqref{cond2:general:norm} require checking
\begin{align}
C\sigma^2 L^2\rho_1^2(\bfSigma)\mu^2(\bb^*)\frac{s\log p}{n}+C\sigma^2 L^2\epsilon\log(1/\epsilon)<1/2,
\end{align}
and that $\frac{\sf_1}{\sc_1^2}<1/2$. Assuming further that $\delta\ge\exp\left(-c_2\frac{n}{\sigma^2L^2}\right)$ for universal constant $c_2>0$ and enlarging $C$ above, item 
$\rm{(iv)}$ is satisfied. With all conditions of Theorem \ref{thm:improved:rate} taking place, the rate in Theorem \ref{thm:response:sparse:regression} follows.\footnote{Note that $\square$ and $\triangle$ are bounded by a numerical constant of 
$\calO(L)$.}

\begin{remark}\label{rem:cones:slope}
With some abuse of notation, let $\Vert\cdot\Vert_\sharp$ denote the Slope norm in $\re^p$ with the sequence $\bar w_j=\log(\bar A p/j)$ for some $\bar A\ge2$. The proof with $\calR=\Vert\cdot\Vert_\sharp$ follows a similar path. One difference is that we need to consider different cones. For $c_0,\gamma>0$, we define

\begin{align}
\overline\calC_s(c_0)&:=\left\{\bv\in\re^p
:\sum_{i=s+1}^p\bar\omega_j\bv_j^\sharp\le c_0\Vert\bv\Vert_2\sqrt{\sum_{j=1}^s\bar\omega_j^2}\right\},\\
\overline\calC_s(c_0,\gamma)&:=\left\{[\bv;\bu]\in\re^p\times\re^n: \gamma\sum_{i=s+1}^p\bar\omega_j\bv_j^\sharp
+\sum_{i=o+1}^n\omega_i\bu_i^\sharp\le c_0\left[\gamma\Vert\bv\Vert_2+\Vert\bu\Vert_2\right]\right\}.
\end{align}

In this setting, $\mu(\bb^*):=\mu(\overline\calC_s(12))$. We also use the bound 
$\mathscr{G}(\bfSigma^{1/2}\mbB_\sharp)\lesssim\rho_1(\bfSigma)$, which follows from Proposition E.2 in  \cite{2018bellec:lecue:tsybakov}. We omit the details.
\end{remark}

\subsection{Proof of Theorem \ref{thm:response:trace:regression}}\label{ss:proof:thm:response:trace:regression}

The proof follows exact same guidelines as the Proof of Theorem \ref{thm:response:sparse:regression}. The changes are that for the nuclear norm $\calR:=\Vert\cdot\Vert_N$, we have
$\mathscr{G}(\frS^{1/2}(\mbB_{\Vert\cdot\Vert_N}))\lesssim\rho_N(\bfSigma)(\sqrt{d_1}+\sqrt{d_2})$ by Lemma H.1 in \cite{2011negahban:wainwright}
and $R\le\sqrt{r}\mu(\bfB^*)$ with $\mu(\bfB^*):=\mu(\calC_{\bfB^*,\Vert\cdot\Vert_N}(12))$.

\section{Proof of Theorem \ref{thm:tr:reg:matrix:decomp}}
\label{s:proof:thm:trace:reg:matrix:decomp}

Recall estimator \eqref{equation:aug:estimator:trace:reg:matrix:decomp}. \begin{tcolorbox}
Throughout Sections \ref{ss:cones:RE:trace:matrix:decomp}  
and
\ref{ss:deterministic:bounds:tr:reg:matrix:decomp}, we see 
$\Vert\cdot\Vert_\Pi$ as a generic pseudo-norm and
$\{(\bfX_i,\xi_i)\}_{i\in[n]}$ as a deterministic sequence satisfying \eqref{equation:str:eq:trace:reg:matrix:decomp} where
$\bxi=(\xi_i)_{i\in[n]}$ and $\frX$ is the design operator associated to the sequence $\{\bfX_i\}_{i\in[n]}$. Probabilistic assumptions are used only in Sections \ref{ss:subgaussian:designs:tr:reg:matrix:decomp} and
 \ref{ss:proof:tr:reg:matrix:decomp}.
\end{tcolorbox}

\subsection{Additional design properties and cones}\label{ss:cones:RE:trace:matrix:decomp}

In the following, let $\calR$ and $\calQ$ be norms on 
$\mdR^p$. We first present a definition bounding the \emph{product process}:
\begin{align}
[\bfV;\bfW]\mapsto\frac{1}{n}\sum_{i\in[n]}\bigg[
\llangle\bfX_i,\bfV\rrangle\llangle\bfX_i,\bfW\rrangle-\esp[\llangle\bfX_i,\bfV\rrangle\llangle\bfX_i,\bfW\rrangle]\bigg].\label{equation:def:product:process}
\end{align}
For convenience, we define the empirical bilinear form
$$
\llangle\bfV,\bfW\rrangle_n:=\frac{1}{n}\sum_{i\in[n]}
\llangle\bfX_i,\bfV\rrangle\llangle\bfX_i,\bfW\rrangle
=\langle\bfV,(\nicefrac{\frX^*\circ\frX}{n})(\bfW)\rangle,$$
where $\frX^*$ denotes the adjunct operator of $\frX$. 

\begin{tcolorbox}
\begin{definition}[$\PP$]\label{def:design:pro:multi:PP}
Given positive numbers $(\sb_1,\sb_2,\sb_3,\sb_4)$, we say that $\frX$ satisfies 
$\PP_{\calR,\calQ}(\sb_1;\sb_2;\sb_3;\sb_4)$ if for all $[\bfV;\bfW]\in\mdR^p$,
\begin{align}
\left|\llangle\bfV,\bfW\rrangle_n-\llangle\bfV,\bfW\rrangle_\Pi\right|&\le \sb_1\left\Vert\bfV\right\Vert_{\Pi}
\Vert\bfW\Vert_\Pi+\sb_2\calR(\bfV)\Vert\bfW\Vert_\Pi+\sb_3\left\Vert\bfV\right\Vert_\Pi\calQ(\bfW)\\
&\quad+\sb_4\calR(\bfV)\calQ(\bfW).
\end{align}
\end{definition}
\end{tcolorbox}

We will also need variations of $\ATP$ and $\MP$ in Definition \ref{def:design:property} which will be used throughout Section \ref{s:proof:thm:trace:reg:matrix:decomp}. 

\begin{tcolorbox}
\begin{definition}[$\ATP$ and $\MP$]\label{def:design:pro:multi:reg:ATP:MP}
Given positive numbers
$(\sc_1,\sc_2,\sc_3)$, we say that $\frX$ satisfies $\ATP_{\calR,\calQ}(\sc_1;\sc_2;\sc_3)$ if for all 
$[\bfV;\bfW]\in\mdR^p$,
\begin{align}
\|\frX^{(n)}(\bfV+\bfW)\|_2^2\ge 
\bigg[\sc_1\Vert[\bfV;\bfW]\Vert_\Pi-\sc_2\calR(\bfV)-\sc_3\calQ(\bfW)\bigg]_+^2-2|\llangle\bfV,\bfW\rrangle_\Pi|.
\end{align}

\quad

In addition, given positive numbers
$(\sf_1,\sf_2,\sf_3)$, we will say that $(\frX,\bxi)$ satisfies $\MP_{\calR,\calQ}(\sf_1;\sf_2;\sf_3)$ if for all $[\bfV;\bfW]\in\mdR^p$, 
\begin{align}
|\langle\bxi^{(n)},\frX^{(n)}(\bfV+\bfW)\rangle|\le 
\sf_1\Vert[\bfV;\bfW]\Vert_\Pi
+\sf_2\calR(\bfV)
+\sf_3\calQ(\bfW).
\end{align}
\end{definition}
\end{tcolorbox}

The next lemma states that $\ATP$ is a consequence of 
$\TP$ and $\PP$. We omit the proof as it follows similar reasoning of Lemma \ref{lemma:ATP}. 
\begin{lemma}[$\TP+\PP\Rightarrow\ATP$]\label{lemma:ATP:tr:reg:matrix:decomp}
Let positive numbers $\sa_1,\bar\sa_1$, $\sa_2$, $\bar\sa_2$, $\sb_1$, $\sb_2$ and $\sb_3$ with $\sb_1<\sa_1\wedge\bar\sa_1$. Suppose that $\frX$ satisfies:
\begin{itemize}
\item[\rm{(i)}] $\TP_{\Vert\cdot\Vert_{\Pi},\calR}(\sa_1;\sa_2)$.
\item[\rm(ii)] $\TP_{\Vert\cdot\Vert,\calQ}(\bar\sa_1;\bar\sa_2)$.
\item[\rm{(iii)}] $\PP_{\Vert\cdot\Vert_{\Pi},\calR,\calQ}(\sb_1;\sb_2;\sb_3;\beta\sb_2\sb_3)$ for some $\beta\in(0,1]$.
\end{itemize}

Then, for any $\alpha>0$ such that $\alpha^2+\sb_1\le(\sa_1\wedge\bar\sa_1)^2$, 
$\frX$ satisfies the $\ATP_{\Vert\cdot\Vert_{\Pi},\calR,\calQ}(\sc_1;\sc_2;\sc_3)$ in Definition \ref{def:design:pro:multi:reg:ATP:MP} with constants
$\sc_1=\sqrt{(\sa_1\wedge\bar\sa_1)^2- \sb_1 -\alpha^2}$, $\sc_2=\sa_2+\sb_2/\alpha$, $\sc_3=\bar\sa_2+\sb_3/\alpha$. Taking $\alpha = (\sa_1\wedge\bar\sa_1)/2$, we obtain that
$\ATP(\sc_1;\sc_2;\sc_3)$ holds with constants
$\sc_1=\sqrt{(3/4)(\sa_1\wedge\bar\sa_1)^2- \sb_1}$, $\sc_2=\sa_2+2\sb_2/(\sa_1\wedge\bar\sa_1)$ and
$\sc_3=\bar\sa_2+2\sb_3/(\sa_1\wedge\bar\sa_1)$.
\end{lemma}

We will need an additional cone definition. 

\begin{tcolorbox}
\begin{definition}
Let $\calR$ and $\calQ$ be decomposable norms on $\mdR^p$ (see Definition \ref{def:decomposable:norm} in Section \ref{ss:cones:RE}). Given $[\bfB,\bfGamma]\in(\mdR^p)^2$, let 
$\calP_{\bfB}$ and $\calP_{\bfGamma}$ denote the projection maps  associated to $(\calR,\bfB)$ and $(\calQ,\bfGamma)$ respectively. Given $c_0,\gamma>0$, we define the cone $\calC=\calC_{\bfB,\bfGamma,\calR,\calQ}(c_0,\gamma)$ by
\begin{align}
\calC&:=\left\{[\bfV;\bfW]: \gamma\calR(\calP_{\bfB}^\perp(\bfV))
+\calQ(\calP_{\bfGamma}^\perp(\bfW))\le c_0\left[\gamma\calR(\calP_{\bfB}(\bfV))+\calQ(\calP_{\bfGamma}(\bfW))\right]\right\}.
\end{align}
We will omit the subscripts $\calR$ and $\calQ$ when the norms are clear in the context. 
\end{definition}
\end{tcolorbox}

\subsection{Deterministic bounds}\label{ss:deterministic:bounds:tr:reg:matrix:decomp}

Throughout this section, we work with the nuclear norm 
$\Vert\cdot\Vert_N$ and the $\ell_1$-norm $\Vert\cdot\Vert_1$ on $\mdR^p$. With some abuse of notation, 
we will denote by $\calP_{\bfB^*}$ the projection associated to $(\Vert\cdot\Vert_N,\bfB^*)$ and by
$\calP_{\bfGamma^*}$ the projection associated to $(\Vert\cdot\Vert_1,\bfGamma^*)$ (see Definition \ref{def:decomposable:norm} in Section \ref{ss:cones:RE}).

\begin{lemma}[Dimension reduction]\label{lemma:dim:reduction:tr:reg:matrix:decomp}
Grant Assumption \ref{assump:low:spikeness} and suppose that:
\begin{itemize}
\item[\rm (i)] $(\frX,\bxi)$ satisfies the $\MP_{\Vert\cdot\Vert_N,\Vert\cdot\Vert_1}(\sf_1;\sf_2;\sf_3)$ 
for some positive numbers $\sf_1$, $\sf_2$ and $\sf_3$. 
\item[\rm{(ii)}] $\frX$ satisfies the $\ATP_{\Vert\cdot\Vert_N,\Vert\cdot\Vert_1}(\sc_1;\sc_2;\sc_3)$ for some positive numbers $\sc_1,\sc_2,\sc_3$.  
\item[\rm (iii)] 
$
\lambda = \gamma\tau
\ge4(\sf_2\vee\sc_2),\quad\text{and}\quad
\tau \ge 4[\sc_3\vee(\sf_3+\nicefrac{4\sa^*}{\sqrt{n}})].
$
\item[\rm (iv)] $2\sf_1\le\sc_1$. 
\end{itemize}
Define
\begin{align}
\triangle&:=(\nicefrac{3\lambda}{2})\Vert\calP_{\bfB^*}(\bfDelta^{\hat\bfB})\Vert_N 
-(\nicefrac{\lambda}{2})\Vert\calP_{\bfB^*}^\perp(\bfDelta^{\hat\bfB})\Vert_N
+(\nicefrac{3\tau}{2})\Vert\calP_{\bfGamma^*}(\bfDelta^{\hat\bfGamma})\Vert_1 
-(\nicefrac{\tau}{2})\Vert\calP_{\bfGamma^*}^\perp(\bfDelta^{\hat\bfGamma})\Vert_1.
\end{align}

Then either $[\bfDelta^{\hat\bfB};\bfDelta^{\hat\bfGamma}]\in\calC_{\bfB^*,\bfGamma^*}(3,\gamma)$ and
\begin{align}
\left(\sc_1\Vert[\bfDelta^{\hat\bfB};\bfDelta^{\hat\bfGamma}]\Vert_{\Pi}-(\nicefrac{\lambda}{2})\Vert\bfDelta^{\hat\bfB}\Vert_N)-(\nicefrac{\tau}{2})\Vert\bfDelta^{\hat\bfGamma}\Vert_1\right)_+^2\le
\sf_1\Vert[\bfDelta^{\hat\bfB};\bfDelta^{\hat\bfGamma}]\Vert_{\Pi}+ \triangle.\label{lemma:dim:reduction:tr:reg:MD:contraction}
\end{align}  
or $[\bfDelta^{\hat\bfB};\bfDelta^{\hat\bfGamma}]\in\calC_{\bfB^*,\bfGamma^*}(6,\gamma)$ and 
\begin{align}
\frac{\sc_1^2}{4}\Vert[\bfDelta^{\hat\bfB};\bfDelta^{\hat\bfGamma}]\Vert_{\Pi}^2\le
\sf_1\Vert[\bfDelta^{\hat\bfB};\bfDelta^{\hat\bfGamma}]\Vert_{\Pi}+ \triangle.\label{lemma:dim:reduction:tr:reg:MD:contraction:2}
\end{align} 
or
\begin{align}
\Vert[\bfDelta^{\hat\bfB};\bfDelta^{\hat\btheta}]\Vert_\Pi&\le 
80\frac{\sf_1^2}{\sc_1^4},\label{lemma:dim:reduction:rate:l2:MD}\\
\lambda\Vert\bfDelta^{\hat\bfB}\Vert_N+\tau\|\bfDelta^{\hat\btheta}\|_1&\le37.4\frac{\sf_1^2}{\sc_1^2}.\label{lemma:dim:reduction:rate:l1:MD}
\end{align}
\end{lemma}
\begin{proof}
By the first order condition of \eqref{equation:aug:estimator:trace:reg:matrix:decomp} at $[\hat\bfB,\hat\bfGamma]$, there exist 
$\bfV\in\partial\Vert\hat\bfB\Vert_N$ and $\bfW\in\partial\Vert\hat\bfGamma\Vert_1$ such that for all $[\bfB;\bfGamma]$ satisfying $\Vert\bfB\Vert_\infty\le\frac{\sa^*}{\sqrt{n}}$, 
\begin{align}
\sum_{i\in[n]}\left[y_i^{(n)}-\frX^{(n)}_i(\hat\bfB+\hat\bfGamma)\right]\llangle\bfX^{(n)}_i,\hat\bfB-\bfB\rrangle&\ge\lambda\llangle\bfV,\hat\bfB-\bfB\rrangle,\\
\sum_{i\in[n]}\left[y_i^{(n)}-\frX^{(n)}_i(\hat\bfB+\hat\bfGamma)\right]\llangle\bfX^{(n)}_i,\hat\bfGamma-\bfGamma\rrangle&\ge\lambda\llangle\bfW,\hat\bfGamma-\bfGamma\rrangle.\label{equation:first:order:condition:MD}
\end{align}
Evaluating at $[\bfB^*;\bfGamma^*]$, which satisfies 
$\Vert\bfB^*\Vert_\infty\le\frac{\sa^*}{\sqrt{n}}$ by assumption, and using  \eqref{equation:str:eq:trace:reg:matrix:decomp}
we get
\begin{align}
\Vert\frX^{(n)}(\bfDelta^{\hat\bfB}+\bfDelta^{\hat\bfGamma})\Vert_2^2\le
\llangle\bxi^{(n)},\frX^{(n)}(\bfDelta^{\hat\bfB}+\bfDelta^{\hat\bfGamma})\rrangle
-\lambda\llangle\bfV,\bfDelta^{\hat\bfB}\rrangle
-\tau\llangle\bfW,\bfDelta^{\hat\bfGamma}\rrangle.
\end{align}
Using $\MP$ in item (i) and
\begin{align}
-\llangle\bfDelta^{\hat\bfB},\bfV\rrangle \le 
\Vert\bfB\Vert_N-\Vert\hat\bfB\Vert_N,\quad\quad
-\llangle\bfDelta^{\hat\bfGamma},\bfW\rrangle \le 
\Vert\bfGamma\Vert_1-\Vert\hat\bfGamma\Vert_1,
\end{align}
we get
\begin{align}
\Vert\frX^{(n)}(\bfDelta^{\hat\bfB}+\bfDelta^{\hat\bfGamma})\Vert_2^2
&\le \sf_1\Vert[\bfDelta^{\hat\bfB};\bfDelta^{\hat\bfGamma}]\Vert_{\Pi}
+\sf_2\Vert\bfDelta^{\hat\bfB}\Vert_N+\sf_3\|\bfDelta^{\hat\bfGamma}\|_1\\
& + \lambda \big(\Vert\bfB^*\Vert_N - \Vert\hat\bfB\Vert_N\big) +  
\tau\big(\|\bfGamma^*\|_1 -\|\hat\bfGamma\|_1\big).
\end{align}
For convenience, we define the quantity
\begin{align}
\widetilde H:=(\sc_2\vee\sf_2)\Vert\bfDelta^{\hat\bfB}\Vert_N
+\left[\sc_3\vee\left(\sf_3+\frac{4\sa^*}{\sqrt{n}}\right)\right]\|\bfDelta^{\hat\bfGamma}\|_1,
\end{align}

By norm duality and isotropy, $|\llangle\bfDelta^{\hat\bfB},\bfDelta^{\hat\bfGamma}\rrangle_\Pi|=
|\llangle\bfDelta^{\hat\bfB},\bfDelta^{\hat\bfGamma}\rrangle|\le\frac{2\sa^*}{\sqrt{n}}\Vert\bfDelta^{\hat\bfGamma}\Vert_1$. This fact, the previous display and 
$\ATP$ in item  (ii) imply
\begin{align}
\left(\sc_1\Vert[\bfDelta^{\hat\bfB};\bfDelta^{\hat\bfGamma}]\Vert_{\Pi}-\widetilde H\right)_+^2\le
\sf_1\Vert[\bfDelta^{\hat\bfB};\bfDelta^{\hat\bfGamma}]\Vert_{\Pi}+ \widetilde H
+\lambda \big(\Vert\bfB^*\Vert_N - \Vert\hat\bfB\Vert_N\big) +  
\tau\big(\|\bfGamma^*\|_1 -\|\hat\bfGamma\|_1\big).\label{lemma:dim:reduction:tr:reg:matrix:dec:eq1}
\end{align}  
We now divide in two cases. 
\begin{description}
\item[Case 1:] $\widetilde H\ge\sf_1\Vert[\bfDelta^{\hat\bfB};\bfDelta^{\hat\bfGamma}]\Vert_{\Pi}$.

We obtain that the RHS of  \eqref{lemma:dim:reduction:tr:reg:matrix:dec:eq1} is upper bounded by 
\begin{align}
&2(\sc_2\vee\sf_2)\Vert\bfDelta^{\hat\bfB}\Vert_N
+2\left[\sc_3\vee\left(\sf_3+\frac{4\sa^*}{\sqrt{n}}\right)\right]\|\bfDelta^{\hat\bfGamma}\|_1
+\lambda \big(\Vert\bfB^*\Vert_N - \Vert\hat\bfB\Vert_N\big) + \tau\big(\|\bfGamma^*\|_1 -\|\hat\bfGamma\|_1\big)\\
&\stackrel{\rm (iii)}{\le}(\nicefrac{\lambda}{2})\Vert\bfDelta^{\hat\bfB}\Vert_N
+(\nicefrac{\tau}{2})\Vert\bfDelta^{\hat\bfGamma}\Vert_1+\lambda \big(\Vert\bfB^*\Vert_N - \Vert\hat\bfB\Vert_N\big) + \tau\big(\|\bfGamma^*\|_1 -\|\hat\bfGamma\|_1\big)\\
&\stackrel{\rm Lemma\mbox{ } \ref{lemma:A1:B}}{\le}
(\nicefrac{3\lambda}{2})\Vert\calP_{\bfB^*}(\bfDelta^{\hat\bfB})\Vert_N 
-(\nicefrac{\lambda}{2})\Vert\calP_{\bfB^*}^\perp(\bfDelta^{\hat\bfB})\Vert_N
+(\nicefrac{3\tau}{2})\Vert\calP_{\bfGamma^*}(\bfDelta^{\hat\bfGamma})\Vert_1 
-(\nicefrac{\tau}{2})\Vert\calP_{\bfGamma^*}^\perp(\bfDelta^{\hat\bfGamma})\Vert_1,
\end{align}
implying $[\bfDelta^{\hat\bfB};\bfDelta^{\hat\bfGamma}]\in\calC_{\bfB^*,\bfGamma^*}(3,\gamma)$. The above bound also implies, again by (iii) and \eqref{lemma:dim:reduction:tr:reg:matrix:dec:eq1}, inequality  \eqref{lemma:dim:reduction:tr:reg:MD:contraction}. 

\item[Case 2:] $\widetilde H\le\sf_1\Vert[\bfDelta^{\hat\bfB};\bfDelta^{\hat\bfGamma}]\Vert_{\Pi}$. 

A similar bound of Case 1, using (iii) and Lemma \ref{lemma:A1:B} with $\nu=1/4$, implies that $\widetilde H +\lambda \big(\Vert\bfB^*\Vert_N - \Vert\hat\bfB\Vert_N\big) +  
\tau\big(\|\bfGamma^*\|_1 -\|\hat\bfGamma\|_1\big)$ is upper bounded by
\begin{align}
(\nicefrac{5\lambda}{4})\Vert\calP_{\bfB^*}(\bfDelta^{\hat\bfB})\Vert_N 
-(\nicefrac{3\lambda}{4})\Vert\calP_{\bfB^*}^\perp(\bfDelta^{\hat\bfB})\Vert_N
+(\nicefrac{5\tau}{4})\Vert\calP_{\bfGamma^*}(\bfDelta^{\hat\bfGamma})\Vert_1 
-(\nicefrac{3\tau}{4})\Vert\calP_{\bfGamma^*}^\perp(\bfDelta^{\hat\bfGamma})\Vert_1\le\triangle.
\end{align}
This fact, $2\sf_1\le\sc_1$ in item (iv) and \eqref{lemma:dim:reduction:tr:reg:matrix:dec:eq1} imply
\begin{align}
\frac{\sc_1^2}{4}\Vert[\bfDelta^{\hat\bfB};\bfDelta^{\hat\bfGamma}]\Vert_{\Pi}^2&\le
\sf_1\Vert[\bfDelta^{\hat\bfB};\bfDelta^{\hat\bfGamma}]\Vert_{\Pi}+\triangle.\label{lemma:dim:reduction:tr:reg:matrix:dec:eq2}
\end{align}  
Define $H:=(\nicefrac{3\lambda}{2})\Vert\calP_{\bfB^*}(\bfDelta^{\hat\bfB})\Vert_N 
+(\nicefrac{3\tau}{2})\Vert\calP_{\bfGamma^*}(\bfDelta^{\hat\bfGamma})\Vert_1$

We further consider two subcases. 
\begin{description}
\item[Case 2.1:] $\sf_1\Vert[\bfDelta^{\hat\bfB};\bfDelta^{\hat\bfGamma}]\Vert_{\Pi}\le H$. 

In that case, we conclude from \eqref{lemma:dim:reduction:tr:reg:matrix:dec:eq2} that 
$[\bfDelta^{\hat\bfB};\bfDelta^{\hat\bfGamma}]\in\calC_{\bfB^*,\bfGamma^*}(6,\gamma)$.

\item[Case 2.2:] $\sf_1\Vert[\bfDelta^{\hat\bfB};\bfDelta^{\hat\bfGamma}]\Vert_{\Pi}\ge H$. 

Let $G:=\Vert[\bfDelta^{\hat\bfB};\bfDelta^{\hat\bfGamma}]\Vert_{\Pi}$. In that case we obtain $\frac{\sc_1^2}{4}G^2\le2\sf_1 G\Rightarrow G\le\frac{8\sf_1}{\sc_1^2}$. Therefore $H\le8\frac{\sf_1^2}{\sc_1^2}$. From \eqref{lemma:dim:reduction:tr:reg:matrix:dec:eq2}, we obtain that 
\begin{align}
\Vert\calP_{\bfB^*}^\perp(\bfDelta^{\hat\bfB})\Vert_N 
+\Vert\calP_{\bfGamma^*}^\perp(\bfDelta^{\hat\bfGamma})\Vert_1\le4\sf_1 G\le32\frac{\sf_1^2}{\sc_1^2},
\end{align}
which further implies
\begin{align}
\lambda\Vert\bfDelta^{\hat\bfB}\Vert_N+\tau\|\bfDelta^{\hat\btheta}\|_1
\le \frac{2}{3}H+32\frac{\sf_1^2}{\sc_1^2}\le37.4\frac{\sf_1^2}{\sc_1^2}.
\end{align}
Finally, from \eqref{lemma:dim:reduction:tr:reg:matrix:dec:eq2} and $H\le8\frac{\sf_1^2}{\sc_1^2}$,
\begin{align}
\frac{\sc_1^2}{4}\Vert[\bfDelta^{\hat\bfB};\bfDelta^{\hat\btheta}]\Vert_\Pi^2
&\le\sf_1\Vert[\bfDelta^{\hat\bfB};\bfDelta^{\hat\btheta}]\Vert_\Pi+8\frac{\sf_1^2}{\sc_1^2},
\end{align}
which implies \eqref{lemma:dim:reduction:rate:l2:MD} by Young's inequality. 
\end{description}
\end{description}
\end{proof}

\begin{theorem}[Trace-regression with matrix decomposition]\label{thm:tr:reg:matrix:decomp:det}
Suppose that, in addition to (i)-(iv) in Lemma \ref{lemma:dim:reduction:tr:reg:matrix:decomp}, the following condition holds:
\begin{itemize}
\item[\rm{(v)}] If we define
\begin{align}
R&:=\Psi_{\Vert\cdot\Vert_N}(\calP_{\bfB^*}(\bfDelta^{\hat\bfB}))\cdot\mu(\calC_{\bfB^*,\Vert\cdot\Vert_N}(12)),\\
Q&:=\Psi_{\Vert\cdot\Vert_1}(\calP_{\bfGamma^*}(\bfDelta^{\hat\bfGamma}))\cdot\mu(\calC_{\bfGamma^*,\Vert\cdot\Vert_1}(12)),
\end{align}
assume that
$$
	\frac{20}{3}\sqrt{\lambda^2R^2+ \tau^2Q^2} \le \sc_1.
$$
\end{itemize}
Then 
\begin{align}
	\big\|[\bfDelta^{\hat\bfB};\bfDelta^{\hat\bfGamma}]\big\|_\Pi	&\le  80\frac{\sf_1^2}{\sc_1^4}\bigvee
	\left[\frac{4\sf_1}{\sc_1^2}+\frac{28}{\sc_1^2}\sqrt{\lambda^2R^2+\tau^2Q^2}\right]\bigvee
	\left[\frac{4\sf_1}{\sc_1^2}+\frac{40}{\sc_1^2}(\lambda R)\vee(\tau Q)\right],\label{thm:tr:reg:matrix:decomp:det:l2}\\
\lambda\Vert\bfDelta^{\hat\bfB}\Vert_N + \tau\|\bfDelta^{\hat\bfGamma}\big\|_1		& \le 
37.4\frac{\sf_1^2}{\sc_1^2}\bigvee
	\left[\frac{14\sf_1^2}{\sc_1^2}+\frac{32}{\sc_1^2}\lambda^2R^2+\tau^2Q^2\right]\bigvee
	\left[\frac{10\sf_1^2}{\sc_1^2}+\frac{270}{\sc_1^2}(\lambda^2 R^2)\vee(\tau^2 Q^2)\right].\label{thm:tr:reg:matrix:decomp:det:l1}
\end{align} 
\end{theorem}
\begin{proof}
In Lemma \ref{lemma:dim:reduction:tr:reg:matrix:decomp}, if \eqref{lemma:dim:reduction:rate:l2:MD}-\eqref{lemma:dim:reduction:rate:l1:MD} hold there is nothing to prove. We thus need to consider the case (A) for which
$[\bfDelta^{\hat\bfB};\bfDelta^{\hat\bfGamma}]\in\calC_{\bfB^*,\bfGamma^*}(3,\gamma)$ and \eqref{lemma:dim:reduction:tr:reg:MD:contraction} hold or case (B) for which 
$[\bfDelta^{\hat\bfB};\bfDelta^{\hat\bfGamma}]\in\calC_{\bfB^*,\bfGamma^*}(6,\gamma)$ and \eqref{lemma:dim:reduction:tr:reg:MD:contraction:2} hold. 

\emph{Case (A)}. We first claim that, since $[\bfDelta^{\hat\bfB};\bfDelta^{\hat\bfGamma}]\in\calC_{\bfB^*,\bfGamma^*}(3,\gamma)$, we may assume that either $\bfDelta^{\hat\bfB}\in\calC_{\bfB^*,\Vert\cdot\Vert_N}(12)$ or $\bfDelta^{\hat\bfGamma}\in\calC_{\bfGamma^*,\Vert\cdot\Vert_1}(12)$. Indeed, otherwise, 
$
9\gamma\Vert\calP_{\bfB^*}(\bfDelta^{\hat\bfB})\Vert_N 
+9\Vert\calP_{\bfGamma^*}(\bfDelta^{\hat\bfGamma})\Vert_1\le0 
$
implying $\bfDelta^{\hat\bfB}=\bfDelta^{\hat\bfGamma}=0$.

\begin{description}
\item[Case 1:] $\bfDelta^{\hat\bfB}\in\calC_{\bfB^*,\Vert\cdot\Vert_N}(12)$ and $\bfDelta^{\hat\bfGamma}\in\calC_{\bfGamma^*,\Vert\cdot\Vert_1}(12)$.  Decomposability,
$[\bfDelta^{\hat\bfB};\bfDelta^{\hat\bfGamma}]\in\calC_{\bfB^*,\bfGamma^*}(3,\gamma)$ and Cauchy-Schwarz imply
\begin{align}
\frac{\triangle}{3}\le(\nicefrac{\lambda}{2})\Vert\bfDelta^{\hat\bfB}\Vert_N + (\nicefrac{\tau}{2})\|\bfDelta^{\hat\bfGamma}\|_1 \le
\frac{5}{2}\sqrt{\lambda^2R^2+ \tau^2Q^2}
\Vert[\bfDelta^{\hat\bfB};\bfDelta^{\hat\bfGamma}]\Vert_\Pi.
\end{align}
Assuming $5\sqrt{\lambda^2R^2+ \tau^2Q^2} \le \sc_1$, a similar argument in Proposition \ref{prop:suboptimal:rate} implies that
\begin{align}
	\big\|[\bfDelta^{\hat\bfB};\bfDelta^{\hat\btheta}]\big\|_\Pi	&\le \frac{4}{\sc_1^2}\sf_1+\frac{6}{\sc_1^2}\sqrt{\lambda^2 R^2+\tau^2Q^2},\\
	\lambda\Vert\bfDelta^{\hat\bfB}\Vert_N 
	+\tau\|\bfDelta^{\hat\bfGamma}\|_1&\le 5\frac{\sf_1^2}{\sc_1^2}+\frac{20}{\sc_1^2}(\lambda^2 R^2+\tau^2 Q^2). 
\end{align}

\item[Case 2:] $\bfDelta^{\hat\bfB}\notin\calC_{\bfB^*,\Vert\cdot\Vert_N}(12)$ and $\bfDelta^{\hat\bfGamma}\in\calC_{\bfGamma^*,\Vert\cdot\Vert_1}(12)$. As 
$[\bfDelta^{\hat\bfB};\bfDelta^{\hat\bfGamma}]\in\calC_{\bfB^*,\bfGamma^*}(3,\gamma)$,
\begin{align} 
9\gamma\Vert\calP_{\bfB^*}(\bfDelta^{\hat\bfB})\Vert_N+
\Vert\calP^\perp_{\bfGamma^*}(\bfDelta^{\hat\bfGamma})\Vert_1\le3\Vert\calP_{\bfGamma^*}(\bfDelta^{\hat\bfGamma})\Vert_1,
\end{align}
implying 
\begin{align}
\frac{\triangle}{3}\le(\nicefrac{\lambda}{2})\Vert\bfDelta^{\hat\bfB}\Vert_N + (\nicefrac{\tau}{2})\|\bfDelta^{\hat\bfGamma}\|_1 \le 20\tau\|\bfDelta^{\hat\bfGamma}\|_1\le\frac{10}{3}\tau Q\Vert\bfDelta^{\hat\bfGamma}\Vert_\Pi.
\end{align}
Assuming $\frac{20}{3}\tau Q\le \sc_1$, we obtain 
\begin{align}
	\big\|[\bfDelta^{\hat\bfB};\bfDelta^{\hat\btheta}]\big\|_\Pi	&\le \frac{4}{\sc_1^2}\sf_1+\frac{40}{\sc_1^2}\tau Q,\\
	\lambda\Vert\bfDelta^{\hat\bfB}\Vert_N 
	+\tau\|\bfDelta^{\hat\bfGamma}\|_1&\le \frac{10\sf_1^2}{3\sc_1^2}+\frac{270}{\sc_1^2}\tau^2 Q^2. 
\end{align}

\item[Case 3:] $\bfDelta^{\hat\bfB}\in\calC_{\bfB^*,\Vert\cdot\Vert_N}(12)$ and $\bfDelta^{\hat\bfGamma}\notin\calC_{\bfGamma^*,\Vert\cdot\Vert_1}(12)$. Similarly to Case 2, 
assuming $\frac{20}{3}\lambda R\le \sc_1$, we obtain 
\begin{align}
	\big\|[\bfDelta^{\hat\bfB};\bfDelta^{\hat\btheta}]\big\|_\Pi	&\le \frac{4}{\sc_1^2}\sf_1+\frac{40}{\sc_1^2}\lambda R,\\
	\lambda\Vert\bfDelta^{\hat\bfB}\Vert_N 
	+\tau\|\bfDelta^{\hat\bfGamma}\|_1&\le \frac{10\sf_1^2}{3\sc_1^2}+\frac{270}{\sc_1^2}\lambda^2 R^2. 
\end{align}
\end{description}

\emph{Case (B).} Without further conditions, $\big\|[\bfDelta^{\hat\bfB};\bfDelta^{\hat\btheta}]\big\|_\Pi\le\frac{4}{\sc_1^2}(\sf_1+\triangle)$. By dividing in subcases as in Case (A), one obtains the bounds
\begin{align}
	\big\|[\bfDelta^{\hat\bfB};\bfDelta^{\hat\btheta}]\big\|_\Pi	&\le \left[\frac{4}{\sc_1^2}\sf_1+\frac{28}{\sc_1^2}\sqrt{\lambda^2 R^2+\tau^2 Q^2}\right]\bigvee
	\left[\frac{4}{\sc_1^2}\sf_1+\frac{12}{\sc_1^2}\tau Q\right]\bigvee
\left[\frac{4}{\sc_1^2}\sf_1+\frac{12}{\sc_1^2}\lambda R\right],\\
	\lambda\Vert\bfDelta^{\hat\bfB}\Vert_N 
	+\tau\|\bfDelta^{\hat\bfGamma}\|_1&\le 
	\left[14\frac{\sf_1^2}{\sc_1^2}+\frac{32}{\sc_1^2}(\lambda^2R^2+\tau^2Q^2)\right]\bigvee
	\left[\frac{4}{\sc_1^2}\sf_1+\frac{16}{\sc_1^2}\tau^2 Q^2\right]
	\bigvee
	\left[\frac{4}{\sc_1^2}\sf_1+\frac{16}{\sc_1^2}\lambda^2 R^2\right].
\end{align}
The proof is finished.
\end{proof}

\subsection{Properties for subgaussian $(\bfX,\xi)$}
\label{ss:subgaussian:designs:tr:reg:matrix:decomp}
 
\begin{tcolorbox}
Throughout this section, we additionally assume $(\bfX,\xi)\in\mdR^p\times\re$ is a centered (not necessarily independent) random pair satisfying Assumption \ref{assump:distribution:subgaussian} and $\{(\bfX_i,\xi_i)\}_{i\in[n]}$ is an iid copy of $(\bfX,\xi)$. Moreover, $\calR$ is any norm on $\mdR^p$ and $\calQ$ is any  norm on
$\re^n$.
\end{tcolorbox}

In this section we prove that all properties of Definitions \ref{def:design:pro:multi:PP} and  \ref{def:design:pro:multi:reg:ATP:MP} are satisfied with high-probability. Recall that $\TP$ in Definition \ref{def:design:property} was already proved in Proposition \ref{prop:gen:TP} in Section \ref{ss:subgaussian:designs}.

We first prove $\PP$. We will use the next lemma which is immediate from Theorem \ref{thm:product:process} in the Appendix for the product process \eqref{equation:def:product:process} over the linear class $\calF=\{\llangle\cdot,\bfV\rrangle:\bfV\in\mdR^p\}$. 

\begin{lemma}\label{lemma:mult:process}
Let $\calB_1$ and $\calB_2$ be bounded subsets of 
$\mbB_\Pi$. There exists universal numerical constant $C>0$, such that, for any $n\ge1$ and $t>0$, with probability at least $1-e^{-t}$,
\begin{align}
\sup_{[\bfV;\bfW]\in \calB_1\times\calB_2}\left|\llangle\bfV,\bfW\rrangle_n-\llangle\bfV,\bfW\rrangle_\Pi\right|&\le \frac{C}{n}\mathscr{G}\big(\frS^{1/2}(\calB_1))\mathscr{G}\big(\frS^{1/2}(\calB_2))\\
&+ \frac{CL}{\sqrt{n}}\left[\mathscr{G}\big(\frS^{1/2}(\calB_1)+\mathscr{G}\big(\frS^{1/2}(\calB_2)\right]\\
&+CL^2\left(\frac{t}{n}+\sqrt{\frac{t}{n}}\right).
\end{align}
\end{lemma}

\begin{proposition}[$\PP$]\label{prop:PP}
There exists universal numerical constant $C>0$, such that for all $\delta\in(0,1]$
and $n\in\mathbb{N}$, with probability at least
$1-\delta$, the following property holds: for all
$[\bfV;\bfW]\in(\mdR^p)^2$,
\begin{align}
\left|\llangle\bfW,\bfV\rrangle_n-\llangle\bfW,\bfV\rrangle_\Pi\right|&\le 
CL^2\left(\sqrt{\frac{\log(1/\delta)}{n}}+\frac{\log(1/\delta)}{n}\right)\Vert\bfV\Vert_\Pi\Vert\bfW\Vert_\Pi\\
&+CL\frac{\mathscr
G\left(\calR(\bfV)\frS^{1/2}(\mbB_\calR)\cap\Vert\bfV\Vert_\Pi\mbB_F\right)}{\sqrt{n}}\Vert\bfW\Vert_\Pi\\
&+CL\frac{\mathscr
G\left(\calQ(\bfW)\frS^{1/2}(\mbB_\calQ)\cap\Vert\bfW\Vert_\Pi\mbB_F\right)}{\sqrt{n}}\big\|\bfV\big\|_\Pi
\\
&+C\frac{\mathscr
G\left(\calR(\bfV)\frS^{1/2}(\mbB_\calR)\cap\Vert\bfV\Vert_\Pi\mbB_F\right)}{\sqrt{n}}\cdot
\frac{\mathscr
G\left(\calQ(\bfW)\frS^{1/2}(\mbB_\calQ)\cap\Vert\bfW\Vert_\Pi\mbB_F\right)}{\sqrt{n}}.
\end{align}
\end{proposition}
\begin{proof}
Let $R_1,R_2>0$ and define the sets
\begin{align}
V_1:=\{\bfV:\Vert\bfV\Vert_\Pi\le 1,\calR(\bfV)\le R_1\},\quad
V_{2}:=\{\bfW:\Vert\bfW\Vert_\Pi\le1,\calQ(\bfW)\le R_2\}.
\end{align}
Note that, 
\begin{align}
\mathscr{G}\left(\frS^{1/2}(V_1)\right)&\le R_1\mathscr G\left(\frS^{1/2}(\mbB_\calR)\cap R_1^{-1}\mbB_F\right),\quad\quad
\mathscr{G}\left(\frS^{1/2}(V_2)\right)&\le R_2\mathscr G\left(\frS^{1/2}(\mbB_\calQ)\cap R_2^{-1}\mbB_F\right).
\end{align}
Define for convenience the functions
\begin{align}
g(r):=\frac{\mathscr G\left(\frS^{1/2}(\mbB_\calR)\cap r^{-1}\mbB_F\right)}{\sqrt{n}}r,
\quad\bar g(\bar r):=\frac{\mathscr G\left(\frS^{1/2}(\mbB_\calQ)\cap {\bar r}^{-1}\mbB_F\right)}{\sqrt{n}} \bar r.
\end{align}
By Lemma \ref{lemma:mult:process}, there is universal constant $C>0$ such that, for any $R_1,R_2>0$ and $\delta\in(0,1]$, with probability at least $1-\delta$,
\begin{align}
\sup_{[\bfV;\bfW]\in V_1\times V_2}|\llangle\bfW,\bfV\rrangle_n-\llangle\bfW,\bfV\rrangle_\Pi|
&\le Cg(R_1)\bar g(R_2) + CL[g(R_1) + \bar g(R_2)]\\
&+CL^2\left(\sqrt{\frac{\log(1/\delta)}{n}}+\frac{\log(1/\delta)}{n}\right).
\end{align}
We can now use a bi-parameter peeling lemma for subexponential tails (analogously to Lemma \ref{lemma:peeling:2dim} for subgaussian tails) with set $V:=\mbB_\Pi\times\mbB_\Pi$, functions
$M(\bfV,\bu):=-|\llangle\bfW,\bfV\rrangle_n-\llangle\bfW,\bfV\rrangle_\Pi|$, 
$h_1(\bfV,\bfW):=\calR(\bfV)$, $h_2(\bfV,\bfW):=\calQ(\bfW)$, $g$ and $\bar g$ and constants $c:=1$ and $b_1:=CL^2(\frac{1}{\sqrt{n}}+\frac{1}{n})$. The claim follows from such lemma and the homogeneity of norms.
\end{proof}

The next proposition follows from the concentration bound for the multiplier process (Theorem \ref{thm:mult:process} in the Appendix) and a peeling lemma. The proof is similar to Proposition \ref{prop:gen:MP} so we omit the details.

\begin{proposition}[$\MP$]\label{prop:gen:MP:matrix:decomp}
For $t,s>0$, let $\triangle(t,s)$ as defined in Proposition \ref{prop:gen:MP}. There exists universal constant $C>0$, $c_0,c\ge2$ such that for all $n\in\mathbb{N}$ and all $\delta\in(0,1/c]$ and $\rho\in(0,1/c_0]$, with probability at least
$1-\delta-\rho$, the following property holds: for all
$[\bfV;\bfW]\in(\mdR^p)^2$,
\begin{align}
\langle\bxi^{(n)},\frX^{(n)}(\bfV+\bfW)\rangle &\le C\sigma L\cdot\triangle(\nicefrac{1}{\delta},\nicefrac{1}{\rho})\cdot\Vert[\bfV;\bfW]\Vert_\Pi\\
&+C\sigma L\left[1+\frac{\sqrt{\log(1/\rho)}}{\sqrt{n}}\right]\frac{\mathscr G\big(\frS^{1/2}(\mbB_{\calR}))}{\sqrt{n}}\calR(\bfV)\\
&+C\sigma L\left[1+\frac{\sqrt{\log(1/\rho)}}{\sqrt{n}}\right]\frac{\mathscr G\big(\frS^{1/2}(\mbB_{\calQ}))}{\sqrt{n}}\calQ(\bfW).
\end{align} 
\end{proposition}

\subsection{Proof of Theorem \ref{thm:tr:reg:matrix:decomp}}
\label{ss:proof:tr:reg:matrix:decomp}

We will apply Sections \ref{ss:subgaussian:designs} and \ref{ss:subgaussian:designs:tr:reg:matrix:decomp} to 
the nuclear and $\ell_1$ norms in $\mdR^p$. By Assumption \ref{assump:low:spikeness}, $\frS$ is the identity operator. As before, 
$\mathscr{G}(\mbB_{\Vert\cdot\Vert_N})\lesssim\sqrt{d_1}+\sqrt{d_2}$
and
$\mathscr{G}(\mbB_{\Vert\cdot\Vert_1})\lesssim\sqrt{\log p}$.  

We first use Proposition \ref{prop:gen:TP} with $\calR=\Vert\cdot\Vert_N$, 
$\epsilon=\frac{c}{C^2L^2}$ for sufficiently small $c\in(0,1)$ and assuming
\begin{align}
\delta\ge \exp\left(-c_1\frac{n}{L^4}\right)\label{equation:delta:n:tr:reg:matrix:decomp}
\end{align}
for large enough constant $c_1>0$. It follows that for 
$\sa_1\in(0,1)$ an universal constant and
$$
\sa_2\asymp L\left(\sqrt{\frac{d_1}{n}}+\sqrt{\frac{d_2}{n}}\right),
$$
on an event $\Omega_1$ of probability at least $1-\delta/4$, property 
$\TP_{\Vert\cdot\Vert_N}(\sa_1;\sa_2)$ is satisfied. Similarly, assuming \eqref{equation:delta:n:tr:reg:matrix:decomp}, we have that for universal constant $\bar\sa_1\in(0,1)$ and 
$$
\bar\sa_2\asymp L\sqrt{\frac{\log p}{n}},
$$
on an event $\Omega_1'$ of probability at least $1-\delta/4$, property 
$\TP_{\Vert\cdot\Vert_1}(\bar\sa_1;\bar\sa_2)$ is satisfied.

From Proposition \ref{prop:PP}, for every $\delta\in(0,1)$ and taking
$$
\sb_1\asymp L^2\left(\sqrt{\frac{\log(1/\delta)}{n}}+\frac{\log(1/\delta)}{n}\right),
\quad
\sb_2\asymp L\left(\sqrt{\frac{d_1}{n}}+\sqrt{\frac{d_2}{n}}\right),\quad
\sb_3\asymp L\sqrt{\frac{\log p}{n}}
$$
and $\sb_4=\frac{\sb_2\sb_3}{L^2}$ (recall $L\ge1$), we have that on an event $\Omega_2$ of probability at least $1-\delta/4$, property $\PP_{\Vert\cdot\Vert_N,\Vert\cdot\Vert_1}(\sb_1;\sb_2;\sb_3;\sb_4)$ is satisfied. 

From Lemma \ref{lemma:ATP}, by enlarging $c_1$ in \eqref{equation:delta:n:tr:reg:matrix:decomp} if necessary, for $\sc_1\in(0,1)$ an universal constant and
$$
\sc_2\asymp L\left(\sqrt{\frac{d_1}{n}}+\sqrt{\frac{d_2}{n}}\right),\quad
\sc_3\asymp L\sqrt{\frac{\log p}{n}},
$$
$\ATP_{\Vert\cdot\Vert_N,\Vert\cdot\Vert_1}(\sc_1;\sc_2;\sc_3)$ is satisfied on $\Omega_1\cap\Omega_1'\cap\Omega_2$.  

We now use Proposition \ref{prop:gen:MP:matrix:decomp} (with 
$\delta=\rho$). By enlarging $c_1$ in \eqref{equation:delta:n:tr:reg:matrix:decomp} if necessary, if we take
$$
\sf_1\asymp\sigma L\frac{1+\sqrt{\log(1/\delta)}}{\sqrt{n}},\quad
\sf_2\asymp\sigma L\left(\sqrt{\frac{d_1}{n}}+\sqrt{\frac{d_2}{n}}\right), \quad
\sf_3\asymp \sigma L\sqrt{\frac{\log p}{n}}, 
$$
we have by Proposition \ref{prop:gen:MP:matrix:decomp} that on an event $\Omega_3$ of probability at least $1-\delta/4$,  
$\MP_{\Vert\cdot\Vert_N,\Vert\cdot\Vert_1}(\sf_1;\sf_2;\sf_3)$ is satisfied.

By an union bound and enlarging constants, for $\delta$ satisfying \eqref{equation:delta:n:tr:reg:matrix:decomp}, on the event 
$\Omega_1\cap\Omega_1'\cap\Omega_2\cap\Omega_3$ of probability at least $1-\delta$, all properties $\TP$, $\PP$, $\ATP$ and 
$\MP$ hold with constants as specified above. We assume such event is realized and invoke Theorem \ref{thm:tr:reg:matrix:decomp:det}. It is straightforward to check 
$\rm{(iii)}$ by the definitions of $\tau$ and $\lambda$ in Theorem \ref{thm:tr:reg:matrix:decomp}. Item $\rm{(iv)}$ is tantamount requiring $\delta\ge \exp(-c_2n/\sigma^2L^2)$ for universal constant $c_2>0$. 
We now check item $\rm{(v)}$. In our setting, $R\le\sqrt{r}\mu(\bfB^*)$ with $\mu(\bfB^*):=\mu(\calC_{\bfB^*,\Vert\cdot\Vert_N}(12))$ and $Q\le\sqrt{s}\mu(\bfGamma^*)$ with $\mu(\bfGamma^*):=\mu(\calC_{\bfGamma^*,\Vert\cdot\Vert_1}(12))$. Condition in item $\rm{(iv)}$ requires
\begin{align}
C\sigma^2L^2\left(\frac{d_1}{n}+\frac{d_2}{n}\right)r\mu^2(\bfB^*)
+
C\left(\sigma^2L^2\frac{\log p}{n}+\frac{(\sa^*)^2}{n}\right)s\mu^2(\bfGamma^*)
<1.
\end{align}
As the feature is isotropic, $\mu(\bfB^*)=\mu(\bfGamma^*)=1$. Thus the above condition holds by assumption. With all conditions of Theorem \ref{thm:tr:reg:matrix:decomp:det} taking place, the rate in Theorem \ref{thm:tr:reg:matrix:decomp} follows.

\subsection{Proof of Proposition \ref{prop:lower:bound:tr:matrix:decomp}}

In trace-regression with matrix decomposition, the design is random. Up to conditioning on the feature data, the proof of Proposition \ref{prop:lower:bound:tr:matrix:decomp} follows almost identical arguments in the proof of Theorem 2 in \cite{2012agarwal:negahban:wainwright} for the matrix decomposition problem \eqref{equation:matrix:decomp} with identity design $(\frX=I)$. We present a sketch here for completeness.  

We prepare the ground so to apply Fano's method. For any 
$\bfTheta:=[\bfB,\bfGamma]\in(\mdR^p)^2$, we define for convenience
$
\Vert\bfTheta\Vert_F^2:=\Vert\bfB\Vert_F^2
+\Vert\bfGamma\Vert_F^2.
$
Given $\eta>0$ and $M\in\mathbb{N}$, a $\eta$-packing of $\calA(r,s,\sa^*)$ of size $M$ is a finite subset 
$\calA=\{\bfTheta_1,\ldots,\bfTheta_M\}$ of 
$\calA(r,s,\sa^*)$ satisfying 
$
\Vert\bfTheta_\ell-\bfTheta_k\Vert_F\ge\eta
$
for all $\ell\neq k$. For model \eqref{equation:str:eq:trace:reg:matrix:decomp}, let $y_1^n:=\{y_i\}_{i\in[n]}$ and $\bfX_1^n:=\{\bfX_i\}_{i\in[n]}$. 
For any $k\in[M]$ and $i\in[n]$, we will denote by 
$
P^k_{y_1^n|\bfX_1^n}
$
(or $P^k_{y_i|\bfX_1^n}$)
the conditional distribution of $y_1^n$ (or $y_i$) given 
$\bfX_1^n$ corresponding to the model \eqref{equation:str:eq:trace:reg:matrix:decomp} with parameters $\bfTheta_k=[\bfB_k,\bfGamma_k]$ belonging to the packing $\calA$. Being normal distributions, they are mutually absolute continuous. We denote by $\KL(\probn\Vert\bfQ)$ the Kullback-Leibler divergence between probability measures $\probn$ and $\bfQ$. In that setting, Fano's method assures that
\begin{align}
\inf_{\hat\bfTheta}\sup_{\bfTheta^*\in\calA(r,s,\sa^*)}\prob_{\bfTheta^*}
\left\{\Vert\hat\bfTheta-\bfTheta^*\Vert_F^2
\ge\eta^2\right\}\ge
1-\frac{\frac{1}{\binom{M}{2}}\sum_{k,\ell=1}^M\esp_{\bfX_1^n}\KL\left(P^k_{y_1^n|\bfX_1^n}\big\Vert P^\ell_{y_1^n|\bfX_1^n}\right)+\log 2}{\log M}.
\label{equation:Fano}
\end{align}
The proof follows from an union bound on the following two separate lower bounds. 

\emph{Lower bound on the low-spikeness bias}. It is sufficient to give a lower bound on $\calA(1,s,\sa^*)$. We define the subset $\calA$ of $\calA(1,s,\sa^*)$ with size $M=4$ by
\begin{align}
\calA:=\{[\bfB^*,-\bfB^*],[-\bfB^*,\bfB^*],
(\nicefrac{1}{2})[\bfB^*,-\bfB^*],
[\bf0,\bf0]\}
\end{align}
using the matrix $\bfB^*\in\mdR^p$ defined by
$$
\bfB^*:=\frac{\sa^*}{\sqrt{n}}
\left[\begin{array}{c}
1\\
0\\
\vdots\\
0
\end{array}\right]
\underbrace{\left[\begin{array}{ccccccc}
1 &
1 &
\cdots &
1 &
0 &
\cdots &
0 
\end{array}\right]^\top}_{\boldf^\top},
$$
where $\boldf\in\re^{d_2}$ has $s$ unit coordinates. It is easy to check that $\calA$ is a $\eta$-packing of 
$\calA(1,s,\sa^*)$ with $\eta=c_0\sa^*\sqrt{\frac{s}{n}}$ for some constant $c_0>0$. For any element $[\bfB_k,\bfGamma_k]$ of $\calA$, $\bfB_k+\bfGamma_k=0$ implying $P^k_{y_1^n|\bfX_1^n}\sim\calN_n(0,n\sigma^2\bfI)$. Hence, for any $k\neq\ell$, 
$$
\KL\left(P^k_{y_1^n|\bfX_1^n}\big\Vert P^\ell_{y_1^n|\bfX_1^n}\right)=0.
$$
From \eqref{equation:Fano}, one obtains a lower bound with rate of order $\sa^*\sqrt{\frac{s}{n}}$ with positive probability. 

\emph{Lower bound on the estimation error}. From the packing constructions in Lemmas 5 and 6 in \cite{2012agarwal:negahban:wainwright}, one may show that, for $d_1,d_2\ge10$, $\sa^*\ge 32\sqrt{\log p}$, $s<p$ and any $\eta>0$, there exists a $\eta$-packing $\calA=\{\bfTheta_k\}_{k\in[M]}$ of $\calA(r,s,\sa^*)$ with size 
\begin{align}
M\ge \frac{1}{4}\exp\left\{\frac{s}{2}\log\frac{p-s}{s/2}+\frac{r(d_1+d_2)}{256}\right\}
,\label{equation:M:lower:bound}
\end{align}
satisfying $\Vert\bfTheta_k\Vert_F\le3\eta$ for any $k\in[M]$. By independence of $\{\xi_i\}_{i\in[n]}$ and $\bfX_1^n$ and isotropy,
\begin{align}
\esp_{\bfX_1^n}\KL\left(P^k_{y_1^n|\bfX_1^n}\big\Vert P^\ell_{y_1^n|\bfX_1^n}\right)=\sum_{i\in[n]}\esp_{\bfX_1^n}\KL\left(P^k_{y_i|\bfX_i}\big\Vert P^\ell_{y_i|\bfX_i}\right)
=\frac{n\Vert\bfTheta_k-\bfTheta_k\Vert_F^2}{2\sigma^2}\le \frac{18n}{\sigma^2}\eta^2.
\end{align}
From \eqref{equation:Fano} and \eqref{equation:M:lower:bound}, one then checks that tacking
\begin{align}
\eta^2:=c_o\sigma^2\left\{\frac{r(d_1+d_2)}{n}
+\frac{s}{n}\log\left(\frac{p-s}{s/2}\right)\right\},
\end{align}
for some constant $c_0>0$, one obtains a lower bound with rate of order $\eta^2$ with positive probability.

\section{Proof of Theorem \ref{thm:robust:MC}}
\label{s:proof:thm:robust:matrix:completion}

Recall Definition \ref{def:design:property} in Section \ref{ss:cones:RE}. We will work with a variation of $\MP$. 

\begin{definition}[$\MP$]\label{def:MP:MC}
Let $\calR$ be a norm on $\mdR^p$ and $\calQ$ be a norm on $\re^n$. Given non-negative numbers $(\sf_1,\sf_2,\sf_4)$, we say $(\frX,\bxi)$ satisfies 
$\MP_{\calR,\calQ}(\sf_1,\sf_2,\sf_4)$ if 
\vspace{-3pt}
\begin{align}
\forall[\bfV,\bu]\in\mdR^p\times\re^n,\quad|\langle\bxi^{(n)},\frM^{(n)}(\bfV, \bu)\rangle|\le 
\sf_1\Vert\bu\Vert_2
+\sf_2\calR(\bfV)
+\sf_4\calQ(\bu).
\end{align}
\end{definition}

We will also need additional cone definitions.
\begin{definition}\label{def:cones:RMC}
Let $\calR$ be a norm on $\mdR^p$ and $\calQ$ be a norm on $\re^n$. Given 
$\sh>0$, we define the cones
\begin{align}
\frC(\sh)&:=\left\{\bfV:
\sh\Vert\bfV\Vert_\infty\le \Vert\bfV\Vert_\Pi
\right\},\\
\mcC_{\calR^{(p)}}(\sh)&:=\left\{\bfV:\sh\Vert\bfV\Vert_\infty\calR^{(p)}(\bfV)\le \Vert\bfV\Vert_\Pi^2
\right\},\\
\mdC_{\calQ}(\sh)&:=\left\{[\bfV,\bu]:\sh\Vert\bfV\Vert_\infty\calQ(\bu)\le \Vert[\bfV,\bu]\Vert_\Pi^2
\right\}.
\end{align}
\end{definition}

The aim of this section is to prove that, with high probability, 
$\ATP$ holds over an appropriate cone $\mbC$ and $\MP_{\calR^{(p)},\Vert\cdot\Vert_\sharp}$ holds everywhere. Throughout this section, we additionally assume $(\bfX,\xi)\in\mdR^p\times\re$ satisfies Assumption \ref{assump:distribution:MC} and $\{(\bfX_i,\xi_i)\}_{i\in[n]}$ is an iid copy of $(\bfX,\xi)$. Given a compact set $V\subset\mdR^p$, we shall need the Gaussian complexity 
$$
\calG(V):=\esp\left[\mathscr G\left(\frX\left(V\right)\right)\right]
=\esp\left[\esp\left[\sup_{\bfV\in V}\sum_{i=1}^n\frX_i(\bfV)\bxi_i\Bigg|\frX\right]\right],
$$
where $\bxi\sim\calN(\bold{0},\bfI_n)$ is independent of $\bfX$. Define also
$\tilde\calG(V):=\sqrt{p}\calG(V)$. We note that, by the dual-norm inequality, given $\bfV\in\mdR^p$, 
$
|\llangle\bfX,\bfV\rrangle|\le\Vert\bfV\Vert_\infty. 
$
In particular, $\Vert\bfV\Vert_\Pi=\esp^{1/2}\llangle\bfX,\bfV\rrangle^2\le\Vert\bfV\Vert_\infty$.

The next result states a \emph{lower bound} on the quadratic process. The main tool to prove this lemma is a \emph{one-sided} version of a Bernstein's inequality for bounded processes due to Bousquet \cite{2013boucheron:lugosi:massart}. 
\begin{lemma}\label{lemma:constrained:TP:N}
Let $\mbS_\infty^{p}:=\{\bfV\in\mdR^p:\Vert\bfV\Vert_\infty=1\}$ and $V\subset R\mbB_\Pi\cap \mbS_\infty^{p}$ for some $R\in[0,1]$. Then, for any $n\ge1$ and $t\ge0$, with probability at least $1-\exp(-t)$,
\begin{align}
\sup_{\bfV\in V}\frac{1}{n}\sum_{i\in[n]}\left[\Vert\bfV\Vert_\Pi^2-\llangle\bfX_i,\bfV\rrangle^2\right]\le2\sqrt{2\pi}\frac{\calG(V)}{n}
+2R\sqrt{\frac{t}{n}}+\frac{t}{3n}.
\end{align}
\end{lemma}
\begin{proof}
Define
$
X_{i,\bfV}:=\Vert\bfV\Vert_\Pi^2-\llangle\bfX_i,\bfV\rrangle^2.
$
By the symmetrization inequality (e.g. Exercise 11.5 in \cite{2013boucheron:lugosi:massart}), we have 
\begin{align}
\esp\left[\sup_{\bfV\in V}\sum_{i\in[n]}X_{i,\bfV}\right]
\le2\esp\left[\sup_{\bfV\in V}\sum_{i\in[n]}\epsilon_i X_{i,\bfV}\right], 
\end{align} 
where $\{\epsilon_i\}_{i\in[n]}$ are iid  Rademacher variables independent of $X:=\{\bfX_i\}_{i\in[n]}$. One may bound the Rademacher complexity by the Gaussian complexity as
\begin{align}
\esp\left[\sup_{\bfV\in V}\sum_{i\in[n]}\epsilon_iX_{i,\bfV}\bigg|X\right]\le\sqrt{\frac{\pi}{2}}\esp\left[\sup_{\bfV\in V}\sum_{i\in[n]}g_iX_{i,\bfV}\bigg|X\right],
\end{align}
where $\{g_i\}_{i\in[n]}$ is an iid $\calN(0,1)$ sequence independent of $\{\bfX_i\}_{i\in[n]}$. 

We will now use an standard argument via Slepian's inequality over the randomness of 
$\{g_i\}_{i\in[n]}$ for the process $\bfV\mapsto Z_{\bfV}:=\sum_{i\in[n]}g_iX_{i,\bfV}$. One has
\begin{align}
\esp[(Z_{\bfV}-Z_{\bfV'})^2|X]
=\sum_{i\in[n]}\llangle\bfX_i,\bfV-\bfV'\rrangle^2\llangle\bfX_i,\bfV+\bfV'\rrangle^2
\le4\Vert\frX(\bfV-\bfV')\Vert_2^2.
\end{align}
Define the Gaussian process $W_{\bfV}:=2\langle\frX(\bfV),\bxi\rangle$ with $\bxi\sim\calN_n(0,\bfI_n)$ independent of $\{\bfX_i\}_{i\in[n]}$. Under the above conditions, Slepian's inequality implies 
\begin{align}
\esp\left[\sup_{\bfV\in V}Z_{\bfV}\bigg|X\right]
\le\esp\left[\sup_{\bfV\in V}W_{\bfV}\bigg|X\right] 
=2\mathscr G(\frX(V)).
\end{align}
Define $Z:=\sup_{\bfV\in V}\sum_{i\in[n]}X_{i,\bfV}$. The previous displays imply 
$\esp Z \le2\sqrt{2\pi}\esp[\mathscr G(\frX(V))]$.

We now establish concentration. For all $i\in[n]$ and $\bfV\in V$, 
$
X_{i,\bfV}\le R^2\le1.  
$
Since, for each $\bfV\in V$, $\{X_{i,\bfV}\}_{i\in[n]}$ are independent and identically distributed, we may apply Bousquet's
 inequality (e.g. Corollary 12.2 in \cite{2013boucheron:lugosi:massart}). We thus get that, for all $t\ge0$, with probability at least $1-\exp(-t)$,
\begin{align}
Z\le\esp Z + \sigma\sqrt{2t} + \frac{t}{3}, 
\end{align}
where $\sigma^2:=\sup_{\bfV\in V}\sum_{i\in[n]}\esp[X_{i,\bfV}^2]$. Since for $\bfV\in V$, $\Vert\bfV\Vert_\Pi\vee\llangle\bfX_i,\bfV\rrangle\le\Vert\bfV\Vert_\infty\le1$, it is easy to check that
$
\esp(\llangle\bfX_i,\bfV\rrangle^2-\Vert\bfV\Vert_\Pi^2)^2
=\esp\llangle\bfX_i,\bfV\rrangle^4 + \Vert\bfV\Vert_\Pi^4
\le 2R^2. 
$
In particular, 
\begin{align}
\sigma^2=\sup_{\bfV\in V}\sum_{i\in[n]}\esp[X_{i,\bfV}^2]
\le2R^2n.
\end{align}
This finishes the proof.
\end{proof}

The next proposition is proved using the previous lemma and a peeling argument. 
\begin{proposition}[$\TP$]\label{prop:TP:MC}
There are universal constants $\sc_0,c_0\in(0,1)$ for which the following holds. Let $\delta\in(0,1)$ and define the number
\begin{align}
\sh_\delta := \frac{\tilde\calG(\mbB_{\calR})}{c_0 n}\bigvee\frac{\sqrt{1+\log(1/\delta)}}{c_0 \sqrt{n}},\label{prop:TP:MC:h}  
\end{align}
and the cone
$
\mbC(\delta):=\frC\left(\sh_\delta\right)\bigcap \mcC_{\calR^{(p)}}\left(\sh_{\delta}\right). 
$
Then, with probability at least $1-\delta$ the design operator $\frX$ satisfies 
\begin{align}
\inf_{\bfV\in \mbC(\delta)}\frac{\Vert\frX^{(n)}(\bfV)\Vert_2}{\Vert\bfV\Vert_\Pi}\ge \sc_0.
\end{align}
\end{proposition}
\begin{proof}
Fix $\alpha>1$ and $\delta\in(0,1)$. We define the set $S(\delta):=\mbC(\delta)\cap\mbS_\infty$ and, for all $\ell\in\mathbb{N}^*$, the set
\begin{align}
S_\ell(\delta)&:=\left\{\bfV\in\mbS_\infty^p:
\alpha^{\ell-1}\sh_\delta\le \Vert\bfV\Vert_\Pi<\alpha^\ell\sh_\delta,\quad
\sh_\delta\calR^{(p)}(\bfV)\le \Vert\bfV\Vert_\Pi^2
\right\}.
\end{align}

Note that for all $\ell\in\mathbb{N}^*$, 
$S_\ell(\delta)\subset r\mbB_{\calR}$ with $r:=\sqrt{p}\alpha^{2\ell}\sh_\delta$. Hence, 
$$
\calG(S_\ell(\delta))=\esp[\mathscr G(\frX(S_\ell(\delta)))]
\le \sqrt{p}\alpha^{2\ell}\sh_\delta \esp[\mathscr G(\frX(\mbB_\calR))] = \alpha^{2\ell}\sh_\delta\tilde\calG(\mbB_\calR).
$$

For any $\ell\in\mathbb{N}^*$, $S_\ell(\delta)\subset\mbS_\infty^p\cap r_\ell\mbB_\Pi$ with $r_\ell:=\min\{1,\alpha^\ell\sh_\delta\}$. Hence, by Lemma \ref{lemma:constrained:TP:N} and the previous displayed bound, it holds that, for any $\ell\in\mathbb{N}^*$ and $t>0$, on an event $\calE_\ell(t)$ of probability $\ge1-e^{-t}$, for all $\bfV\in S_\ell(\delta)$, 
\begin{align}
\Vert\bfV\Vert_\Pi^2-\Vert\frX^{(n)}(\bfV)\Vert_2^2&\le2\sqrt{2\pi}\frac{\alpha^{2\ell}\sh_\delta\tilde\calG(\mbB_\calR)}{n} + 2r_\ell\sqrt{\frac{t}{n}}+\frac{t}{3n}. 
\end{align} 

Let $c\in(0,1)$ absolute constant to be determined. By an union bound the event 
$
\calE:=\cap_{\ell=1}^\infty\calE_\ell\left(
cn\alpha^{2\ell}\sh_\delta^2
\right)
$
is such that 
\begin{align}
\prob(\calE^c)\le \sum_{\ell=1}^\infty\exp\left(-cn\alpha^{2\ell}\sh_\delta^2\right)
\le \sum_{\ell=1}^\infty\exp\left(-2cn\log(\alpha)\sh_\delta^2\cdot\ell\right)
\le \frac{\exp\left(-2cn\log(\alpha)\sh_\delta^2\right)}{1-\exp\left(-2cn\log(\alpha)\sh_\delta^2\right)}.\label{proof:prop:TP:MC:eq1} 
\end{align}
Next, we assume that
$$
\sh_\delta \ge \frac{\tilde\calG(\mbB_\calR)}{c n}
\bigvee
\sqrt{\frac{\log(2)}{2c\log(\alpha)\cdot n}}
\bigvee
\sqrt{\frac{\log(2/\delta)}{2c\log(\alpha) \cdot n}}. 
$$
In particular, $\prob(\calE^c)\le\delta$ and 
$
\frac{\sh_\delta\tilde\calG(\mbB_\calR)}{n}\le c\sh_\delta^2. 
$

The rest of the proof will occur on the event $\calE$. Let 
$\bfV\in S(\delta)$. There is $\ell\in\mathbb{N}^*$ such that $\bfV\in S_\ell(\delta)$. Since $\mbS_\infty^p\subset\mbB_\Pi$, we have two cases:
\begin{description}
\item[Case 1:] $\alpha^\ell\sh_\delta<1$. From \eqref{proof:prop:TP:MC:eq1}, 
\begin{align}
\Vert\bfV\Vert_\Pi^2-\Vert\frX^{(n)}(\bfV)\Vert_2^2&\le2\sqrt{2\pi}\cdot c (\alpha^{\ell}\sh_\delta)^2 
+ 2\sqrt{c}(\alpha^{\ell}\sh_\delta)^2
+\frac{c}{3}(\alpha^{\ell}\sh_\delta)^2\\
&\le\left(2\sqrt{2\pi}\cdot c + 2\sqrt{c}+\frac{c}{3}\right)\alpha^2\Vert\bfV\Vert_\Pi^2. 
\end{align} 

\item[Case 2:] $\alpha^{\ell-1}\sh_\delta<1\le\alpha^\ell\sh_\delta$. From \eqref{proof:prop:TP:MC:eq1}, 
\begin{align}
\Vert\bfV\Vert_\Pi^2-\Vert\frX^{(n)}(\bfV)\Vert_2^2&\le2\sqrt{2\pi}\cdot c (\alpha^{\ell}\sh_\delta)^2 
+ 2\sqrt{c}(\alpha^{\ell}\sh_\delta)
+\frac{c}{3}(\alpha^{\ell}\sh_\delta)^2\\
&\le\left(2\sqrt{2\pi}\cdot c + 2\sqrt{c}+\frac{c}{3}\right)(\alpha^{\ell}\sh_\delta)^2\\
&\le\left(2\sqrt{2\pi}\cdot c + 2\sqrt{c}+\frac{c}{3}\right)\alpha^2\Vert\bfV\Vert_\Pi^2. 
\end{align}
\end{description}

Choose $c=c(\alpha)\in(0,1)$ small enough so that 
$
a:=((2\sqrt{2\pi} + 3^{-1})c + 2\sqrt{c})\alpha^2 < 1
$
and set $\sc_0^2:=1-a$. We conclude that 
$
\inf_{\bfV\in \mbC(\delta)\cap\mbS_\infty^p}\frac{\Vert\frX^{(n)}(\bfV)\Vert_2}{\Vert\bfV\Vert_\Pi}\ge \sc_0.
$
The fact that 
$
\inf_{\bfV\in \mbC(\delta)}\frac{\Vert\frX^{(n)}(\bfV)\Vert_2}{\Vert\bfV\Vert_\Pi}\ge \sc_0
$
follows from homogeneity of norms and the fact that $\mbC(\delta)$ is a cone. This concludes the proof. 
\end{proof}

\begin{proposition}[$\ATP$]\label{prop:ARE:MC}
Let $\sc_0,c_0\in(0,1)$ be the universal constants of Proposition \ref{prop:TP:MC}. Given $\delta\in(0,1)$, let $\sh_\delta$ be the number in \eqref{prop:TP:MC:h} and define the cone
$$
\bar\mbC(\delta):=\left(
\frC\left(\sh_\delta\right)\times\re^n
\right)
\bigcap 
\left(
\mcC_{\calR^{(p)}}\left(\sh_{\delta}\right)\times\re^n
\right)
\bigcap \mdC_{\Vert\cdot\Vert_\sharp}\left(\frac{4}{\sc_0^2\sqrt{n}}\right). 
$$
Then, with probability at least $1-\delta$ the design operator $\frM$ satisfies 
\begin{align}
\inf_{[\bfV,\bu]\in \bar \mbC(\delta)}\frac{\Vert\frM^{(n)}(\bfV,\bu)\Vert_2}{\Vert[\bfV,\bu]\Vert_\Pi}\ge \sc_0/\sqrt{2}.
\end{align}
\end{proposition}
\begin{proof}
The proof will happen on the event for which the statement of Proposition \ref{prop:TP:MC} is satisfied. First, we claim that, for all $[\bfV,\bu]\in\mdR^p\times\re^n$,
\begin{align}
\langle\frX^{(n)}(\bfV),\bu\rangle &\le\frac{\Vert\bfV\Vert_\infty}{\sqrt{n}}\Vert\bu\Vert_\sharp.
\end{align}
Indeed, 
$
\sum_{i\in[n]}\bu_i\llangle\bfX_i^{(n)},\bfV\rrangle\le\frac{1}{\sqrt{n}}\sum_{i\in[n]}\bu_i^\sharp\llangle\bfX_i,\bfV\rrangle^\sharp
\le\frac{\Vert\bfV\Vert_\infty}{\sqrt{n}}\Vert\bu\Vert_\sharp,
$
since $\min_{i\in[n]}\omega_i=\omega_n\ge1$ as the sequence $\{\omega_i\}$ is non-increasing.

From the above fact and Proposition \ref{prop:TP:MC}, we have that, for all 
$[\bfV,\bu]\in\bar\mbC(\delta)$, 
\begin{align}
\Vert\frX^{(n)}(\bfV)+\bu\Vert_2^2&=\Vert\frX^{(n)}(\bfV)\Vert_2^2+\Vert\bu\Vert_2^2+2\langle\frX^{(n)}(\bfV),\bu\rangle\\
&\ge \sc_0^2\Vert\bfV\Vert_\Pi^2 + \Vert\bu\Vert_2^2 
- 2\frac{\Vert\bfV\Vert_\infty}{\sqrt{n}}\Vert\bu\Vert_\sharp\\
&\ge (\sc_0^2/2)\Vert\bfV,\bu\Vert_\Pi^2. 
\end{align}
This finishes the proof. 
\end{proof}

We now aim in proving $\MP$ with $\calR=\Vert\cdot\Vert_N$, the nuclear norm. We will use the following well known result which makes use of Bernstein-type inequalities for random matrices \cite{2002:ahlswede:winter},\cite{2012tropp}.\footnote{Note that 
$\max_{k,\ell}\{R_k,C_\ell\}\ge\frac{1}{d_1\wedge d_2}$.}  

\begin{lemma}[Lemma 2 in \cite{2011koltchinskii:lounici:tsybakov}, Lemma 5 in \cite{2014klopp}]\label{lemma:MP:tropp}
Suppose  $\{(\eta,\bfX)\}\cup\{(\eta_i,\bfX_i)\}_{i\in[n]}$ is an iid sequence taking values on $\re\times\calX$ such that $\eta$ has zero mean, variance $\sigma_\eta^2$ and $0<|\eta|_{\psi_2}<\infty$. For $t>0$, define
\begin{align}
\triangle_\eta(t):=\max\left\{\sigma_\eta\sqrt{\max_{k,\ell}\{R_k,C_\ell\}\frac{(\log t+\log d)}{n}},
|\eta|_{\psi_2}\log^{1/2}\left(\frac{|\eta|_{\psi_2} m}{\sigma_\eta}\right)\frac{(\log t +\log d)}{n}\right\}.
\end{align}

Then for some absolute constant $\sfC>0$, for all $\delta\in(0,1)$, with probability at least $1-\delta$, 
\begin{align}
\left\Vert\frac{1}{n}\sum_{i\in[n]}\eta_i\bfX_i\right\Vert_{\op}\le \sfC\triangle_\eta(2/\delta).
\end{align}
\end{lemma}

\begin{proposition}[$\MP$]\label{prop:MP:MC}
For all $\delta\in(0,1]$ and $n\in\mathbb{N}$, with probability of at least $1-\delta$, for all $[\bfV,\bu]\in\mdR^p\times\re^n$,
\begin{align}
\langle\bxi^{(n)},\frM^{(n)}(\bfV,\bu)\rangle&\le 
C\sqrt{p}\triangle_\xi(4/\delta)\cdot\calR^{(p)}(\bfV)
+ C\sigma\frac{\sqrt{1+\log(2/\delta)}}{\sqrt{n}}\Vert\bu\Vert_2
+ C\sigma\frac{\mathscr G(\mbB^n_\sharp)}{\sqrt{n}}\Vert\bu\Vert_\sharp.
\end{align}
\end{proposition}
\begin{proof}
From Lemma \ref{lemma:MP:tropp}, on an event $\calE_1$ of probability $\ge1-\delta/2$, 
\begin{align}
\left\Vert\frac{1}{n}\sum_{i\in[n]}\xi_i\bfX_i\right\Vert_{\op}\le \sfC\triangle_\xi(4/\delta).
\end{align}

We shall also use that, on an event $\calE_2$ of probability $\ge1-\delta/2$, for all $\bu\in\re^n$, 
\begin{align}
\frac{1}{\sqrt{n}}\sum_{i\in[n]}\xi_i\bu_i\le C\sigma\sqrt{\frac{1+\log(2/\delta)}{n}} + C\sigma\frac{\mathscr G(\mbB^n_\sharp)}{\sqrt{n}}\Vert\bu\Vert_\sharp. \label{proof:prop:MP:MC:eq1}
\end{align}
Indeed, from Lemma \ref{lemma:noise:concentration} in Section \ref{ss:subgaussian:designs} with the set $U:=\{\bu\in\mbB^n_2:\Vert\bu\Vert_\sharp\le r\}$ we obtain that, with probability $\ge1-\delta$,  
\begin{align}
\sup_{\bu\in U}\frac{\langle\bxi,\bu\rangle}{\sqrt{n}}\le C\sigma\frac{\mathscr G(\mbB^n_\sharp)}{\sqrt{n}}r + C\sigma\sqrt{\frac{2\log(1/\delta)}{n}}.
\end{align}
This fact and Lemma \ref{lemma:peeling:1dim} in the Appendix with set $V:=\mbB_2^n$ and functions $M(\bu):=-\langle\bxi,\bu\rangle/\sqrt{n}$, $h(\bu):=\Vert\bu\Vert_\sharp$ and $g(r):=r\mathscr G(\mbB^n_\sharp)$ entails that with probability $\ge1-\delta$, for all $\bu\in\mbB_2^n$, 
\begin{align}
\frac{1}{\sqrt{n}}\sum_{i\in[n]}\xi_i\bu_i\le C\sigma\sqrt{\frac{1+\log(1/\delta)}{n}} + C\sigma\frac{\mathscr G(\mbB^n_\sharp)}{\sqrt{n}}\Vert\bu\Vert_\sharp. 
\end{align}
This fact and homogeneity of norms imply that, on the same event, the property displayed above holds for all $\bu\in\re^n$. This proves \eqref{proof:prop:MP:MC:eq1}. 

On the event $\calE_1\cap\calE_2$ of probability $\ge1-\delta$, the inequality claimed in the lemma holds for all $[\bfV,\bu]\in\mdR^p\times\re^n$ since, by the dual norm inequality,  
\begin{align}
\langle\bxi^{(n)},\frM^{(n)}(\bfV,\bu)\rangle &= 
\frac{1}{n}\sum_{i\in[n]}\xi_i\llangle\bfX_i,\bfV\rrangle
+ \frac{1}{\sqrt{n}}\sum_{i\in[n]}\xi_i\bu_i\\
&\le \sqrt{p}\left\Vert\frac{1}{n}\sum_{i\in[n]}\xi_i\bfX_i\right\Vert_{\op}\calR^{(p)}(\bfV) 
+ \frac{1}{\sqrt{n}}\sum_{i\in[n]}\xi_i\bu_i. 
\end{align} 
\end{proof}

\subsection{Deterministic bounds}\label{ss:deterministic:bounds:MC}
Throughout this section, $\calR$ is a any decomposable norm. For convenience, we will work with the normalized regularization $\calR^{(p)}=\calR/\sqrt{p}$. Recall Definition \ref{def:dim:red:cones}, for given $c_0,\gamma,\eta>0$, of the cone  $\calC_{\bfB^*,\calR^{(p)}}(c_0,\gamma,\eta)$.
\begin{lemma}\label{lemma:dim:reduction:MC}
Suppose that $\Vert\bfB^*\Vert_\infty\le\sa$ and
\begin{itemize}
\item[\rm (i)] $(\frX,\bxi)$ satisfies 
$\MP_{\calR^{(p)},\Vert\cdot\Vert_\sharp}(\sf_1,\sf_2,\sf_4)$ in Definition \ref{def:MP:MC}. 
\item[\rm (ii)] 
$
\lambda = \gamma\tau
\ge2\sf_2,
$
$
\tau \ge 2\sf_4.
$
\end{itemize}
Then
\begin{align}
\Vert\frM^{(n)}(\bfDelta_{\bfB^*},\bfDelta^{\hat\btheta})\Vert_2^2
\le\sf_1\Vert\bfDelta^{\hat\btheta}\Vert_2+\triangle,\label{lemma:dim:reduction:MC:contraction}
\end{align}
where 
\begin{align}
\triangle:=(\nicefrac{3\lambda}{2})(\calR^{(p)}\circ\calP_{\bfB^*})(\bfDelta_{\bfB^*}) 
-(\nicefrac{\lambda}{2})(\calR^{(p)}\circ\calP_{\bfB^*}^\perp)(\bfDelta_{\bfB^*})
+(\nicefrac{3\tau\Omega}{2})\Vert\bfDelta^{\hat\btheta}\Vert_2 -(\nicefrac{\tau}{2})\sum_{i=o+1}^n\omega_i(\bfDelta^{\hat\btheta})_i^\sharp.
\end{align}
In particular, $[\bfDelta_{\bfB^*},\bfDelta^{\hat\btheta}]\in\calC_{\bfB^*,\calR^{(p)}}\left(3,\gamma,2\frac{\sf_1}{3\tau}+\Omega\right)$.
\end{lemma}
\begin{proof}[Proof sketch]
A similar argument used in Lemma \ref{lemma:dim:reduction} entails
\begin{align}
\Vert\frM^{(n)}(\bfDelta_{\bfB^*},\bfDelta^{\hat\btheta})\Vert_2^2
\le \langle\bxi^{(n)},\frM^{(n)}(\bfDelta_{\bfB^*},\bfDelta^{\hat\btheta})\rangle
+ \lambda(\calR^{(p)}(\bfB^*)-\calR^{(p)}(\hat\bfB))
+ \tau(\Vert\btheta^*\Vert_\sharp-\Vert\hat\btheta\Vert_\sharp).
\end{align} 
By condition (i)-(ii) and Lemmas \ref{lemma:A1:B} and \ref{lemma:A1} (with 
$\nu=1/2$), we obtain \eqref{lemma:dim:reduction:MC:contraction}.
\end{proof}

\begin{theorem}[Robust matrix completion]\label{thm:rate:MC:deterministic}
Define the cone
\begin{align}
\mbC:=\left(\frC\left(\sh_1\right)\times\re^n\right)
\bigcap 
\left(\mcC_{\calR^{(p)}}\left(\sh_2\right)\times\re^n\right)
\bigcap 
\mdC_{\Vert\cdot\Vert_\sharp}\left(\sh_3\right),
\end{align}
for positive constants $(\sh_1,\sh_2,\sh_3)$. Suppose that
\begin{itemize}
\item[\rm (i)] $(\frX,\bxi)$ satisfies the 
$\MP_{\calR^{(p)},\Vert\cdot\Vert_\sharp}(\sf_1,\sf_2,\sf_4)$. 
\item[\rm{(ii)}] For some $\sc_1>0$, the design satisfies 
$
\inf_{[\bfV,\bu]\in \mbC}\frac{\Vert\frM^{(n)}(\bfV,\bu)\Vert_2}{\Vert[\bfV,\bu]\Vert_\Pi}\ge \sc_1.
$
\item[\rm (iii)] 
$
\lambda = \gamma\tau \ge2\sf_2
$
and $\tau \ge 2\sf_3$.
\end{itemize}
Define the quantities 
$
\mathcal{L}:=\frac{\sa\sh_2}{\lambda}\vee\frac{\sa\sh_3}{\tau}
$
and
$$
R:=\Psi_{\calR^{(p)}}\left(\calP_{\bfB^*}(\bfDelta_{\bfB^*})\right)\mu\left(\calC_{\bfB^*,\calR^{(p)}}(4)\right).
$$

Then either $\Vert\bfDelta_{\bfB^*}\Vert_\Pi\le2\sa\sh_1$ or  
\begin{align}
\Vert[\bfDelta_{\bfB^*},\bfDelta^{\hat\btheta}]\Vert_\Pi
&\le\left[\mathcal{L}\sqrt{64\lambda^2R^2+(32\tau\Omega+20\sf_1)^2}\right]\bigvee
\left[\frac{1}{\sc_1^2}\sqrt{2.25\lambda^2R^2+(6\tau\Omega+4\sf_1)^2}\right],\\
\lambda\calR^{(p)}(\bfDelta_{\bfB^*}) + \tau \|\bfDelta^{\hat\btheta}\|_\sharp&\le\left[8\mathcal{L}(4\lambda^2R^2+(8\tau\Omega+5\sf_1)^2 )\right]\bigvee
\left[\frac{1}{\sc_1^2}(9\lambda^2R^2+(16\tau\Omega+10\sf_1)^2)
\right].
\end{align}
\end{theorem}
\begin{proof}
We will divide the proof in different cases. We recall that by Lemma \ref{lemma:dim:reduction:MC}, $[\bfDelta_{\bfB^*},\bfDelta^{\hat\btheta}]\in\mcC$ with $\mcC:=\calC_{\bfB^*,\calR^{(p)}}(3,\gamma,\Omega+2\frac{\sf_1}{3\tau})$.
\begin{description}
\item[Case 1:] $\bfDelta_{\bfB^*}\notin\frC\left(\sh_1\right)$.

By definition, 
$
\Vert\bfDelta_{\bfB^*}\Vert_\Pi\le\sh_1\Vert\bfDelta_{\bfB^*}\Vert_\infty
\le2\sa\sh_1
$
and we are done. 

\item[Case 2:] $[\bfDelta_{\bfB^*},\bfDelta^{\hat\btheta}]\notin\mcC_{\calR^{(p)}}\left(\sh_2\right)$. 

In that case 
\begin{align}
\Vert[\bfDelta_{\bfB^*},\bfDelta^{\hat\btheta}]\Vert_\Pi^2\le\sh_2\Vert\bfDelta_{\bfB^*}\Vert_\infty\calR^{(p)}(\bfDelta_{\bfB^*})
\le2\sa(\nicefrac{\sh_2}{\lambda})\lambda\calR^{(p)}(\bfDelta_{\bfB^*}).\label{thm:rate:MC:deterministic:eq1}
\end{align}
We now consider two cases. 
\begin{description}
\item[Case 2.1:] $4\calR^{(p)}(\calP_{\bfB^*}(\bfDelta_{\bfB^*})) \ge \calR^{(p)}(\calP_{\bfB^*}^\perp(\bfDelta_{\bfB^*}))$. 

We have $\bfDelta_{\bfB^*}\in\calC_{\bfB^*,\calR^{(p)}}(4)$. Decomposability of $\calR$, 
$[\bfDelta_{\bfB^*},\bfDelta^{\hat\btheta}]\in\mcC$ 
and Cauchy-Schwarz further imply
\begin{align}
	\lambda\calR^{(p)}(\bfDelta_{\bfB^*}) + \tau \|\bfDelta^{\hat\btheta}\|_\sharp
	&\le 4\lambda\calR^{(p)}(\calP_{\bfB^*}(\bfDelta_{\bfB^*})) + (4\tau\Omega+2\sf_1)\|\bfDelta^{\hat\btheta}\|_2\\
	&\le \sqrt{16\lambda^2R^2+ (4\tau\Omega+2\sf_1)^2}
	\big\|[\bfDelta_{\bfB^*},\bfDelta^{\hat\btheta}]\big\|_\Pi.
	\label{thm:rate:MC:deterministic:eq2}
\end{align}
The above display and \eqref{thm:rate:MC:deterministic:eq1} imply
\begin{align}
\big\|[\bfDelta_{\bfB^*},\bfDelta^{\hat\btheta}]\big\|_\Pi
\le4\sa(\nicefrac{\sh_2}{\lambda})\sqrt{4\lambda^2R^2+ (2\tau\Omega+\sf_1)^2},
\end{align}
and 
\begin{align}
\lambda\calR^{(p)}(\bfDelta^{\hat\bfB}) + \tau \|\bfDelta^{\hat\btheta}\|_\sharp\le8\sa(\nicefrac{\sh_2}{\lambda})[4\lambda^2R^2+ (2\tau\Omega+\sf_1)^2].
\end{align}

\item[Case 2.2:] $4\calR^{(p)}(\calP_{\bfB^*}(\bfDelta_{\bfB^*})) < \calR^{(p)}(\calP_{\bfB^*}^\perp(\bfDelta_{\bfB^*}))$. 

As 
$[\bfDelta_{\bfB^*},\bfDelta^{\hat\btheta}]\in\mcC$,
\begin{align} 
\lambda\calR^{(p)}(\calP_{\bfB^*}(\bfDelta_{\bfB^*}))+
\tau\sum_{i=o+1}^n\omega_i(\bfDelta^{\hat\btheta})^\sharp_i\le(3\tau\Omega+2\sf_1)\|\bfDelta^{\hat\btheta}\|_2.\label{thm:rate:MC:deterministic:eq2'}
\end{align}
This fact, decomposability of $\calR$ and Cauchy-Schwarz imply
\begin{align}
		\lambda\calR^{(p)}(\bfDelta_{\bfB^*}) 
		+\tau\big\|\bfDelta^{\hat\btheta}\big\|_\sharp
		&\le 4\lambda\calR^{(p)}(\calP_{\bfB^*}(\bfDelta_{\bfB^*})) + (4\tau\Omega+2\sf_1)\|\bfDelta^{\hat\btheta}\|_2\\
		&\le (16\tau\Omega+10\sf_1)\,\big\|\bfDelta^{\hat\btheta}\big\|_2.
		\label{thm:rate:MC:deterministic:eq3}
\end{align}
The above display and \eqref{thm:rate:MC:deterministic:eq1} imply
\begin{align}
\big\|[\bfDelta_{\bfB^*},\bfDelta^{\hat\btheta}]\big\|_\Pi
\le4\sa(\nicefrac{\sh_2}{\lambda})(8\tau\Omega+5\sf_1),
\end{align}
and 
\begin{align}
\lambda\calR^{(p)}(\bfDelta_{\bfB^*}) 
		+\tau\big\|\bfDelta^{\hat\btheta}\big\|_\sharp
\le8\sa(\nicefrac{\sh_2}{\lambda})(8\tau\Omega+5\sf_1)^2.
\end{align}
\end{description}

\item[Case 3:] $[\bfDelta_{\bfB^*},\bfDelta^{\hat\btheta}]\notin\mdC_{\Vert\cdot\Vert_\sharp}\left(\sh_3\right)$. 

In that case 
\begin{align}
\Vert[\bfDelta_{\bfB^*},\bfDelta^{\hat\btheta}]\Vert_\Pi^2\le\sh_3\Vert\bfDelta_{\bfB^*}\Vert_\infty
\Vert\bfDelta^{\hat\btheta}\Vert_\sharp
\le2\sa(\nicefrac{\sh_3}{\tau})\tau\Vert\bfDelta^{\hat\btheta}\Vert_\sharp.
\end{align}

As in Case 2, by dividing in the same two subcases and using the bounds \eqref{thm:rate:MC:deterministic:eq2} and \eqref{thm:rate:MC:deterministic:eq3} we get 
\begin{align}
\big\|[\bfDelta_{\bfB^*},\bfDelta^{\hat\btheta}]\big\|_\Pi
&\le4\sa(\nicefrac{\sh_3}{\tau})\sqrt{4\lambda^2R^2+ (2\tau\Omega+\sf_1)^2},\\
\lambda\calR^{(p)}(\bfDelta_{\bfB^*}) + \tau \|\bfDelta^{\hat\btheta}\|_\sharp&\le8\sa(\nicefrac{\sh_3}{\tau})[4\lambda^2R^2+ (2\tau\Omega+\sf_1)^2].
\end{align}
or
\begin{align}
\big\|[\bfDelta_{\bfB^*},\bfDelta^{\hat\btheta}]\big\|_\Pi
&\le4\sa(\nicefrac{\sh_3}{\tau})(8\tau\Omega+5\sf_1),\\
\lambda\calR^{(p)}(\bfDelta_{\bfB^*}) 
		+\tau\big\|\bfDelta^{\hat\btheta}\big\|_\sharp
&\le8\sa(\nicefrac{\sh_3}{\tau})(8\tau\Omega+5\sf_1)^2.
\end{align}

\item[Case 4:] $[\bfDelta_{\bfB^*},\bfDelta^{\hat\btheta}]\in\mbC$. 

As above, we further split in two cases.
\begin{description}
\item[Case 4.1:] $4\calR^{(p)}(\calP_{\bfB^*}(\bfDelta_{\bfB^*})) \ge \calR^{(p)}(\calP_{\bfB^*}^\perp(\bfDelta_{\bfB^*}))$. 

Hence 
$\bfDelta_{\bfB^*}\in\calC_{\bfB^*,\calR^{(p)}}(4)$. Decomposability of $\calR$, 
$[\bfDelta_{\bfB^*},\bfDelta^{\hat\btheta}]\in\mcC$ and Cauchy-Schwarz give
\begin{align}
	\sf_1\|\bfDelta^{\hat\btheta}\|_2+\triangle&\le(\nicefrac{3\lambda}{2})\calR^{(p)}(\calP_{\bfB^*}(\bfDelta_{\bfB^*})) + (\nicefrac{3\tau\Omega}{2}+\sf_1)\|\bfDelta^{\hat\btheta}\|_2\\
	&\le (\nicefrac{3}{2})\left\{\lambda^2R^2+ \left(\tau\Omega+\frac{2\sf_1}{3}\right)^2\right\}^{1/2}\Vert[\bfDelta_{\bfB^*},\bfDelta^{\hat\btheta}]\Vert_\Pi.
\end{align}
From the above display, \eqref{lemma:dim:reduction:MC:contraction} and condition (ii), we obtain
\begin{align}
\Vert[\bfDelta_{\bfB^*},\bfDelta^{\hat\btheta}]\Vert_\Pi\le (\nicefrac{3}{2})\frac{\sqrt{\lambda^2R^2+(\tau\Omega+(\nicefrac{2\sf_1}{3}))^2}}{\sc_1^2}.
\end{align}
Finally, \eqref{thm:rate:MC:deterministic:eq2} and the previous display imply
\begin{align}
	\lambda\calR^{(p)}(\bfDelta_{\bfB^*}) + \tau \|\bfDelta^{\hat\btheta}\|_\sharp
	&\le \frac{3}{\sc_1^2}\sqrt{\lambda^2R^2+(\tau\Omega+(\nicefrac{2\sf_1}{3}))^2}\sqrt{4\lambda^2R^2+ (2\tau\Omega+\sf_1)^2}\\
	&\le \frac{1}{\sc_1^2}[9\lambda^2R^2+ (3\tau\Omega+2\sf_1)^2].
\end{align}

\item[Case 4.2:] $4\calR^{(p)}(\calP_{\bfB^*}(\bfDelta_{\bfB^*})) < \calR^{(p)}(\calP_{\bfB^*}^\perp(\bfDelta_{\bfB^*}))$. 

As 
$[\bfDelta_{\bfB^*},\bfDelta^{\hat\btheta}]\in\mcC$, relation \eqref{thm:rate:MC:deterministic:eq2'} holds. This fact implies that
\begin{align}
	\sf_1\|\bfDelta^{\hat\btheta}\|_2+\triangle&\le(\nicefrac{3\lambda}{2})\calR^{(p)}(\calP_{\bfB^*}(\bfDelta_{\bfB^*})) + (\nicefrac{3\tau\Omega}{2}+\sf_1)\|\bfDelta^{\hat\btheta}\|_2
	\le(6\tau\Omega+4\sf_1)\|\bfDelta^{\hat\btheta}\|_2.
\end{align}
The above display, \eqref{lemma:dim:reduction:MC:contraction} and 
condition (ii) yield
\begin{align}
\Vert[\bfDelta_{\bfB^*},\bfDelta^{\hat\btheta}]\Vert_\Pi\le \frac{6\tau\Omega+4\sf_1}{\sc_1^2}.
\end{align}
Finally, \eqref{thm:rate:MC:deterministic:eq3} and the previous display imply
\begin{align}
		\lambda\calR^{(p)}(\bfDelta_{\bfB^*}) 
		+\tau\big\|\bfDelta^{\hat\btheta}\big\|_\sharp
		\le \frac{(16\tau\Omega+10\sf_1)^2}{\sc_1^2}.
\end{align}
\end{description}
\end{description}
To finish, we note that the bounds in the statement of the theorem is the maximum of all the bounds established in the above cases.
\end{proof}

\subsection{Proof of Theorem \ref{thm:robust:MC}}\label{ss:proof:thm:robust:matrix:completion}

We now set $\calR:=\Vert\cdot\Vert_N$. We have
$$
\tilde\calG(\mbB_\calR)=\sqrt{p}\esp\left[\mathscr G\left(\frX\left(\mbB_{\Vert\cdot\Vert_N}\right)\right)\right]\le 
\sqrt{p}\esp\left[\left\Vert\sum_{i\in[n]}g_i
 \bfX_i\right\Vert_{\op}\right],
$$
with $\{g_i\}_{i\in[n]}$ iid $\calN(0,1)$ independent of 
$\{\bfX_i\}_{i\in[n]}$. The assumptions of Lemma \ref{lemma:MP:tropp} apply to 
$\eta=\calN(0,1)$ so integration yields
\begin{align}
\frac{\tilde\calG(\mbB_\calR)}{n}\lesssim\max\left\{\sqrt{\frac{p\max_{k,\ell}\{R_k,C_\ell\}}{n}(\log d)},\frac{\sqrt{p}}{n}(\log m)^{1/2}(\log d)\right\}.
\end{align}
As before, Proposition E.2 in  \cite{2018bellec:lecue:tsybakov} implies that 
$\mathscr{G}(\mbB_{\Vert\cdot\Vert_\sharp}^n)\lesssim1$. 

Fix $\delta\in(0,1)$.  Let the numbers 
$
\sh_3\asymp \frac{1}{\sqrt{n}},
$
and
\begin{align}
\sh_1=\sh_2&\asymp \sqrt{p\max_{k,\ell}\{R_k,C_\ell\}\frac{\log d}{n}}
\bigvee
\sqrt{p}\log^{1/2}(m)\frac{\log d}{n}
\bigvee
\sqrt{\frac{1+\log(1/\delta)}{n}},
\end{align}
and cone
\begin{align}
\bar\mbC(\delta):=\frC\left(\sh_1\right)\times\re^n
\bigcap \mcC_{\calR^{(p)}}\left(\sh_2\right)\times\re^n
\bigcap \mdC_{\Vert\cdot\Vert_\sharp}\left(\sh_3\right).
\end{align}
By Proposition \ref{prop:ARE:MC}, there is universal constant $\sc_1\in(0,1)$ such that, on an event $\calE_1$ of probability $\ge1-\delta/2$, 
\begin{align}
\inf_{[\bfV,\bu]\in \bar \mbC(\delta)}\frac{\Vert\frM^{(n)}(\bfV,\bu)\Vert_2}{\Vert[\bfV,\bu]\Vert_\Pi}\ge \sc_1.
\end{align}

Take
\begin{align}
\sf_1&\asymp\sigma\frac{1+\sqrt{\log(2/\delta)}}{\sqrt{n}},\\
\sf_2&\asymp\max\left\{\sigma_\xi\sqrt{p\max_{k,\ell}\{R_k,C_\ell\}\frac{\log(4/\delta)+\log d}{n}},
\sqrt{p}\sigma\log^{1/2}\left(\frac{\sigma m}{\sigma_\xi}\right)\frac{\log(4/\delta) +\log d}{n}\right\},\\
\sf_4&\asymp\frac{\sigma}{\sqrt{n}}.
\end{align}
By Proposition \ref{prop:MP:MC}, we have on an event $\calE_2$ of probability $\ge1-\delta/2$, ${\MP}_{\Vert\cdot\Vert_N,\Vert\cdot\Vert_\sharp}(\sf_1,\sf_2,\sf_4)$ is satisfied. 

The proof will now hold on the event $\calE_1\cap\calE_2$ of probability $\ge1-\delta$. With the tuning parameters $\lambda\asymp (\nicefrac{\sf_2}{\sigma})(\sa\vee\sigma)$ and $\tau\asymp (\nicefrac{\sf_4}{\sigma})(\sa\vee\sigma)$ stated in Theorem \ref{thm:robust:MC}, all conditions of Theorem \ref{thm:rate:MC:deterministic} hold on such event implying the claimed rates, noting that 
$
R\le\sqrt{r}\mu(\bfB^*)
$
with 
$$
\mu(\bfB^*):=\mu^{(p)}\left(\calC_{\bfB^*,\Vert\cdot\Vert_N^{(p)}}(4)\right)=\mu^{(p)}\left(\calC_{\bfB^*,\Vert\cdot\Vert_N}(4)\right), 
$$
and that $\calL\lesssim1$, since $\sh_2/\lambda\lesssim\frac{1}{\sa\vee\sigma}$ and $\sh_3/\tau\lesssim\frac{1}{\sa\vee\sigma}$. 

\section{Appendix}

\subsection{Peeling lemmas}

\begin{lemma}[\cite{2019dalalyan:thompson}]\label{lemma:peeling:1dim}
Let $g$  be  right-continuous, non-decreasing function
from $\mathbb{R}_+$ to $\mathbb{R}_+$ and  $h$ be function from
$V$ to $\mathbb{R}_+$.
Assume that for some  constants $b\in\mathbb{R}_+$ and $c\ge 1$, for every $r>0$ and for any $\delta\in(0,1/(c\vee 7))$, we have
\begin{align}
A(r,\delta) = \Big\{\inf_{\bv\in V: h(\bv)\le r,}
M(\bv) \ge - g(r)-b\sqrt{\log(1/\delta)}\Big\},
\end{align}
with probability at least $1-c\delta$. 

Then, for any $\delta\in(0,1/(c\vee 7))$, with probability at least $1-c\delta$, we have for all $\bv\in V$,
\begin{align}
    M(\bv)
    \ge -1.2(g\circ h)(\bv) -
		b\big(3+\sqrt{\log(9/\delta)}\big).
\end{align}
\end{lemma} 

\begin{lemma}[\cite{2019dalalyan:thompson}]\label{lemma:peeling:2dim}
Let $g,\bar g$  be  right-continuous, non-decreasing functions
from $\mathbb{R}_+$ to $\mathbb{R}_+$ and  $h,\bar h$ be functions from
$V$ to $\mathbb{R}_+$.
Assume that for some  constants $b\in\mathbb{R}_+$ and $c\ge 1$, for every $r,\bar r>0$ and for any $\delta\in(0,1/(c\vee 7))$, we have
\begin{align}
A(r,\bar r,\delta) = \Big\{\inf_{\bv\in V: (h,\bar h)(\bv)\le (r,\bar r)}
M(\bv) \ge - g(r) -  \bar g(\bar r)-b\sqrt{\log(1/\delta)}\Big\},
\end{align}
with probability at least $1-c\delta$. 

Then, for any $\delta\in(0,1/(c\vee 7))$, with probability at least $1-c\delta$, we have for all $\bv\in V$,
\begin{align}
    M(\bv)    \ge -1.2(g\circ h)(\bv) -1.2(\bar g\circ\bar h)(\bv)-
		b\big(4.8+\sqrt{\log(81/\delta)}\big).
\end{align}
\end{lemma} 

\begin{lemma}\label{lemma:peeling:multiplier:process}
Let $g,\bar g$  be  right-continuous, non-decreasing functions
from $\mathbb{R}_+$ to $\mathbb{R}_+$ and  $h,\bar h$ be functions from $V$ to $\mathbb{R}_+$. Let $b>0$ be a constant and $c,c_0\ge2$ be universal constants. Assume that for every $r,\bar r>0$ and every $\delta\in(0,1/c)$ and every $\rho\in(0,1/c_0)$, the event $A(r,\bar r,\delta,\rho)$ defined by the inequality
\begin{align}
\inf_{\bv\in V: (h,\bar h)(\bv)\le (r,\bar r)}
M(\bv) &\ge - \left[1+\frac{1}{\sqrt{n}}\sqrt{\log(1/\rho)}\right]b\frac{g(r)}{\sqrt{n}}-b\frac{\bar g(\bar r)}{\sqrt{n}}\\
&-\frac{b}{\sqrt{n}}\sqrt{\log(1/\delta)}-\frac{b}{n}\log(1/\delta)-\frac{b}{n}\sqrt{\log(1/\delta)\log(1/\rho)},
\end{align}
has probability at least $1-c\delta-c_0\rho$. 

Then, for universal constant $C>0$, with probability at least $1-\delta-\rho$, we have that for all $\bv\in V$, 
\begin{align}
M(\bv) &\ge - C\left[1+\frac{\sqrt{\log(1/\rho)}}{\sqrt{n}}\right]b\frac{(g\circ h)(\bv)}{\sqrt{n}}-Cb\frac{(\bar g\circ \bar h)(\bv)}{\sqrt{n}}\\
&-C\frac{b}{\sqrt{n}}\sqrt{\log(1/\delta)}-C\frac{b}{n}\log(1/\delta)-C\frac{b}{n}\sqrt{\log(1/\delta)\log(1/\rho)}\\
&-C\frac{b}{n}\sqrt{\log(1/\rho)}-C\frac{b}{\sqrt{n}}[1+\sqrt{\log(1/\rho)}].
\end{align}
\end{lemma}
\begin{proof}
The proof is an adaptation of the proof of Lemma \ref{lemma:peeling:2dim} valid for subgaussian tails. The computations are slightly more involved in our case as 
$\sqrt{\log(1/\rho)}$ multiplies the peeled function $g$. We only sketch the proof for the one-dimensional case ($\bar g\equiv0$, $\bar h\equiv0$, 
$\bar r\equiv0$). During the proof, $C$ is an universal constant that may change.

Let $\eta,\epsilon>1$ be two parameters to be chosen later on.  We set $\mu_0=0$ and, for $k\ge1$, $\mu_k := \mu \eta^{k-1}$. For $k\in\mathbb{N}^*$, we define\footnote{Here $g^{-1}$ is the generalized inverse
defined by $g^{-1}(x) = \inf\{a\in\mathbb{R}_+: g(a)\ge x\}$. }
$\nu_k  :=g^{-1}(\mu_k)$ and the set
$$
V_k :=\{\bv\in V : \mu_k\le (g\circ h)(\bv)< \mu_{k+1}\}.
$$ 
The union bound and the fact that $\sum_{k\ge 1} k^{-1-\epsilon}\le
1+\epsilon^{-1}$ imply that the event
$$
A := \bigcap_{k=1}^\infty A\left(\nu_k, 0,\frac{\epsilon\delta}{(1+\epsilon)k^{1+\epsilon}},\frac{\epsilon\rho}{(1+\epsilon)k^{1+\epsilon}}\right),	
$$
has a probability at least $1-c\delta-c_0\rho$. For convenience, we define $\triangle(t):=\log\{(1+\epsilon)/(\epsilon t)\}$ and $\triangle_k:=({1+\epsilon})\log k$. Throughout the proof, assume that this event is realized:
\begin{align}
		\forall k\in\mathbb{N}^*
		\quad
		\begin{cases}
    \forall\bv\in V \text{ such that } h(\bv)\le \nu_k\text{ we have } \\
		M(\bv) \ge - [1+(\nicefrac{1}{\sqrt{n}})\sqrt{\triangle(\rho)+\triangle_k}]b\frac{g(\nu_k)}{\sqrt{n}}
		-(\nicefrac{b}{\sqrt{n}})\sqrt{\triangle(\delta)+\triangle_k}	
		-(\nicefrac{b}{n})[\triangle(\delta)+\triangle_k]\\
		\quad\quad\quad\quad-(\nicefrac{b}{n})\sqrt{[\triangle(\delta)+\triangle_k][\triangle(\rho)+\triangle_k]}.
		\end{cases}\label{lemma:peeling:MP:eq1}
\end{align}

For every $\bv\in V$, there is $\ell\in\mathbb{N}^*$ such that
$\bv\in V_\ell$. If $\ell\ge 1$, then $h(\bv)\le \nu_{\ell+1}$. This fact and \eqref{lemma:peeling:MP:eq1} implies
\begin{align}
    M(\bv) &\ge - [1+(\nicefrac{1}{\sqrt{n}})\sqrt{\triangle(\rho)+\triangle_\ell}]b(\nicefrac{\mu}{\sqrt{n}})\eta^\ell
    -(\nicefrac{b}{\sqrt{n}})\sqrt{\triangle(\delta)+\triangle_\ell}
		-(\nicefrac{b}{n})[\triangle(\rho)+\triangle_\ell]\\
		&-(\nicefrac{b}{n})\sqrt{[\triangle(\delta)+\triangle_\ell][\triangle(\rho)+\triangle_\ell]}\\
    &\ge-[1+(\nicefrac{1}{\sqrt{n}})\sqrt{\triangle(\rho)}]b(\nicefrac{\eta^2}{\sqrt{n}})(g\circ h)(\bv)
    -(\nicefrac{b}{\sqrt{n}})\sqrt{\triangle(\delta)}
		-(\nicefrac{b}{n})\triangle(\rho)
		-(\nicefrac{b}{n})\sqrt{\triangle(\delta)\triangle(\rho)}
		-\lozenge_\ell,
    \end{align}
where 
\begin{align}
\lozenge_\ell&:=[1+(\nicefrac{1}{\sqrt{n}})\sqrt{\triangle(\rho)+\triangle_\ell}](b\nicefrac{\mu}{\sqrt{n}})\eta^\ell
    +(\nicefrac{b}{\sqrt{n}})\sqrt{\triangle(\delta)+\triangle_\ell}
		+(\nicefrac{b}{n})[\triangle(\rho)+\triangle_\ell]\\
		&+(\nicefrac{b}{n})\sqrt{[\triangle(\delta)+\triangle_\ell][\triangle(\rho)+\triangle_\ell]}\\
		&-[1+(\nicefrac{1}{\sqrt{n}})\sqrt{\triangle(\rho)}](\nicefrac{b\mu}{\sqrt{n}})\eta^{\ell+1}
		-(\nicefrac{b}{\sqrt{n}})\sqrt{\triangle(\delta)}
		-(\nicefrac{b}{n})\triangle(\rho)
		-(\nicefrac{b}{n})\sqrt{\triangle(\delta)\triangle(\rho)}.
\end{align}
By appropriately choosing $\mu>0$ and $\epsilon,\eta>1$, a standard calculation shows that $\sup_{\ell\ge1}\lozenge_\ell\le C\frac{b}{\sqrt{n}}[1+\sqrt{\triangle(\rho)}]$.
   
If $\ell=0$, then \eqref{lemma:peeling:MP:eq1} with $k=1$ and using $g(\nu_1)=\mu$ lead to
\begin{align}
    M(\bv) &\ge - [1+(\nicefrac{1}{\sqrt{n}})\sqrt{\triangle(\rho)}](\nicefrac{b\mu}{\sqrt{n}})
    -(\nicefrac{b}{\sqrt{n}})\sqrt{\triangle(\delta)}
		-(\nicefrac{b}{n})\triangle(\rho)-(\nicefrac{b}{n})\sqrt{\triangle(\delta)\triangle(\rho)}.
    \end{align}
Joining the two lower bounds establish the claim.
\end{proof}

\begin{lemma}\label{lemma:peeling:TP:MC}
Let $h,\bar h$ be functions from $V$ to $\mathbb{R}_+$. Let $b>0$ and $\omega\ge1$ be constants and $c\ge1$ be a universal constant. Assume that for every $r,\bar r>0$ and every $\delta\in(0,1/c)$, the event $A(r,\bar r,\delta)$ defined by the inequality
\begin{align}
\inf_{\bv\in V: (h,\bar h)(\bv)\le (r,\bar r)}
M(\bv) &\ge - b\frac{\omega}{n}r\bar r
-b\sqrt{\frac{\log(1/\delta)}{n}}\cdot\bar r-\frac{b}{n}\log(1/\delta),
\end{align}
has probability at least $1-c\delta$. 

Then, for universal constant $C>0$, with probability at least $1-c\delta$, we have that for all $\bv\in V$, 
\begin{align}
M(\bv) &\ge-Cb\frac{\omega}{n}[1+h(\bv)\bar h(\bv)]
-C\frac{b}{\sqrt{n}}\sqrt{\log(1/\delta)}\cdot\bar h(\bv)
-C\frac{b}{\sqrt{n}}[1+\sqrt{\log(1/\delta)}]
-C\frac{b}{n}\log(1/\delta).
\end{align}
\end{lemma}
\begin{proof}
The proof is an adaptation of proof of Lemmas \ref{lemma:peeling:2dim} and \ref{lemma:peeling:multiplier:process}. One notable change is that the peeled functions $g(r)=r$ and $\bar g(\bar r)=\bar r$ are multiplied. We only give a sketch of the proof. In the following, $C$ is an universal constant that may change.

Let $\eta,\epsilon>1$ be two parameters to be chosen later on.  Let $\mu_0:=0$ and $\mu_k := \mu \eta^{k-1}$ for $k\ge1$. For $k,\bar k\in\mathbb{N}^*$, we define the set 
$$
V_{k,\bar k}:=\left\{\bv\in V : \mu_k\le h(\bv)< \mu_{k+1},\quad\mu_{\bar k}\le \bar h(\bv)< \mu_{\bar k+1}\right\}.
$$
The union bound and the fact that $\sum_{k,\bar k\ge 1} (k\bar k)^{-1-\epsilon}\le (1+\epsilon^{-1})^2$ imply that  the event
$$
A = \bigcap_{k,\bar k\ge1}^\infty A\Big(\mu_k,\bar\mu_{\bar k}, \frac{\epsilon^2
\delta}{(1+\epsilon)^2(k\bar k)^{1+\epsilon}}\Big)
$$
has probability at least $1-c\delta$. To ease notation, set
$\triangle(\delta):=\log\{(1+\epsilon)^2/(\epsilon^2 \delta)\}$ and $\triangle_{k,\bar k}:=({1+\epsilon})\log (k\bar k)$. We assume in the sequel that the event $A$ is realized, that is
\begin{align}
		\forall k,\bar k\in\mathbb{N}^*
		\quad
		\begin{cases}
    \forall\bv\in V \text{ such that }(h,\bar h)(\bv)\le (\mu_k,\mu_{\bar k})\text{ we have } \\
		M(\bv) \ge - b\frac{\omega}{n}\mu_k\mu_{\bar k}
		-(\nicefrac{b}{\sqrt{n}})\sqrt{\triangle(\delta)+\triangle_{k,\bar k}}\cdot\mu_{\bar k}	
		-(\nicefrac{b}{n})[\triangle(\delta)+\triangle_{k,\bar k}].
		\end{cases}\label{lemma:peeling:TP:MC:eq1}
\end{align}

For every $\bv\in V$, there are $\ell,\bar \ell\in\mathbb{N}^*$ such that
$\bv\in V_{\ell,\bar \ell}$. If $\ell\ge1$ or $\bar \ell\ge 1$, then $h(\bv)\le \mu_{\ell+1}$ and $\bar h(\bv)\le \mu_{\bar \ell+1}$. This fact and \eqref{lemma:peeling:TP:MC:eq1} imply
\begin{align}
    M(\bv) &\ge -b\frac{\omega}{n}\mu^2\eta^\ell\eta^{\bar \ell}-\frac{b}{\sqrt{n}}\sqrt{\triangle(\delta)+\triangle_{\ell+1,\bar \ell+1}}\cdot\mu\eta^{\bar \ell}
    -\frac{b}{n}[\triangle(\delta)+\triangle_{\ell+1,\bar \ell+1}]\\
    &\ge-b\frac{\omega}{n}\eta^4h(\bv)\bar h(\bv)-\frac{b}{\sqrt{n}}\sqrt{\triangle(\delta)}\cdot\eta^2\bar h(\bv)    
    -\frac{b}{n}\triangle(\delta)-\lozenge_{\ell,\bar \ell}
    \end{align}
where 
\begin{align}
\lozenge_{\ell,\bar \ell}&:=b\frac{\omega}{n}\mu^2\eta^\ell\eta^{\bar \ell}+\frac{b}{\sqrt{n}}\sqrt{\triangle(\delta)+\triangle_{\ell+1,\bar \ell+1}}\cdot\mu\eta^{\bar \ell}
    +\frac{b}{n}[\triangle(\delta)+\triangle_{\ell+1,\bar \ell+1}]\\
    &-b\frac{\omega}{n}\mu^2\eta^{k+1}\eta^{\bar \ell+1}
    -\frac{b}{\sqrt{n}}\sqrt{\triangle(\delta)}\mu\eta^{\bar \ell+1}-\frac{b}{n}\triangle(\delta).
\end{align}
By appropriately choosing $\mu,\eta>1$, one shows that 
$\sup_{\ell,\bar \ell\in\mathbb{N}^*}\lozenge_{\ell,\bar \ell}\le C\frac{b}{\sqrt{n}}[1+\sqrt{\triangle(\delta)}]$ for a universal constant $C>0$. 
   
If $\ell=\bar \ell=0$, then \eqref{lemma:peeling:TP:MC:eq1} with $k=\bar k=1$ leads to
\begin{align}
    M(\bv) \ge - b\frac{\omega}{n}\mu^2
		-(\nicefrac{b\mu}{\sqrt{n}})\sqrt{\triangle(\delta)}
		-(\nicefrac{b}{n})\triangle(\delta).
    \end{align}
Joining the two lower bounds establish the claim.
\end{proof}

\subsection{Multiplier process}\label{ss:multiplier:process}

\begin{tcolorbox}
Throughout this section, $(B,\calB,\probn)$ is a probability space, $(\xi,X)$ is a random (possibly not independent) pair taking values 
on $\re\times B$ and $X$ has marginal distribution $\probn$. $\{(\xi_i, X_i)\}_{i\in[n]}$ will denote an iid copy of $(\xi,X)$ and $\hat\probn$ be denotes the empirical measure associated to $\{X_i\}_{i\in[n]}$. 
\end{tcolorbox}

We are interested in concentration inequalities for the \emph{multiplier process} \cite{2016mendelson} 
$$
M(f):=\frac{1}{n}\sum_{i\in[n]}(\xi_if(X_i)-\esp[\xi f(X)]),
$$ 
defined over a subgaussian class of measurable functions $f:B\rightarrow\re$. 

Let $L_{\psi_2}=L_{\psi_2}(\probn)$ be the family of measurable functions $f:B\rightarrow\re$ having finite $\psi_2$-norm 
$$
\|f\|_{\psi_2}:=|f(X)|_{\psi_2}:=\inf\{c>0:\esp[\psi_2(\nicefrac{f(X)}{c})]\le1\}
$$
where $\psi_2(t):=e^{t^2}-1$. We will also assume that the $\psi_2$-norm of $\xi$, denoted also by 
$\Vert\xi\Vert_{\psi_2}$, is finite. Given $f,g\in L_{\psi_2}$, we set 
$
\dist_{\psi_2}(f,g):=\|f-g\|_{\psi_2},
$
$
\langle f,g\rangle_n:=\hat\probn fg
$
and
$
\Vert f\Vert_n:=\sqrt{\langle f,f\rangle_n}.
$
We recall the H\"older-type inequality $\Vert fg\Vert_{\psi_1}\le\Vert f\Vert_{\psi_2}\Vert g\Vert_{\psi_2}$. 

One pioneering idea is of ``generic chaining'' developed by Talagrand \cite{2014talagrand} and recently refined by Dirksen, Bednorz, Mendelson and collaborators \cite{2007mendelson:pajor:tomczak-jaegermann, 2015dirksen, 2014bednorz, 2016mendelson}. The following notion of complexity is used in generic chaining bounds.

\begin{tcolorbox}
\begin{definition}[$\gamma_{2,p}$-functional]
Let $(T,\dist)$ be a pseudo-metric space. We say a sequence $(T_k)$ of subsets of $T$ is \emph{admissible} if $|T_0|=1$ and $|T_k|\le 2^{2^{k}}$ for $k\in\mathbb{N}$ and $\cup_{k\ge0} T_k$ is dense in $T$. Let $\calA$ denote the class of all such admissible subset sequences. Given $p\ge1$, the 
$\gamma_{2,p}$-functional with respect to $(T,\dist)$ is the quantity
\begin{align}
\gamma_{2,p}(T,\dist):=\inf_{(T_k)\in\calA}\sup_{t\in T}\sum_{k\ge\lfloor\log_2 p\rfloor}2^{k/2}\dist(t,T_k).
\end{align}
We will say that $(T_k)\in\calA$ is optimal if it achieves the infimum above. Set $\gamma_{2}(T,\dist):=\gamma_{2,1}(T,\dist)$.
\end{definition}
\end{tcolorbox}

For the rest of this section, we set  $\dist:=\dist_{\psi_2}$ and omit the subscript $\dist$ for convenience. Given a subclass $F\subset L_{\psi_2}$, we let $\Delta(F):=\sup_{f,f'\in F}\dist(f,f')$ and $\bar\Delta(F):=\sup_{f\in F}\dist(f,0)$.

We will prove the following theorem.
\begin{theorem}[Multiplier process]\label{thm:mult:process}

There exist universal constants $c>0$, such that for all $f_0\in F$, $n\ge1$, $u\ge 1$ and $v\ge1$, with probability at least $1-ce^{-u/4}-ce^{-nv}$,
\begin{align}
\sup_{f\in F}|M(f)-M(f_0)|&\lesssim
\left(\sqrt{v}+1\right)\Vert\xi\Vert_{\psi_2}\frac{\gamma_2(F)}{\sqrt{n}}
+\left(\sqrt{\frac{2u}{n}}+\frac{u}{n}
+\sqrt{\frac{uv}{n}}\right)\Vert\xi\Vert_{\psi_2}\bar\Delta(F)
\end{align}
\end{theorem}
\begin{remark}\label{rem:multiplier:process:mendelson}
Mendelson \cite{2016mendelson} established impressive concentration inequalities for the multiplier process. In fact, they hold for much more general $(\xi,X)$  having heavier tails (see Theorems 1.9 and 4.4 in \cite{2016mendelson}). When specifying these bounds to  subgaussian classes, however, the confidence parameter $u>0$ multiplies the complexity functional $\gamma_2(F)$. In Theorem \ref{thm:mult:process}, $u>0$ does not multiply $\gamma_2(F)$, a fact that will be useful for our purposes mentioned in item (c) of Section \ref{s:intro}.
\end{remark}

Our proof is inspired by Dirksen's method \cite{2015dirksen} which obtained concentration inequalities for the \emph{quadratic process} (Theorem \ref{thm:bednorz} in Section \ref{ss:subgaussian:designs}). One key observation used by Dirksen \cite{2015dirksen} and Bednorz \cite{2014bednorz} is that one must bound the chain differently for $k\le \lfloor\log_2 n\rfloor$, the so called ``subgaussian path'' and $k\ge \lfloor\log_2 n\rfloor$, the ``subexponential path''. In bounding the multiplier process, we additionally introduce a ``lazy walked'' chain,  a technique already present in Talagrand's original bound for the \emph{empirical process} \cite{2014talagrand}. The following lemma is proved similarly to Lemma 5.4 in \cite{2015dirksen} so we omit the proof.
\begin{lemma}\label{lemma:increment:bounds:mult:process}
Let $f,f' \in L_{\psi_2}$. 

If for $k\in\mbN$, $2^{k/2}\le\sqrt{n}$, then for any $u>0$, with probability at least $1-2\exp(-(2^k+u))$,
\begin{align}
|M(f)-M(f')|\le \left[(1+\sqrt{2})\frac{2^{k/2}}{\sqrt{n}}+\sqrt{\frac{2u}{n}}+\frac{u}{n}\right]\Vert\xi\Vert_{\psi_2}\Vert f-f'\Vert_{\psi_2}.
\end{align}

If for $k\in\mbN$, $2^{k/2}\ge\sqrt{n}$, then for any $u\ge1$, with probability at least $1-2\exp(-(2^k+u))$,
\begin{align}
\Vert f-f'\Vert_{n}\le(\sqrt{u}+2^{k/2})\frac{[2(1+\sqrt{2})+1]^{1/2}}{\sqrt{n}}\dist(f,f').
\end{align}
\end{lemma}

\begin{proof}[Proof of Theorem \ref{thm:mult:process}]
Let $(F_k)$ be an optimal admissible sequences for 
$\gamma_{2}(F)$. Let $(\calF_k)$ be defined by $\calF_0:=F_0$ and $\calF_k:=\cup_{j\le k}F_j$ so that $|\calF_k|\le2|F_k|=2^{2^k+1}$. Set $k_0:=\min\{k\ge1:2^{k/2}>\sqrt{n}\}$ and let us define 
$
\calI:=\{k\in\mbN:\ell<k<k_0\}
$
and
$
\calJ:=\{k\in\mbN:k\ge k_0\}
$. 
Given $k\in\mbN$ and $f\in F$, let 
$
\Pi_k(f)\in\argmin_{f'\in\calF_k}\dist(f,f'),
$
Given $f\in F$, we take some $\Pi_0(f)\in F$ and for any $j\in\mathbb{N}$, we define the ``lazy walk'' chain selection by:  
\begin{align}
k_j(f):=\inf\left\{j\ge k_{j-1}(f):\dist(f,\Pi_{j}(f))\le\frac{1}{2}\dist(f,\Pi_{k_{j-1}(f)}(f))\right\}.
\end{align}
For simplicity of notation, we will set $\pi_j(f):=\Pi_{k_j(f)}(f)$. For $f\in F$, our proof will rely on the chain: 
\begin{align}
M(f)-M(\pi_0(f))
&=\sum_{j:k_j(f)\in\calJ}\left[M(\pi_{j+1}(f))-M(\pi_{j}(f))\right]
+\sum_{j:k_j(f)\in\calI}\left[M(\pi_{j}(f))-M(\pi_{j-1}(f))\right],\label{thm:mult:process:eq:chain}
\end{align}
where we have used that $\cup_{k\ge0}\calF_k$ is dense on $F$. 

Fix $u\ge1$. Given any $k\in\mathbb{N}$, define the event 
$\Omega_{k,\calI,u}$ for which, for all $f,f'\in \calF_k$, we have
\begin{align}
|M(f)-M(f')|\le \left[(1+\sqrt{2})\frac{2^{k/2}}{\sqrt{n}}+\sqrt{\frac{2u}{n}}+\frac{u}{n}\right]\Vert\xi\Vert_{\psi_2}\Vert f-f'\Vert_{\psi_2}.\label{thm:mult:process:eq1}
\end{align}
Define also the event $\Omega_{k,\calJ,u}$ for which, for all $f,f'\in \calF_k$, we have
\begin{align}
\Vert f-f'\Vert_n&\le(\sqrt{u}+2^{k/2})\frac{[2(1+\sqrt{2})+1]^{1/2}}{\sqrt{n}}\Vert f-f'\Vert_{\psi_2}.\label{thm:mult:process:eq2}
\end{align}
For simplicity, define the vector 
$\bxi:=(\xi_i)_{i\in[n]}$ and $\Vert\bxi\Vert_n:=\frac{1}{\sqrt{n}}\Vert\bxi\Vert_2$. Given $v\ge1$, we define the event 
$\Omega_{\xi,v}$, for which
\begin{align}
\Vert\bxi\Vert_n\le[2(1+\sqrt{2})+1]^{1/2}\Vert\xi\Vert_{\psi_2}\sqrt{v}.\label{thm:mult:process:eq2'} 
\end{align}

By an union bound over all possible pairs 
$(\pi_{k-1}(f),\pi_{k}(f))$ we have $|\Omega_{k,\calI,u}|\le|\calF_{k-1}||\calF_{k}|\le2^{2^{k+1}}$. If 
$\Omega_{\calI,u}:=\cap_{k\in\calI}\Omega_{k,\calI,u}$, the first bound on Lemma \ref{lemma:increment:bounds:mult:process} for $k\in\calI$ and a standard union bound using the geometric series\footnote{See e.g. a modification of Lemma A.4 in \cite{2015dirksen}.} imply that there is universal constant $c>0$
$$
\prob(\Omega_{\calI,u}^c)\le ce^{-u/4}.
$$
Similarly, the second bound in Lemma \ref{lemma:increment:bounds:mult:process} for $k\in\calJ$ imply that for the event  
$\Omega_{\calJ,u}:=\cap_{k\in\calJ}\Omega_{k,\calJ,u}$, we have 
$$
\prob(\Omega_{\calJ,u}^c)\le ce^{-u/4}.
$$
Using Bernstein's inequality for $\{\xi_i\}_{i\in[n]}$ we get $\prob(\Omega_{\xi,v}^c)\le ce^{-vn}$. Hence, the event 
$\Omega_{u,v}:=\Omega_{\calI,u}\cap\Omega_{\calJ,u}\cap\Omega_{\xi,v}$ has 
$\prob(\Omega_{u,v}^c)\le ce^{-u/4}+ce^{-vn}$. 

We next fix $u\ge2$ and $v\ge1$ and assume that $\Omega_{u,v}$ always holds. We now bound the chain over $\calI$ and $\calJ$ separately.

\emph{The subgaussian path $\calI$.} Given $j$ such that $k_j(f)\in\calI$, since $\pi_j(f),\pi_{j-1}(f)\in\calF_{k_j(f)}$, we may apply \eqref{thm:mult:process:eq1} to $k:=k_j(f)$ so that  
\begin{align}
|M(\pi_{j}(f))-M(\pi_{j-1}(f))|\le \left[(1+\sqrt{2})\frac{2^{k_j(f)/2}}{\sqrt{n}}+\sqrt{\frac{2u}{n}}+\frac{u}{n}\right]\Vert\xi\Vert_{\psi_2}\Vert \pi_{j}(f)-\pi_{j-1}(f)\Vert_{\psi_2}.
\end{align}
We note that, by triangle inequality and minimality of $k_{j-1}(f)$,  
\begin{align}
\Vert \pi_{j}(f)-\pi_{j-1}(f)\Vert_{\psi_2}
\le\dist(f,\calF_{k_j(f)})+\dist(f,\calF_{k_{j-1}(f)})\le\dist(f,\calF_{k_j(f)})+2\dist(f,\calF_{k_{j}(f)-1}),
\end{align}
so that 
\begin{align}
\sum_{j:k_j(f)\in\calI}2^{k_j(f)/2}\Vert \pi_{j}(f)-\pi_{j-1}(f)\Vert_{\psi_2}
\le(1+2\sqrt{2})\gamma_2(F). \label{thm:mult:process:sum:gamma1}
\end{align}
Moreover, by the definition of the lazy walked chain and a geometric series bound,
\begin{align}
\sum_{j:k_j(f)\in\calI}\Vert \pi_{j}(f)-\pi_{j-1}(f)\Vert_{\psi_2}\le4\dist(f,\calF_0)\le4\bar\Delta(F). \label{thm:mult:process:sum:delta1}
\end{align}
We thus conclude that 
\begin{align}
\left|\sum_{j:k_j(f)\in\calI}[M(\pi_{j}(f))-M(\pi_{j-1}(f))]\right|&\le 
4\left(\sqrt{\frac{2u}{n}}+\frac{u}{n}\right)\Vert\xi\Vert_{\psi_2}\bar\Delta(F)
+(1+\sqrt{2})(1+2\sqrt{2})\Vert\xi\Vert_{\psi_2}\frac{\gamma_{2}(F)}{\sqrt{n}}.	\label{thm:mult:process:subgaussian}
\end{align}

\emph{The subexponential path $\calJ$.} Let us denote by $\bfQ$ the joint distribution of $(\xi,X)$ and $\hat\bfQ$ the empirical distribution associated to $\{(\xi_i,X_i)\}_{i\in[n]}$. In particular, $M(f)=\hat\bfQ(\cdot)f-\bfQ\hat\bfQ(\cdot)f$. By Jensen's and triangle inequalities,
\begin{align}
\left|\sum_{j:k_j(f)\in\calJ}[M(\pi_{j+1}(f))-M(\pi_{j}(f))]\right|
&\le\sum_{j:k_j(f)\in\calJ}\hat\bfQ(\cdot)\left|\pi_{j+1}(f)-\pi_{j}(f)\right|\\
&+\sum_{j:k_j(f)\in\calJ}\bfQ\hat\bfQ(\cdot)\left|\pi_{j+1}(f)-\pi_{j}(f)\right|.\label{thm:mult:process:eq3}
\end{align}
For convenience, we set $\hat T_j:=\hat\bfQ(\cdot)\left|\pi_{j+1}(f)-\pi_{j}(f)\right|$.

Given $j$ such that $k_j(f)\in\calJ$, since 
$\pi_{j+1}(f),\pi_{j}(f)\in\calF_{k_{j+1}(f)}$, we may apply \eqref{thm:mult:process:eq2} to $k:=k_{j+1}(f)$. This fact, \eqref{thm:mult:process:eq2'} and Cauchy-Schwarz yield 
\begin{align}
\hat T_j&\le\Vert\bxi\Vert_n\Vert\pi_{j+1}(f)-\pi_{j}(f)\Vert_n
\le\sqrt{v}[2(1+\sqrt{2})+1]\Vert\xi\Vert_{\psi_2}(\sqrt{u}+2^{k_{j+1}(f)/2})\frac{1}{\sqrt{n}}\Vert \pi_{j+1}(f)-\pi_{j}(f)\Vert_{\psi_2}. 
\end{align}
In a similar fashion, we can also state that with probability at least $1-ce^{-u/4}$,  
\begin{align}
\frac{\hat T_j}{\Vert\bxi\Vert_n}\le(\sqrt{u}+2^{k_{j+1}(f)/2})\frac{[2(1+\sqrt{2})+1]^{1/2}}{\sqrt{n}}\Vert\pi_{j+1}(f)-\pi_{j}(f)\Vert_{\psi_2},
\end{align}
so integrating the tail leads to
\begin{align}
\left\{\esp\left(\frac{\hat T_j}{\Vert\bxi\Vert_n}\right)^2\right\}^{1/2}\le c2^{k_{j+1}(f)/2}\frac{[2(1+\sqrt{2})+1]^{1/2}}{\sqrt{n}}\Vert\pi_{j+1}(f)-\pi_{j}(f)\Vert_{\psi_2},
\end{align}
and by H\"older's inequality,
\begin{align}
\probn\hat T_j\le \left\{\esp\Vert\bxi\Vert_n^2\right\}^{1/2}\left\{\esp\left(\frac{\hat T_j}{\Vert\bxi\Vert_n}\right)^2\right\}^{1/2}
\le c[2(1+\sqrt{2})+1]^{1/2}\frac{\Vert\xi\Vert_{\psi_2}}{\sqrt{n}}\frac{2^{k_{j+1}(f)/2}\Vert\pi_{j+1}(f)-\pi_{j}(f)\Vert_{\psi_2}}{\sqrt{n}}.
\end{align}
We thus conclude from \eqref{thm:mult:process:eq3} and a analogous reasoning to \eqref{thm:mult:process:sum:gamma1} and \eqref{thm:mult:process:sum:delta1} that 
\begin{align}
\left|\sum_{j:k_j(f)\in\calJ}[M(\pi_{j+1}(f))-M(\pi_{j}(f))]\right|&\le 
c_1^2\sqrt{v}\Vert\xi\Vert_{\psi_2}\sqrt{u}\frac{4\bar\Delta(F)}{\sqrt{n}}\\
&+\left(c_1^2\sqrt{v}+c_1\frac{c}{\sqrt{n}}\right)\Vert\xi\Vert_{\psi_2}(1+2\sqrt{2})\frac{\gamma_2(F)}{\sqrt{n}},
\end{align}
with $c_1:=[2(1+\sqrt{2})+1]^{1/2}$.

From the above bound, \eqref{thm:mult:process:subgaussian} and \eqref{thm:mult:process:eq:chain} we conclude that, for any $u\ge2$ and $v\ge1$, on the event $\Omega_{u,v}$ of probability at least $1-ce^{-u/4}-ce^{-nv}$, we have the bound stated in the theorem.
\end{proof}

\subsection{Product process}

With the same setting and definitions of Section \ref{ss:multiplier:process}, we are now interested in concentration bounds for the \emph{product process} \cite{2016mendelson}
$$
A(f,g):=\hat\probn(fg-\probn fg)=
\frac{1}{n}\sum_{i\in[n]}
\bigg\{f(X_i)g(X_i)-\esp f(X_i)g(X_i) \bigg\},
$$
defined over two distinct subgaussian classes $F$ and $G$ of measurable functions. We will prove the following theorem. 
\begin{theorem}[Product process]\label{thm:product:process}
Let $F,G$ be subclasses of $L_{\psi_2}$. For any $1\le p<\infty$,
\begin{align}
\Lpnorm{\sup_{(f,g)\in F\times G}|A(f,g)|}\lesssim\frac{\gamma_{2,p}(F)\gamma_{2,p}(G)}{n}
+\bar\Delta(F)\frac{\gamma_{2,p}(G)}{\sqrt{n}}+\bar\Delta(G)\frac{\gamma_{2,p}(F)}{\sqrt{n}}+\bar\Delta(F)
\bar\Delta(G)\left(\sqrt{\frac{p}{n}}+\frac{p}{n}\right).
\end{align}
In particular, there exist universal constants $c,C>0$, such that for all $n\ge1$ and $u\ge 1$, with probability at least $1-e^{-u}$,
\begin{align}
\sup_{(f,g)\in F\times G}\left|A(f,g)\right|&\le C\left[\frac{\gamma_{2}(F)\gamma_{2}(G)}{n}
+\bar\Delta(F)\frac{\gamma_{2}(G)}{\sqrt{n}}+\bar\Delta(G)\frac{\gamma_{2}(F)}{\sqrt{n}}\right]\\
&+c\sup_{(f,g)\in F\times G}\Vert fg-\probn fg\Vert_{\psi_1}\left(\sqrt{\frac{u}{n}}+\frac{u}{n}\right).
\end{align}
\end{theorem}
\begin{remark}
Again, Mendelson \cite{2016mendelson} proved concentration inequalities for the product process for much more general classes having heavier tails (see Theorem 1.13 in \cite{2016mendelson}). When specifying these bounds for  subgaussian classes, the confidence parameter $u>0$ multiplies the complexity functionals. The fact that in Theorem \ref{thm:product:process} the confidence parameter $u>0$ does not multiply the complexity functionals 
$\gamma_2(F)$ and $\gamma_2(G)$ will be useful for the purpose specified in item (c) in Section \ref{s:intro}. 
\end{remark}

Before proving the theorem we will need some auxiliary results. The following lemma is proved similarly to Lemma 5.4 in \cite{2015dirksen} so we omit the proof. 
\begin{lemma}\label{lemma:increment:bounds}
Let $f,f'$ and $g,g'$ in $L_{\psi_2}$. 

If for $k\in\mbN$, $2^{k/2}\le\sqrt{n}$, then for any $u\ge1$, with probability at least $1-2\exp(-2^ku)$,
\begin{align}
|A(f,g)-A(f',g')|\le 2(1+\sqrt{2})\frac{u2^{k/2}}{\sqrt{n}}\Vert fg-f'g'\Vert_{\psi_1}.
\end{align}

If for $k\in\mbN$, $2^{k/2}\ge\sqrt{n}$, then for any $u\ge1$, with probability at least $1-2\exp(-2^ku)$,
\begin{align}
\Vert f-f'\Vert_{n}\le\sqrt{u}2^{k/2}\frac{[2(1+\sqrt{2})+1]^{1/2}}{\sqrt{n}}\dist(f,f').
\end{align}
\end{lemma}

As in the proof of Theorem \ref{thm:mult:process}, we combine Dirksen's method \cite{2015dirksen} with Talagrand's \cite{2014talagrand} ``lazy-walked'' chain. One difference now is that we will explicitly need Dirksen's bound for the \emph{quadratic process}
$$
A(f):=\hat\probn(f^2-\probn f^2).
$$
The following proposition is a corollary of the proof of Theorem 5.5 in \cite{2015dirksen}.
\begin{proposition}[Dirksen \cite{2015dirksen}, Theorem 5.5]\label{prop:quadratic:process}
Let $F\subset L_{\psi_2}$. Given $1\le p<\infty$, set $\ell:=\lfloor\log_2p\rfloor$ and $k_0:=\min\{k>\ell:2^{k/2}>\sqrt{n}\}$. Let $(\calF_k)$ be an optimal admissible sequence for $\gamma_{2,p}(F,\dist)$ and, for any $f\in F$ and $k\in\mbN$, let $\pi_k(f)\in\argmin_{f'\in\calF_\ell}\dist(f,f')$.  Then there exists universal constant $c>0$ such that for all $n\in\mathbb{N}$ and $u\ge2$, with probability at least $1-ce^{-pu/4}$,
\begin{align}
\sup_{f\in F}\sup_{k\ge k_0}\left|A(\pi_{k}(f))\right|^{1/2}-\sup_{f\in F}\left|A(\pi_{k_0}(f))\right|^{1/2}&\le  
\sqrt{u}\left[25\frac{\gamma_{2,p}(F,\dist)}{\sqrt{n}}+\left(85	\frac{\bar\Delta(F)\gamma_{2,p}(F,\dist)}{\sqrt{n}}\right)^{1/2}\right].
\end{align}

Moreover, for all $n\in\mathbb{N}$ and $u\ge1$, with probability at least $1-ce^{-pu/4}$,
\begin{align}
\sup_{f\in F}\left|A(\pi_{k_0}(f))\right|^{1/2}\le\sqrt{u}[4(1+\sqrt{2})+2]^{1/2}\bar\Delta(F).
\end{align}
\end{proposition}

Finally, we need two additional lemmas. 
\begin{lemma}[Lemma A.3 in Dirksen \cite{2015dirksen}]\label{lemma:dirksen:A3}
Fix $1\le p<\infty$, set $\ell:=\lfloor\log_2 p\rfloor$ and let $(X_t)_{t\in T}$ be a finite collection of real-valued random variables with $|T|\le2^{2^\ell}$. 

Then 
\begin{align}
\left(\esp\sup_{t\in T}|X_t|^p\right)^{1/p}\le
2\sup_{t\in T}(\esp|X_t|^p)^{1/p}.\label{lemma:dirksen:A3:moment}
\end{align}
\end{lemma}

\begin{lemma}[Lemma A.5 in Dirksen \cite{2015dirksen}]\label{lemma:dirksen:A5}
Fix $1\le p<\infty$ and $0<\alpha<\infty$. Let $\gamma\ge0$ and suppose that $\xi$ is a positive random variable such that for some $c\ge1$ and $u_*>0$, for all $u\ge u_*$, 
\begin{align}
\prob(\xi>\gamma u)\le c\exp(-pu^\alpha/4).\label{lemma:dirksen:A5:tail}
\end{align}
Then, for a constant $\tilde c_\alpha>0$, depending only on $\alpha$, 
\begin{align}
(\esp\xi^p)^{1/p}\le\gamma(\tilde c_\alpha c+u_*).	
\end{align}
\end{lemma}

\begin{proof}[Proof of Theorem \ref{thm:product:process}]
Let $(\calF_k)$ and $(\calG_k)$ be optimal admissible sequences for $\gamma_{2,p}(F)$  and 
$\gamma_{2,p}(G)$ respectively. Set $\ell:=\lfloor\log_2 p\rfloor$, $k_0:=\min\{k>\ell:2^{k/2}>\sqrt{n}\}$ and let us define 
$
\calI:=\{k\in\mbN:\ell<k<k_0\}
$
and
$
\calJ:=\{k\in\mbN:k\ge k_0\}
$. Given $(f,g)\in F\times G$, for any $k\in\mathbb{N}$, we take the usual selections
$
\pi_k(f)\in\argmin_{f'\in\calF_k}\dist(f,f'),
$
and 
$
\Pi_k(g)\in\argmin_{g'\in\calG_k}\dist(g,g').
$
For convenience, we also define 
$\calP_k(f,g):=A(\pi_{k}(f),\Pi_{k}(g))$ and 
$\mcP_k(f,g):=\pi_{k}(f)\Pi_{k}(g)$. Our proof will rely on the chain: 
\begin{align}
A(f,g)-A(\pi_\ell(f),\Pi_\ell(g))
&=\sum_{k\in\calJ}\left[\calP_{k+1}(f,g)-\calP_{k}(f,g)\right]
+\sum_{k\in\calI}\left[\calP_{k}(f,g)-\calP_{k-1}(f,g)\right],\label{thm:product:process:eq:chain}
\end{align}
where we have used that $\cup_{k\ge0}\calF_k\times\calG_k$ is dense on $F\times G$. 

Fix $u\ge2$. Given any $k\in\mathbb{N}$, define the event 
$\Omega_{k,\calI,u,p}$ for which, for all $f\in F$ and $g\in G$, we have
\begin{align}
|\calP_{k}(f,g)-\calP_{k-1}(f,g)|\le 2(1+\sqrt{2})u2^{k/2}\frac{\Vert\mcP_{k}(f,g)-\mcP_{k-1}(f,g)]\Vert_{\psi_1}}{\sqrt{n}}.\label{thm:product:process:eq1}
\end{align}
Define also the event $\Omega_{k,\calJ,u,p}$ for which, for all $f\in F$ and $g\in G$, we have both inequalities:
\begin{align}
\Vert \pi_{k+1}(f)-\pi_{k}(f)\Vert_{n}&\le\sqrt{u}2^{k/2}\frac{[2(1+\sqrt{2})+1]^{1/2}}{\sqrt{n}}\Vert\pi_{k+1}(f)-\pi_{k}(f)\Vert_{\psi_2},\\
\Vert \Pi_{k+1}(g)-\Pi_{k}(g)\Vert_{n}&\le\sqrt{u}2^{k/2}\frac{[2(1+\sqrt{2})+1]^{1/2}}{\sqrt{n}}\Vert\Pi_{k+1}(g)-\Pi_{k}(g)\Vert_{\psi_2}.\label{thm:product:process:eq2}
\end{align}
By an union bound over all possible 4-tuples 
$(\pi_{k-1}(f),\pi_{k}(f),\Pi_{k-1}(g),\Pi_{k}(g))$ we have $|\Omega_{k,\calI,u,p}|\le|\calF_{k-1}||\calF_{k}||\calG_{k}||\calG_{k-1}|\le2^{2^{k+2}}$. If $\Omega_{\calI,u,p}:=\cap_{k\in\calI}\Omega_{k,\calI,u,p}$, the first bound on Lemma \ref{lemma:increment:bounds} for $k\in\calI$ and Lemma A.4 in \cite{2015dirksen} (using that $k>\ell$ over $\calI$) imply that there is universal constant $c>0$
$$
\prob(\Omega_{\calI,u,p}^c)\le ce^{-pu/4}.
$$
Similarly, the second bound in Lemma \ref{lemma:increment:bounds} for $k\in\calJ$ and Lemma A.4 in \cite{2015dirksen} (using that $k>\ell$ over $\calJ$) imply that for the event  
$\Omega_{\calJ,u,p}:=\cap_{k\in\calJ}\Omega_{k,\calJ,u,p}$, we have 
$$
\prob(\Omega_{\calJ,u,p}^c)\le ce^{-pu/4}.
$$
We now also define the event $\Omega_{u,p}$ as the intersection of $\Omega_{\calI,u,p}\cap\Omega_{\calJ,u,p}$ and the events for which both inequalities of Proposition \ref{prop:quadratic:process} hold for both classes $F$ and $G$. Clearly, by such proposition and the two previous displays we have 
$\prob(\Omega_{u,p}^c)\le ce^{-pu/4}$ from an union bound.

We next fix $u\ge2$ and assume that $\Omega_{u,p}$ always holds. We now bound the chain over $\calI$ and $\calJ$ separately. 

\emph{Subgaussian path $\calI$.}
From \eqref{thm:product:process:eq1} and the inequality
\begin{align}
\Vert\mcP_{k}(f,g)-\mcP_{k-1}(f,g)\Vert_{\psi_1}
&\le\Vert\pi_{k}(f)-\pi_{k-1}(f)\Vert_{\psi_2}\Vert\Pi_{k}(g)\Vert_{\psi_2} 
+\Vert\pi_{k-1}(f)\Vert_{\psi_2} \Vert\Pi_{k}(g)-\Pi_{k-1}(g)\Vert_{\psi_2}\\ 
&\le \bar\Delta(G)[\dist(f,\pi_{k}(f))+\dist(f,\pi_{k-1}(f))]
+\bar\Delta(F)[\dist(g,\Pi_{k}(g))+\dist(g,\Pi_{k-1}(g))]
\end{align}
implying
\begin{align}
\left|\sum_{k\in\calI}[\calP_{k}(f,g)-\calP_{k-1}(f,g)]\right|&\le 
2(1+\sqrt{2})^2\frac{u}{\sqrt{n}}
\left[
\bar\Delta(F)\gamma_{2,p}(G)+\bar\Delta(G)\gamma_{2,p}(F)\right].	\label{thm:product:process:subgaussian}
\end{align}

\emph{Subexponential path $\calJ$.}  
Note that $A(f,g)=\hat\probn fg-\probn(\hat\probn fg)$ and thus, by Jensen's and triangle inequalities,
\begin{align}
\left|\sum_{k\in\calJ}[\calP_{k+1}(f,g)-\calP_{k}(f,g)]\right|&\le
\left|\sum_{k\in\calJ}\hat\probn[\mcP_{k+1}(f,g)-\mcP_{k}(f,g)]\right|
+\left|\sum_{k\in\calJ}\probn\hat\probn[\mcP_{k+1}(f,g)-\mcP_{k}(f,g)]\right|\\
&\le\sum_{k\in\calJ}\hat\probn\left|\mcP_{k+1}(f,g)-\mcP_{k}(f,g)\right|
+\sum_{k\in\calJ}\probn\hat\probn\left|\mcP_{k+1}(f,g)-\mcP_{k}(f,g)\right|.\label{thm:product:process:eq3}
\end{align}
Let us denote $\hat T_k:=\hat\probn\left|\mcP_{k+1}(f,g)-\mcP_{k}(f,g)\right|$. We have the split
\begin{align}
\hat T_k&\le\left|\hat\probn\pi_{k}(f)[\Pi_{k+1}(g)-\Pi_{k}(g)]\right|
+\left|\hat\probn\Pi_{k+1}(g)[\pi_{k+1}(f)-\pi_{k}(f)]\right|.
\end{align}
By Cauchy-Schwarz,
\begin{align}
\left|\hat\probn\pi_{k}(f)[\Pi_{k+1}(g)-\Pi_{k}(g)]\right|
&\le\Vert\pi_{k}(f)\Vert_n\Vert\Pi_{k+1}(g)-\Pi_{k}(g)]\Vert_n\\
&\le\left[\{A(\pi_k(f))\}^{1/2}+\{\probn\pi_k^2(f)\}^{1/2}\right]\Vert\Pi_{k+1}(g)-\Pi_{k}(g)]\Vert_n,
\end{align}
which together with \eqref{thm:product:process:eq2}, bounds in Proposition \ref{prop:quadratic:process} and $\sqrt{u}\ge1$ give
\begin{align}
\left|\hat\probn\pi_{k}(f)[\Pi_{k+1}(g)-\Pi_{k}(g)]\right|&\le \frac{c_2u2^{k/2}}{\sqrt{n}}
\left[c_1\bar\Delta(F)+25\frac{\gamma_{2,p}(F)}{\sqrt{n}}
+\left\{\frac{85\bar\Delta(F)\gamma_{2,p}(F)}{\sqrt{n}}\right\}^{1/2}\right]\dist(\Pi_{k+1}(g),\Pi_{k}(g))\\
&\le c_2\frac{u2^{k/2}}{\sqrt{n}}\dist(\Pi_{k+1}(g),\Pi_{k}(g))
\left[c_3\bar\Delta(F)+c_4\frac{\gamma_{2,p}(F)}{\sqrt{n}}\right],
\end{align}
by Young's inequality and constants $c_1:=\{4(1+\sqrt{2}+2)\}^{1/2}+1$, $c_2:=\{2(1+\sqrt{2}+1)\}^{1/2}$, $c_3:=c_1+\frac{\sqrt{85}}{2}$ and $c_4:=25+\frac{\sqrt{85}}{2}$.
An identical bound gives 
\begin{align}
\left|\hat\probn\Pi_{k+1}(g)[\pi_{k+1}(f)-\pi_{k}(f)]\right|
&\le c_2\frac{u2^{k/2}}{\sqrt{n}}\dist(\pi_{k+1}(f),\pi_{k}(f))
\left[c_3\bar\Delta(G)+c_4\frac{\gamma_{2,p}(G)}{\sqrt{n}}\right].
\end{align}
We thus conclude that 
\begin{align}
\hat T_k&\le  c_2\frac{u2^{k/2}}{\sqrt{n}}\dist(\Pi_{k+1}(g),\Pi_{k}(g))
\left[c_3\bar\Delta(F)+c_4\frac{\gamma_{2,p}(F)}{\sqrt{n}}\right]\\
&+c_2\frac{u2^{k/2}}{\sqrt{n}}\dist(\pi_{k+1}(f),\pi_{k}(f))
\left[c_3\bar\Delta(G)+c_4\frac{\gamma_{2,p}(G)}{\sqrt{n}}\right].
\end{align}
Note that, in fact, we have proved that the above bound on $\hat T_k$ holds with probability at least $1-c\exp(-pu/4)$ for any $u\ge2$. Thus, from Lemma \ref{lemma:dirksen:A5} we have, for some universal constant $c_0>0$,  
\begin{align}
\probn\hat T_k&\le  c_0c_2\frac{u2^{k/2}}{\sqrt{n}}\dist(\Pi_{k+1}(g),\Pi_{k}(g))
\left[c_3\bar\Delta(F)+c_4\frac{\gamma_{2,p}(F)}{\sqrt{n}}\right]\\
&+c_0c_2\frac{u2^{k/2}}{\sqrt{n}}\dist(\pi_{k+1}(f),\pi_{k}(f))
\left[c_3\bar\Delta(G)+c_4\frac{\gamma_{2,p}(G)}{\sqrt{n}}\right].
\end{align}
Using the previous two bounds in \eqref{thm:product:process:eq3} gives, after using the triangle inequality for $\dist$, summing over $k\in\calJ$ and using the definition of $\gamma_{2,p}(F)$ and $\gamma_{2,p}(G)$ (recalling that $k>\ell$), 
\begin{align}
\left|\sum_{k\in\calJ}[\calP_{k+1}(f,g)-\calP_{k}(f,g)]\right|&\le (1+c_0)(1+2^{-1/2})c_2\frac{u}{\sqrt{n}}\gamma_{2,p}(G)
\left[c_3\bar\Delta(F)+c_4\frac{\gamma_{2,p}(F)}{\sqrt{n}}\right]\\
&+(1+c_0)(1+2^{-1/2})c_2\frac{u}{\sqrt{n}}\gamma_{2,p}(F)
\left[c_3\bar\Delta(G)+c_4\frac{\gamma_{2,p}(G)}{\sqrt{n}}\right].
\end{align}

From the above bound, \eqref{thm:product:process:subgaussian} and \eqref{thm:product:process:eq:chain} we conclude that, for any $u\ge2$, on the event $\Omega_{u,p}$ of probability at least $1-e^{-pu/4}$, we have 
\begin{align}
&\sup_{(f,g)\in F\times G}|A(f,g)|^{1/2}-\sup_{(f,g)\in F\times G}|A(\pi_\ell(f),\Pi_\ell(g))|^{1/2}\\
&\le \sqrt{u}\left[\frac{c_5}{\sqrt{n}}\left(\bar\Delta(F)\gamma_{2,p}(G)+\bar\Delta(G)\gamma_{2,p}(F)\right)+c_6\frac{\gamma_{2,p}(F)\gamma_{2,p}(G)}{n}\right]^{1/2},
\end{align}
with $c_5:=2(1+\sqrt{2}^2+(1+c_0)(1+2^{-1/2}))c_2c_3$ and
$c_6:=2(1+c_0)c_2c_4(1+2^{-1/2})$. This and Lemma \ref{lemma:dirksen:A5} (with $\alpha=2$) imply that 
\begin{align}
&\Lpnorm{\sup_{(f,g)\in F\times G}|A(f,g)|^{1/2}-\sup_{(f,g)\in F\times G}|A(\pi_\ell(f),\Pi_\ell(g))|^{1/2}}\\
&\le c\left[\frac{1}{\sqrt{n}}\left(\bar\Delta(F)\gamma_{2,p}(G)+\bar\Delta(G)\gamma_{2,p}(F)\right)+\frac{\gamma_{2,p}(F)\gamma_{2,p}(G)}{n}\right]^{1/2}.
\end{align}
We also have from Lemma \ref{lemma:dirksen:A3},
\begin{align}
\left(\esp\sup_{(f,g)\in F\times G}|A(\pi_\ell(f),\Pi_\ell(g))|^{p/2}\right)^{2/p}&\le4\sup_{(f,g)\in F\times G} \left(\esp|A(\pi_\ell(f),\Pi_\ell(g))|^{p/2}\right)^{2/p}\\
&\le c\sup_{(f,g)\in F\times G}\left[\Vert \mcP_\ell(f,g)-\probn\mcP_\ell(f,g)\Vert_{\psi_1}\left(\sqrt{\frac{p}{n}}+\frac{p}{n}\right)\right],
\end{align}
where the second inequality follows from Bernstein's inequality for $A(\pi_\ell(f),\Pi_\ell(g))$ and Lemma A.2 in \cite{2015dirksen}. The two previous displays finish the proof. 
\end{proof}

\subsection{Oracle inequalities with $\delta$-uniform subgaussian optimal rates}\label{ss:oracle:inequalities}

\emph{Oracle inequalities} provide a statistical guarantee on the prediction risk for \emph{miss-specified} models \cite{BRT}. In the high-dimensional setting, a $n$-sized iid measurement of $y = f + \xi$ is represented in vector form by
\begin{align}\label{eq:model:miss:specified}
\by = \boldf + \bxi,
\end{align}
where one does \emph{not} have the condition $\boldf=\mbX\bb^*$ with $\bb^*$ sparse. Similarly to estimation rates in the well-specified case, it is of practical relevance to obtain optimal oracle inequalities valid for estimators tuned \emph{independently} of the failure probability 
$\delta$. Oracle inequalities with the optimal \emph{subgaussian} rate $\sqrt{s\log(ep/s)/n}+\sqrt{\log(1/\delta)/n}$ are of preference. See item (c) of Section \ref{s:intro} for related discussion.

It was shown for the first time in Bellec et al. \cite{2018bellec:lecue:tsybakov} that the Lasso or Slope estimators can be tuned independently of $\delta$; see Theorems 4.2 and 6.1 and Corollaries 4.3 and 6.2 in that paper. This is valid with  subgaussian data with noise independent of features. In Proposition 3.2, it is also shown in the Gaussian setting that Lasso indeed satisfies an oracle inequality with the subgaussian rate 
$\sqrt{s\log p/n}+\sqrt{\log(1/\delta)/n}$, improving on previous non-subgaussian rates of the form $\sqrt{s\log (p/\delta)/n}$. For better appreciation, we recall their Theorem 6.1 for the Slope estimator
\begin{align}\label{eq:slope:estimator}
\hat\bb\in 
&\argmin_{\bb\in\re^p}\frac{1}{2n}\sum_{i=1}^n\left(y_i-\langle\bx_i,\bb\rangle\right)^2+\lambda\Vert\bb\Vert_\sharp,
\end{align}
where $\Vert\cdot\Vert_\sharp$ denotes the Slope norm with the sequence $\bar\bomega:=(\bar\omega_j)_{j\in[p]}$ given by
$
\bar\omega_j:=\log(2 p/j)
$.
We restate their result up to absolute constants using our notation and assuming the standard normalization assumption of the design matrix columns. Using the notation in \cite{2018bellec:lecue:tsybakov}, we define
$
\bar\Omega_s:=\sqrt{\sum_{j=1}^s\bar\omega_j^2},
$
and 
$$
\calC_{\WRE}(s,c_0):=\left\{\bv\in\re^p
:\Vert\bv\Vert_\sharp\le (1+c_0)\bar\Omega_s\Vert\bv\Vert_2\right\}.
$$
Given a convex cone $\mbC$ on $\re^p$, we define
$$
\hat\mu_n(\mbC):=\sup_{\bv\in\mbC}\frac{\Vert\bv\Vert_2}{\Vert\mbX^{(n)}\bv\Vert_2}.
$$

\begin{theorem}[Theorem 6.1 in \cite{2018bellec:lecue:tsybakov}]
Take tuning 
$
\lambda\asymp \sigma/\sqrt{n}.
$
Let $\delta_0\in(0,1)$. 

Then, with probability at least $1-\delta_0$, for all $\boldf\in\re^n$, all $s\in[p]$ and all $\bb\in\re^p$ such that $\Vert\bb\Vert_0=s$, 
\begin{align}
\lambda\Vert\hat\bb-\bb\Vert_\sharp
+\frac{1}{n}\Vert\mbX\hat\bb-\boldf\Vert_2^2
\le \frac{1}{n}\Vert\mbX\bb-\boldf\Vert_2^2
+C(s,\delta_0)\frac{\sum_{j=1}^s\bar\omega_j^2}{n}.
\end{align}
In above, 
\begin{align}
C(s,\delta_0)\asymp  \hat\mu^2(s)\bigvee\frac{\log(1/\delta_0)}{s\log(2p/s)},\quad\mbox{and}\quad \hat\mu(s):=\hat\mu_n(\calC_{\WRE}(s,7)).
\end{align}
\end{theorem}

The tuning of the above oracle inequality is \emph{independent} of $\delta$ with the optimal rate 
$
\sum_{j=1}^s\bar\omega_j^2/n\asymp 
s\log(2ep/s)/n:=r_n^2.
$
This rate is the optimal \emph{subgaussian} rate, but the fundamentally assume the noise is independent of features. Considering the previous observations, it still an open question if either $\delta$-uniform oracle inequalities with the optimal subgaussian rate can be achieved for Lasso or Slope when the noise is dependent on the features. We next present such type of results for Lasso, Slope and nuclear norm regularization. The main tool we use is the new concentration inequality for the \emph{Multiplier Process} proven in Section \ref{ss:multiplier:process}. The following oracle inequalities are valid even when $\xi$ might depend on the feature vector $\bx$, as long as 
$\esp[\xi\bx]=0$. This condition is satisfied, for example, in the least-squares statistical learning framework where 
\begin{align}
\bb^*\in\argmin_{\bb}\esp[(y-\langle\bx,\bb\rangle)^2],
\end{align}
i.e. $y=f +\xi$ with $f=\langle\bx,\bb^*\rangle$ and $\esp[\xi\bx]=0$. In this particular case, miss-specification occurs when $\bb^*$ is not exactly sparse. 

\begin{theorem}[Slope]\label{thm:oracle:ineq:slope}
Grant Assumption \ref{assump:distribution:subgaussian} and assume that $\esp[\xi\bx]=0$. Let $\rho_1(\bfSigma)$ as in Theorem \ref{thm:response:sparse:regression}. Recall definition of the cone $\overline\calC_s(6)$ in Remark \ref{rem:cones:slope}. Take tuning $\lambda\asymp\sigma L\rho_1(\bfSigma)/\sqrt{n}$. 

Then there are universal constant $c_1>0$, $C_1,C_2>0$ such that the following holds. Given any $\delta\in(0,1)$ such that
\begin{align}
\delta\ge \exp\left(-c_1\frac{n}{L^4}\right),
\end{align}
with probability at least $1-\delta$, 
for all $\boldf\in\re^n$, all $s\in[p]$ and all $\bb\in\re^p$ such that $\Vert\bb\Vert_0\le s$,
\begin{align}
\lambda\Vert\hat\bb-\bb\Vert_\sharp
+\frac{1}{n}\Vert\mbX(\hat\bb)-\boldf\Vert_2^2\le
\frac{1}{n}\Vert\mbX(\bb)-\boldf\Vert_2^2
+ C(s)\frac{\sum_{j=1}^s\bar\omega_j^2}{n}
+ C_2\sigma^2 L^2\frac{1+\log(1/\delta)}{n}.
\end{align}
In above, 
\begin{align}
C(s):=C_1\sigma^2 L^2\rho_1^2(\bfSigma)\hat\mu^2(s),\quad
\mbox{and}\quad \hat\mu(s):=\hat\mu_n(\overline\calC_s(6)).
\end{align} 
\end{theorem}
We note that $C(s)$ is independent of $\delta$. For simplicity, consider the normalized with case $\sigma=L=1$ and $\rho_1(\bfSigma)\le1$. Then the rate in the oracle inequality above is the subgaussian rate
$$
C(s)\frac{\sum_{j=1}^s\bar\omega_j^2}{n}
+ C_2\sigma^2 L^2\frac{1+\log(1/\delta)}{n}
\asymp \hat\mu^2(s)\frac{s\log(2ep/s)}{n} + \frac{1+\log(1/\delta)}{n}.
$$

An analogous oracle inequality holds for Lasso.
\begin{theorem}[Lasso]\label{thm:oracle:ineq:lasso}
Grant Assumption \ref{assump:distribution:subgaussian} and assume that $\esp[\xi\bx]=0$. Let $\rho_1(\bfSigma)$ as in Theorem \ref{thm:response:sparse:regression}. Recall the cone $\calC_{\bb,\Vert\cdot\Vert_1}(6)$ for any $\bb\in\re^p$ in Definition \ref{def:dim:red:cones}. Take tuning 
$\lambda\asymp\sigma L\rho_1(\bfSigma)\sqrt{\log p/n}$. 

Then there are universal constant $c_1>0$, $C_1,C_2>0$ such that the following holds. Given any $\delta\in(0,1)$ such that
\begin{align}
\delta\ge \exp\left(-c_1\frac{n}{L^4}\right),
\end{align}
with probability at least $1-\delta$, 
for all $\boldf\in\re^n$, all $s\in[p]$ and all 
$\bb\in\re^p$ such that $\Vert\bb\Vert_0\le s$,
\begin{align}
\lambda\Vert\hat\bb-\bb\Vert_1
+\frac{1}{n}\Vert\mbX(\hat\bb)-\boldf\Vert_2^2\le
\frac{1}{n}\Vert\mbX(\bb)-\boldf\Vert_2^2
+ C(\bb)\frac{s\log p}{n}
+ C_2\sigma^2 L^2\frac{1+\log(1/\delta)}{n}.
\end{align}
In above, 
\begin{align}
C(\bb):=C_1\sigma^2 L^2\rho_1^2(\bfSigma)\hat\mu^2(\bb),\quad
\mbox{and}\quad \hat\mu(\bb):=\hat\mu_n(\calC_{\bb,\Vert\cdot\Vert_1}(6)).
\end{align} 
\end{theorem}

Consider now the estimator
\begin{align}
\hat\bfB\in 
&\argmin_{\bfB\in\re^p}\frac{1}{2n}\sum_{i=1}^n\left(y_i-\llangle\bfX_i,\bfB\rrangle\right)^2+\lambda\Vert\bfB\Vert_N.
\end{align}
Recall that $\frX$ is the design operator with coordinates $\frX_i(\bfB):=\llangle\bfX_i,\bfB\rrangle$. We can state an $\delta$-uniform oracle inequality for $\hat\bfB$ with optimal subgaussian rate for low-rank miss-specified models. 

\begin{theorem}[Nuclear norm]\label{thm:oracle:ineq:nuclear:norm}
Grant Assumption \ref{assump:distribution:subgaussian} and assume that $\esp[\xi\bfX]=0$. Let $\rho_N(\bfSigma)$ as in Theorem \ref{thm:response:trace:regression}. Recall the cone $\calC_{\bfB,\Vert\cdot\Vert_N}(6)$ for any $\bfB\in\mdR^p$ in Definition \ref{def:dim:red:cones}. Take tuning $\lambda\asymp\sigma L\rho_N(\bfSigma)(\sqrt{d_1/n}+\sqrt{d_2/n})$. 

Then there are universal constant $c_1>0$, $C_1,C_2>0$ such that the following holds. Given any $\delta\in(0,1)$ such that
\begin{align}
\delta\ge \exp\left(-c_1\frac{n}{L^4}\right),
\end{align}
with probability at least $1-\delta$, 
for all $\boldf\in\re^n$, all $r\in\{1,\ldots,d_1\wedge d_2\}$ and all 
$\bfB\in\mdR^p$ with $\rank(\bfB)\le r$,
\begin{align}
\lambda\Vert\hat\bfB-\bfB\Vert_N
+\frac{1}{n}\Vert\frX(\hat\bfB)-\boldf\Vert_F^2\le
\frac{1}{n}\Vert\frX(\bfB)-\boldf\Vert_F^2
+ C(\bfB)\frac{rd_1+rd_2}{n}
+ C_2\sigma^2 L^2\frac{1+\log(1/\delta)}{n}.
\end{align}
In above, 
\begin{align}
C(\bfB):=C_1\sigma^2 L^2\rho_N^2(\bfSigma)\hat\mu^2(\bfB),\quad
\mbox{and}\quad \hat\mu(\bfB):=\hat\mu_n(\calC_{\bfB,\Vert\cdot\Vert_N}(6)).
\end{align} 
\end{theorem}

We only prove Theorem \ref{thm:oracle:ineq:slope} as the proof of Theorems \ref{thm:oracle:ineq:lasso} and \ref{thm:oracle:ineq:nuclear:norm} follow similar argument. Next, $\MP_{\calR}(\sf_1;\sf_2)$ has the analogous definition in Definition \ref{def:design:property} when 
$\bu\equiv0$.

\begin{lemma}[Dimension reduction]\label{lemma:dim:reduction:or:ineq}
Suppose 
\begin{itemize}
\item[\rm (i)] $(\mbX,\bxi)$ satisfies the $\MP_{\Vert\cdot\Vert_\sharp}(\sf_1;\sf_2)$ 
for some positive numbers $\sf_1$ and $\sf_2$.
\item[\rm{(ii)}] $\mbX$ satisfies the $\TP_{\Vert\cdot\Vert_\sharp}(\sc_1;\sc_2)$ for some positive numbers 
$\sc_1,\sc_2$. 
\item[\rm (iii)] 
$
\lambda\ge2[\sf_2+(\nicefrac{\sf_1\sc_2}{\sc_1})].
$
\end{itemize}
Then, for all $\boldf\in\re^n$, all $s\in[p]$ and all $\bb\in\re^p$ such that $\Vert\bb\Vert_0\le s$, either 
\begin{align}
\Vert\mbX^{(n)}(\hat\bb)-\boldf^{(n)}\Vert_2^2\le
\Vert\mbX^{(n)}(\bb)-\boldf^{(n)}\Vert_2^2,
\end{align}
or $\hat\bb-\bb\in\overline\calC_{s}(6)$ or
\begin{align}
\Vert\mbX^{(n)}(\bfDelta)\Vert_2\le2\frac{\sf_1}{\sc_1},\quad\mbox{and}\quad
\lambda\sum_{j=1}^s\bar\omega_j(\bfDelta)_i^\sharp\le \frac{4\sf_1^2}{3\sc_1^2}.\label{lemma:dim:reduction:rate:l1:or:ineq}
\end{align}
\end{lemma}
\begin{proof}
Using $\by^{(n)}=\boldf^{(n)}+\bxi^{(n)}$, the first order condition of \eqref{eq:slope:estimator} at 
$\hat\bb$ implies: there exist 
$\bv\in\partial\Vert\hat\bb\Vert_\sharp$ such that for all $\bb$ and $\bfDelta:=\hat\bb-\bb$,
\begin{align}
\langle\mbX^{(n)}(\hat\bb)-\boldf^{(n)},\mbX^{(n)}(\bfDelta)\rangle\le
\langle\bxi^{(n)},\mbX^{(n)}(\bfDelta)\rangle
-\lambda\langle\bv,\bfDelta\rangle.
\end{align}
Suppose that $\Vert\bb\Vert_0\le s$. Next we use that
$-\langle\bfDelta,\bv\rangle \le \|\bb\|_\sharp-\|\hat\bb\|_\sharp$. This fact, items (i)-(iii) and proceeding as in the chain of inequalities  in \eqref{lemma:dim:reduction:eq1}, entail
\begin{align}
\langle\boldf^{(n)}-\mbX^{(n)}(\hat\bb),\mbX^{(n)}(\bfDelta)\rangle
&\le\frac{\sf_1}{\sc_1}\Vert\mbX^{(n)}(\bfDelta)\Vert_2
+(\nicefrac{\lambda}{2})\|\bfDelta\|_\sharp
 +  \lambda \big(\|\bb\|_\sharp -\|\hat\bb\|_\sharp\big)\\
&\le\frac{\sf_1}{\sc_1}\Vert\mbX^{(n)}(\bfDelta)\Vert_2+\triangle,
\label{lemma:dim:reduction:eq1:or:ineq}
\end{align}
where, in last inequality, we used Lemma \ref{lemma:A1} (since $\Vert\bb\Vert_0\le s$) with $\nu:=1/2$ and defined
\begin{align}
\triangle:=(\nicefrac{3\lambda}{2})\sum_{j=1}^s\bar\omega_j(\bfDelta)_j^\sharp -(\nicefrac{\lambda}{2})\sum_{j=s+1}^p\bar\omega_j(\bfDelta)_j^\sharp
=2\lambda\sum_{j=1}^s\bar\omega_j(\bfDelta)_j^\sharp -(\nicefrac{\lambda}{2})\Vert\bfDelta\Vert_\sharp.
\end{align}

Moreover, from the parallelogram law, 
\begin{align}
I:=\langle\mbX^{(n)}(\hat\bb)-\boldf^{(n)},\mbX^{(n)}(\bfDelta)\rangle
=\frac{1}{2}\Vert\mbX^{(n)}(\hat\bb)-\boldf^{(n)}\Vert_2^2
+\frac{1}{2}\Vert\mbX^{(n)}(\bfDelta)\Vert_2^2
- \frac{1}{2}\Vert\mbX^{(n)}(\bb)-\boldf^{(n)}\Vert_2^2
\label{eq:parallel:law}
\end{align}
If $I\le0$ then the statement of the lemma holds trivially. Next, we assume that $I\ge0$.

Define $G:=\Vert\mbX^{(n)}(\bfDelta)\Vert_2$ and $H:=(\nicefrac{3\lambda}{2})\sum_{j=1}^s\bar\omega_j(\bfDelta)_i^\sharp$. We split the argument in two cases. 
\begin{description}
\item[Case 1:] $\frac{\sf_1}{\sc_1}G\le H$. In that case, from \eqref{lemma:dim:reduction:eq1:or:ineq} and $I\ge0$ we obtain 
$\bfDelta\in\overline\calC_{s}(6)$.

\item[Case 2:] $\frac{\sf_1}{\sc_1}G\ge H$. In that case we obtain $G^2\le\frac{2\sf_1}{\sc_1}G\Rightarrow G\le\frac{2\sf_1}{\sc_1}$. Therefore $H\le2\frac{\sf_1^2}{\sc_1^2}$. 
\end{description}
This finishes the proof. 
\end{proof}

\begin{theorem}\label{thm:slope:or:ineq:det}
Suppose that items (i)-(iii) in Lemma \ref{lemma:dim:reduction:or:ineq} hold. 

Then, for all $\boldf\in\re^n$, all $s\in[p]$ and all $\bb\in\re^p$ such that $\Vert\bb\Vert_0\le s$, 
\begin{align}
\lambda\Vert\hat\bb-\bb\Vert_\sharp
+\Vert\mbX^{(n)}(\hat\bb)-\boldf^{(n)}\Vert_2^2\le
\Vert\mbX^{(n)}(\bb)-\boldf^{(n)}\Vert_2^2
+\left[\frac{\sf_1}{\sc_1}+2\lambda\bar\Omega_s\hat\mu_n(\overline\calC_s(6))\right]^2\bigvee\frac{28\sf_1^2}{3\sc_1^2}.
\end{align}
\end{theorem}
\begin{proof}
Let $\bfDelta:=\hat\bb-\bb$ and define $\calL(\bv):=\frac{1}{2}\Vert\mbX^{(n)}(\bv)-\boldf^{(n)}\Vert_2^2$. From \eqref{lemma:dim:reduction:eq1:or:ineq},  \eqref{eq:parallel:law} and definition of $\triangle$,
\begin{align}
(\nicefrac{\lambda}{2})\Vert\bfDelta\Vert_\sharp+\calL(\hat\bb)\le\calL(\bb)
-\frac{1}{2}\Vert\mbX^{(n)}(\bfDelta)\Vert_2^2
+\frac{\sf_1}{\sc_1}\Vert\mbX^{(n)}(\bfDelta)\Vert_2
+2\lambda\sum_{j=1}^s\bar\omega_j(\bfDelta)_j^\sharp.
\label{thm:slope:or:ineq:det:eq1}
\end{align}

Recall the 3 cases in Lemma \ref{lemma:dim:reduction:or:ineq}. The 1st is 
$\calL(\hat\bb)\le\calL(\bb)$ which implies the theorem trivially. 

The 2nd case is when $\bfDelta\in\overline\calC_s(6)$. For simplicity, let $\hat\mu_s:=\hat\mu_n(\overline\calC_s(6))$. By decomposability and Cauchy-Schwarz,
\begin{align}
2\lambda\sum_{j=1}^s\bar\omega_j(\bfDelta)_j^\sharp\le2\lambda\bar\Omega_s\hat\mu_s\|\mbX^{(n)}(\bfDelta)\|_2.
\label{thm:slope:or:ineq:det:eq2}
\end{align}
From \eqref{thm:slope:or:ineq:det:eq1}, \eqref{thm:slope:or:ineq:det:eq2}
 and Young's inequality 
$$
\left[\frac{\sf_1}{\sc_1}+2\lambda\bar\Omega\hat\mu_s\right]\Vert\mbX^{(n)}(\bfDelta)\Vert_2\le
\frac{1}{2}\left[\frac{\sf_1}{\sc_1}+2\lambda\bar\Omega_s\hat\mu_s\right]^2+\frac{1}{2}\Vert\mbX^{(n)}(\bfDelta)\Vert_2^2,
$$
we get 
\begin{align}
(\nicefrac{\lambda}{2})\Vert\bfDelta\Vert_\sharp+\calL(\hat\bb)\le\calL(\bb)
+\frac{1}{2}\left[\frac{\sf_1}{\sc_1}+2\lambda\bar\Omega_s\hat\mu_s\right]^2.
\end{align}

Finally, the 3rd case is when the bounds \eqref{lemma:dim:reduction:rate:l1:or:ineq} hold. In that case, \eqref{thm:slope:or:ineq:det:eq1} implies
\begin{align}
(\nicefrac{\lambda}{2})\Vert\bfDelta\Vert_\sharp+\calL(\hat\bb)\le\calL(\bb)
+\frac{2\sf_1^2}{\sc_1^2}+\frac{8\sf_1^2}{3\sc_1^2}.
\end{align}

The proof is complete by noting that the bound in the statement of the theorem is larger than the two previous displays.
\end{proof}

The Proof of Theorem \ref{thm:oracle:ineq:slope} now follows from the previous theorem and the probabilistic guarantees in Section \ref{ss:subgaussian:designs} for $L$-subgaussian designs.  
\begin{proof}[Proof of Theorem \ref{thm:oracle:ineq:slope}]
Proposition E.2 in  \cite{2018bellec:lecue:tsybakov} implies 
$\mathscr{G}(\bfSigma^{1/2}\mbB_\sharp)\lesssim\rho_1(\bfSigma)$. 

From Proposition \ref{prop:gen:TP}, assuming
\begin{align}\label{eq:delta:or:ineq}
\delta\ge \exp\left(-c_1\frac{n}{L^4}\right),
\end{align}
for large enough universal $c_1>0$, it follows that for universal $\sc_1\in(0,1)$ and
$
\sc_2\asymp L\frac{\rho_1(\bfSigma)}{\sqrt{n}},
$
on an event $\Omega_1$ of probability at least $1-\delta/2$, property 
$\TP_{\Vert\cdot\Vert_\sharp}(\sc_1;\sc_2)$ is satisfied. 

From Proposition \ref{prop:gen:MP} (with $\delta=\rho$), by enlarging $c_1$ if necessary, if we take
$$
\sf_1\asymp\sigma L\frac{1+\sqrt{\log(1/\delta)}}{\sqrt{n}},\quad\mbox{and}\quad
\sf_2\asymp\sigma L\frac{\rho_1(\bfSigma)}{\sqrt{n}},
$$
we have that on an event 
$\Omega_2$ of probability at least $1-\delta/2$,  
$\MP_{\Vert\cdot\Vert_\sharp}(\sf_1;\sf_2)$ is satisfied.

Together with definition of $\lambda$, an union bound implies that all conditions (i)-(iii) in Theorem \ref{thm:slope:or:ineq:det} are met and Theorem \ref{thm:oracle:ineq:slope} follows. 
\end{proof}

\bibliographystyle{plain}

\end{document}